\documentclass[10pt,a4paper,reqno]{amsart}
\allowdisplaybreaks
\numberwithin{equation}{section}
\usepackage[normalem]{ulem}
\makeatletter 
\def\@tocline#1#2#3#4#5#6#7{\relax
  \ifnum #1>\c@tocdepth 
  \else
    \par \addpenalty\@secpenalty\addvspace{#2}%
    \begingroup \hyphenpenalty\@M
    \@ifempty{#4}{%
      \@tempdima\csname r@tocindent\number#1\endcsname\relax
    }{%
      \@tempdima#4\relax
    }%
    \parindent\z@ \leftskip#3\relax \advance\leftskip\@tempdima\relax
    \rightskip\@pnumwidth plus4em \parfillskip-\@pnumwidth
    #5\leavevmode\hskip-\@tempdima
      \ifcase #1
       \or\or \hskip 1em \or \hskip 2em \else \hskip 3em \fi%
      #6\nobreak\relax
    \dotfill\hbox to\@pnumwidth{\@tocpagenum{#7}}\par
    \nobreak
    \endgroup
  \fi}
\makeatother

\usepackage{transparent}
\usepackage[margin=1.3in]{geometry}
\usepackage{amsmath,amsthm,amssymb}
\usepackage[initials,msc-links]{amsrefs}
\usepackage{url}
\usepackage[dvipsnames]{xcolor}
\usepackage[english=american]{csquotes}
\usepackage{enumerate}
\usepackage{textcomp}
\usepackage{mathrsfs}
\usepackage{mathtools}
\usepackage{color, colortbl}
\definecolor{Gray}{gray}{0.9}
\usepackage{bbm}
\usepackage{hyperref}
\usepackage{stmaryrd}
\hypersetup{
    unicode=false,          
    pdftoolbar=true,        
    pdfmenubar=true,        
    pdffitwindow=false,     
    pdfstartview={FitH},    
    pdftitle={My title},    
    pdfauthor={Author},     
    pdfsubject={Subject},   
    pdfcreator={Creator},   
    pdfproducer={Producer}, 
    pdfkeywords={keyword1, key2, key3}, 
    pdfnewwindow=true,      
    colorlinks=true,       
    linkcolor=blue ,          
    citecolor=blue ,        
    filecolor=magenta,      
    urlcolor=Aquamarine           
}
\usepackage{xcolor}
\usepackage{epsfig,color,graphicx,graphics}
\usepackage{caption}
\usepackage{amsfonts}
\usepackage{tikz}
\usepackage{siunitx}
\usepackage{booktabs,colortbl, array}
\usepackage{pgfplotstable}
\pgfplotsset{compat=1.8}

\definecolor{rulecolor}{RGB}{0,71,171}
\definecolor{tableheadcolor}{gray}{0.92}

\newtheorem{theorem}{Theorem}[section]
\newtheorem{lemma}[theorem]{Lemma}
\newtheorem{proposition}[theorem]{Proposition}
\newtheorem{corollary}[theorem]{Corollary}

\theoremstyle{definition}
\newtheorem{definition}[theorem]{Definition}
\newtheorem{remark}[theorem]{Remark}

\newtheorem{assumption}[theorem]{Assumption}

\newcommand{\Irm}{\mathrm{I}}

\newcommand{\Trm}{\mathrm{T}}

\newcommand{\Acal}{\mathcal{A}}
\newcommand{\Bcal}{\mathcal B}

\newcommand{\Fcal}{\mathcal{F}}
\newcommand{\Gcal}{\mathcal{G}}
\newcommand{\Hcal}{\mathcal{H}}

\newcommand{\Mcal}{\mathcal{M}}

\newcommand{\Qcal}{\mathcal{Q}}

\newcommand{\Scal}{\mathcal{S}}

\newcommand{\Ffrak}{\mathfrak{F}}

\newcommand{\Ascr}{\mathscr{A}}
\newcommand{\Bscr}{\mathscr{B}}
\newcommand{\Cscr}{\mathscr{C}}

\newcommand{\Fscr}{\mathscr{F}}

\newcommand{\Hscr}{\mathscr{H}}

\newcommand{\Lscr}{\mathscr{L}}

\newcommand{\Nscr}{\mathscr{N}}

\newcommand{\Pscr}{\mathscr{P}}

\newcommand{\Rscr}{\mathscr{R}}

\newcommand{\Abf}{\mathbf{A}}
\newcommand{\Bbf}{\mathbf{B}}
\newcommand{\Cbf}{\mathbf{C}}

\newcommand{\Ebf}{\mathbf{E}}

\newcommand{\Gbf}{\mathbf{G}}
\newcommand{\Hbf}{\mathbf{H}}
\newcommand{\Ibf}{\mathbf{I}}
\newcommand{\Jbf}{\mathbf{J}}

\newcommand{\Sbf}{\mathbf{S}}

\newcommand{\mbf}{\boldsymbol{m}}

\DeclareMathOperator{\graph}{graph}

\DeclareMathOperator{\diverg}{div}

\DeclareMathOperator{\dist}{dist}

\newcommand{\N}{\mathbb{N}}
\newcommand{\R}{\mathbb{R}}

\newcommand{\loc}{\mathrm{loc}}

\newcommand{\spt}{\mathrm{spt}}
\newcommand{\gr}{\mathrm{gr}}

\newcommand{\sing}{\mathrm{sing}}
\newcommand{\Sing}{\mathrm{Sing}}

\newcommand{\toweakstar}{\overset{*}\rightharpoonup}

\newcommand{\BV}{\mathrm{BV}}

\newcommand{\eps}{\epsilon}

\newcommand{\Ebb}{\mathbb{E}}

\renewcommand{\eps}{\varepsilon}
\newcommand{\vphi}{\varphi}

\makeatletter
\renewcommand*\env@matrix[1][*\c@MaxMatrixCols c]{%
    \hskip -\arraycolsep
    \let\@ifnextchar\new@ifnextchar
    \array{#1}}
\makeatother

\DeclareMathOperator{\Lip}{Lip}

\newcommand{\mres}{\mathbin{\vrule height 1.6ex depth 0pt width
        0.13ex\vrule height 0.13ex depth 0pt width 1.3ex}}

\def\vint_#1{\mathchoice%
    {\mathop{\kern 0.2em\vrule width 0.6em height 0.69678ex depth -0.58065ex
            \kern -0.8em \intop}\nolimits_{\kern -0.4em#1}}%
    {\mathop{\kern 0.1em\vrule width 0.5em height 0.69678ex depth -0.60387ex
            \kern -0.6em \intop}\nolimits_{#1}}%
    {\mathop{\kern 0.1em\vrule width 0.5em height 0.69678ex depth -0.60387ex
            \kern -0.6em \intop}\nolimits_{#1}}%
    {\mathop{\kern 0.1em\vrule width 0.5em height 0.69678ex depth -0.60387ex
            \kern -0.6em \intop}\nolimits_{#1}}}
\makeatletter
\newcommand*{\RangeX}{%
    {%
        \mathpalette\@RangeOf{X}%
    }%
}
\newcommand*{\@RangeOf}[2]{%
    \sbox0{$\m@th#1\mathsf{#2}$}%
    \mathsf{#2}%
    \kern-\wd0 %
    \mkern2.75mu\relax
    \nonscript\mkern.25mu\relax
    \mathsf{#2}%
}
\makeatother

\newcommand{\aveint}[2]{\mathchoice%
    {\mathop{\kern 0.2em\vrule width 0.6em height 0.69678ex depth -0.58065ex
            \kern -0.8em \intop}\nolimits_{\kern -0.45em#1}^{#2}}%
    {\mathop{\kern 0.1em\vrule width 0.5em height 0.69678ex depth -0.60387ex
            \kern -0.6em \intop}\nolimits_{#1}^{#2}}%
    {\mathop{\kern 0.1em\vrule width 0.5em height 0.69678ex depth -0.60387ex
            \kern -0.6em \intop}\nolimits_{#1}^{#2}}%
    {\mathop{\kern 0.1em\vrule width 0.5em height 0.69678ex depth -0.60387ex
            \kern -0.6em \intop}\nolimits_{#1}^{#2}}}

\newcommand\res{\mathop{\hbox{\vrule height 7pt width .3pt depth 0pt\vrule height .3pt width 5pt depth 0pt}}\nolimits}

\title[Singularities of area-minimizing currents mod$(q)$]{Fine Structure of Singularities in Area-Minimizing Currents Mod$(q)$}

\author[C. De Lellis]{Camillo De Lellis}
\address{School of Mathematics, Institute for Advanced Study, 1 Einstein Dr., Princeton NJ 08540, USA}
\email{camillo.delellis@ias.edu}

\author[P. Minter]{Paul Minter}
\address{Department of Mathematics, Fine Hall, Princeton University, Washington Road, Princeton, NJ, 08540, USA; School of Mathematics, Institute for Advanced Study, 1 Einstein Dr., Princeton, NJ, 08540, USA.}
\email{pm6978@princeton.edu\textnormal{,} pminter@ias.edu}

\author[A. Skorobogatova]{Anna Skorobogatova}
\address{Department of Mathematics, Fine Hall, Princeton University, Washington Road, Princeton, NJ 08540, USA}
\email{as110@princeton.edu}

\begin{document}

\maketitle

\begin{abstract}
    We study fine structural properties related to the interior regularity of $m$-dimensional area minimizing currents mod$(q)$ in arbitrary codimension. We show: (i) the set of points where at least one tangent cone is translation invariant along $m-1$ directions is locally a connected $C^{1,\beta}$ submanifold, and moreover such points have unique tangent cones; (ii) the remaining part of the singular set is countably $(m-2)$-rectifiable, with a unique flat tangent cone (with multiplicity) at $\mathcal{H}^{m-2}$-a.e. flat singular point. These results are consequences of fine excess decay theorems as well as almost monotonicity of a universal frequency function.
\end{abstract}

\tableofcontents

\section{Introduction and main results}
Let $T$ be an $m$-dimensional area-minimizing current mod$(q)$ in a (sufficiently regular) Riemannian manifold, for some positive integer $q$ (see e.g. \cite{DLHMS} for the relevant definitions and background). The interior regularity of such $T$ has been studied extensively, starting with results by Federer \cite{Federer1970}, Taylor \cite{JTaylor76}, White \cite{White-mod4,White-regularity} for specific small moduli $q$. More recently, a number of structural results for the interior singularities have been established in the works \cite{DLHMS, DLHMSS, DLHMSS-excess-decay, MW, Sk-modp}, for arbitrary moduli $q$. However, finer structural properties, such as the fact that the singular set consists locally of connected $(m-1)$-dimensional $C^{1,\alpha}$ submanifolds, up to a countably $(m-2)$-rectifiable set, have only been demonstrated for hypersurfaces for general moduli $q$. The goal of this article is to resolve this question in arbitrary codimension, establishing analogous finer properties of the singular set, and moreover give a satisfactory description of singular points in the largest possible non-flat strata. 

The overall result on the structure of the singular set is summarized in the following theorem. As usual, if a point $p$ is not in the support of the boundary of $T$, we say that it is regular if there is a neighborhood of it in which $T$ is represented by a constant multiple of a regular minimal submanifold. The complement of the union of this set and of the support of $\partial T$ will be denoted by ${\rm Sing}\, (T)$ and will be referred to as the (interior) singular set of $T$.

\begin{theorem}\label{c:main}
Assume that $m \geq 2$, $n\geq\bar{n}\geq 2$, and $q\geq 3$ are positive integers and that
\begin{itemize}
\item[(i)] $\Sigma\subset \mathbb R^{m+n}$ is a complete embedded $C^{3,\alpha_0}$ $(m+\bar n)$-dimensional submanifold for some positive $\alpha_0>0$;
\item[(ii)] $U\subset \mathbb R^{m+n}$ is a bounded open set;
\item[(iii)] $T$ is a representative mod$(q)$ of an area-minimizing flat chain mod$(q)$ in $\Sigma \cap U$ with $\partial T \res U = 0$  mod$(q)$. 
\end{itemize}
 Then the interior singular set ${\rm Sing}\, (T)$ of $T$ is the disjoint union of 
\begin{itemize}
\item[(a)] a subset $\mathcal{R}\subset {\rm Sing}\, (T)$ which is countably $(m-2)$-rectifiable;
\item[(b)] a subset $\mathcal{S} = {\rm Sing}\, (T) \setminus \mathcal{R}$ with the property that for every $p\in \mathcal{S}$ there is a neighborhood $U'$ such that $U'\cap \mathcal{S}$ is a $C^{1, \beta}$ $(m-1)$-dimensional submanifold of $U'\cap \Sigma$ without boundary, where $\beta = \beta(q,m,n,\bar{n})>0$.
\end{itemize}
Moreover:
\begin{itemize}
    \item[(c)] $T$ is an area-minimizing integral current in the open set $U\setminus \overline{\mathcal{S}}$ and there is a continuous orientation of the submanifold $\mathcal{S}$ such that $\partial T \res U = q \llbracket \mathcal{S} \rrbracket$;
    \item[(d)] at $\Hcal^{m-2}$-a.e. singular point where a tangent cone to $T$ is either a union of planes meeting in a subspace of dimension $m-2$, or a single plane (with multiplicity), the tangent cone is unique; moreover, at every singular point where a tangent cone to $T$ is invariant along an $(m-1)$-dimensional subspace, the tangent cone is unique.
    \end{itemize}
\end{theorem}

The latter theorem builds upon the following result, which is concerned
with those singular points that have at least one tangent cone that is translation invariant in exactly $m-1$ dimensions. It says the following.
\begin{theorem}\label{t:main}
Assume that $m \geq 2$, $n\geq\bar{n}\geq 2$, $q\geq 3$, $T$, $\Sigma$, and $U$ be as in Theorem \ref{t:main}.
Let $p\in \spt^q (T)\cap U$ be a point at which there is a tangent cone with an $(m-1)$-dimensional spine. Then there is a neighborhood $U'$ of $p$ in which the set 
\[
\mathcal{S} := \left\{x\in U' : \Theta (T, x) \geq \frac{q}{2} \right\}
\] 
is a connected $C^{1,\beta}$ $(m-1)$-dimensional submanifold of $U'$ without boundary, where $\beta = \beta (q,m,n, \bar n) > 0$. Moreover:
\begin{itemize}
\item[(a)] at every point $x\in \mathcal{S}$ there is a unique tangent cone to $T$, with $T_x \mathcal{S}$ its spine, and $\Theta(T,x) = \frac{q}{2}$;
\item[(b)] for $\|T\|$-a.e. $x\in U'\setminus \mathcal{S}$ the density $\Theta (T, x)$ is an integer strictly smaller than $\frac{q}{2}$;
\item[(c)] $T\res (U'\setminus \mathcal{S})$ is an area-minimizing integral current in the open set $U'\setminus \mathcal{S}$;
\item[(d)] $\mathcal{S}$ has a continuous orientation such that $\partial T \res U' = q \llbracket \mathcal{S} \rrbracket$.
\end{itemize}
\end{theorem}

The reader might wonder why we are unable to establish uniqueness of tangent cones at $\mathcal{H}^{m-2}$-a.e. singular point, rather than merely those singular points described in the conclusion (d) of Theorem \ref{c:main}. The reason is that if a tangent cone to such a current is translation invariant along exactly $(m-2)$-directions, we do not know whether it must in general be a superposition of planes (to our knowledge no counterexamples are known). If this were the case, our results would give uniqueness of the tangent cone at $\mathcal{H}^{m-2}$ singular point. This is unlike in \cite{DMS}, where we are able to deduce $\Hcal^{m-2}$-a.e. uniqueness of tangent cones for area-minimizing {\em integral} currents, which in that setting is precisely due to such a classification of $(m-2)$-invariant area-minimizing integral cones.

When the modulus $q$ is odd, previous work of \cite{DLHMS, NV_varifolds} establishes that the singular set has Hausdorff dimension at most $m-1$, as well as countable $(m-2)$-rectifiability of all singular points where no tangent cone was translation invariant along $m-1$ directions. Combining this with Theorem \ref{t:main}, this readily implies (a), (b) and (c) from Theorem \ref{c:main} for odd $q$. In this work, we overcome all of the previous difficulties (including for even $q$) to establish the optimal result, in particular answering several conjectures in \cite[Conjecture 1.5, 1.11]{DLHMS}. We note that a key difference between odd and even $q$ is the presence of flat singularities with top density $Q = \frac{q}{2}$ in the case of even $q$, which require substantially more work in demonstrating the desired rectifiability claim (which was only demonstrated for hypersurfaces in \cite{Sk-modp}).

In the setting of Theorem \ref{t:main},  when all of the half-planes in the tangent cone at $p$ have multiplicity one (and so one has a full sheeting theorem, namely Allard's regularity theorem), we have the following additional corollary determining the full local structure near such multiplicity one cones.

\begin{corollary}\label{c:mult-1}
    Assume that $m\geq 2$, $n\geq\bar{n}\geq 2$, and $q\geq 3$ are positive integers. Let $S = \sum^q_{i=1}\llbracket \Hbf_i\rrbracket$ be a cone with an $(m-1)$-dimensional spine (see Proposition \ref{p:classification-cones} below) and $\Hbf_i$ distinct (i.e. all half-planes are multiplicity one). Then, there exists $\eps_0 = \eps_0(S,q,m,n,\bar{n})\in (0,1)$ such that the following holds. Let $T,\Sigma$, and $U$ be as in Theorem \ref{t:main}. Let $p\in \spt^q(T)\cap U$ be a point at which, for $\Sbf = \spt(S)$ and some $\tau>0$,
    $$\mathbb{E}(T,\Sbf,\Bbf_\tau(p))\leq \eps_0^2,$$
    using the notation as in Theorem \ref{t:decay} below. Then, $T\res \Bbf_{\tau/2}(p)$ is a $C^{1,\beta}$ perturbation of $S$, in the sense of \cite{MW}*{Theorem C}. Here, $\beta = \beta(q,m,n,\bar{n})$.
\end{corollary}

In fact, $\eps_0$ in Corollary \ref{c:mult-1} only depends on the minimal angle between the half-planes in $S$. In particular, Corollary \ref{c:mult-1} applies to any area-minimizer mod$(4)$ about a singular point where one tangent cone is supported on a union of half-planes meeting along an $(n-1)$-dimensional spine, giving a complete local structural result near such singular points (and thus uniqueness of the tangent cone). Moreover, in the mod$(4)$ case, Corollary \ref{c:mult-1} can be rephrased as a ``fine $\eps$-regularity'' theorem, like in \cite{MW}*{Theorem 3.1}, with an $\eps_0$ depending only on $q,m,n,\bar{n}$, where instead we assume that $T$ is much closer to $\Sbf$ than any plane, i.e.
\begin{equation}\label{e:mult-1-ratio}
        \Ebb(T,\Sbf,\Bbf_\tau(p)) \leq \eps^2_0 \Ebf^p(T,\Bbf_\tau(p)),
\end{equation}
again using the notation of Theorem \ref{t:decay}.

\begin{remark}\label{remark:main}
    In the statement of Theorem \ref{t:main}, rather than assuming that $T$ has a tangent cone $S$ with an $(m-1)$-dimensional spine at the point $p$, it in fact suffices to simply ask that $T\mres \Bbf_\tau(p)$ is sufficiently close to a cone $S$ with an $(m-1)$-dimensional spine, in the sense that
    \[
        \Ebb(T,\Sbf, \Bbf_\tau(p)) \leq \eps^2_0
    \]
    for $\Sbf=\spt(S)$ and some $\tau, \eps_0 >0$, with $\eps_0$ allowed to depend on $\Sbf$ (see Proposition \ref{p:classification-cones} and Definition \ref{def:L2_height_excess}). In fact, it suffices to ask that
    \[
        \Ebb(T,\Sbf,\Bbf_1(p))<\eps^2_1\Ebf^p(T,\Bbf_1(p)),
    \]
    where $\eps_1$ only depends on $q,m,n,\bar{n}$ (and in particular is independent of $\Sbf$).
\end{remark}

\subsection{Structure of article}
The article is divided up into two parts. In Part 1, we establish Theorem \ref{t:main}. The key ingredient is establishing an excess decay theorem when $T$ is much closer to a cone which is translation invariant along $m-1$ directions than it is to any plane, known as a \emph{fine} $\eps$-regularity theorem, the ideas for which go back to the work of Simon \cite{Simon_cylindrical} (for multiplicity 1, ``non-fine'', settings) and Wickramasekera \cite{W14_annals} (in a higher multiplicity setting and ``fine'' setting, in codimension 1). In particular, such a result allows for a ``degenerate'' setting, when we allow $T$ to be close to a plane also (provided that it remains quantitatively closer to the $(m-1)$-invariant cone, in the sense of $L^2$ height excess). For all of this, we will need to utilize an improved height bound established in our previous work \cite{DMS} on area-minimizing integral currents in arbitrary codimension (see also \cite{KW2}), showing that it can in fact also be applied in the settings needed for this article.

In Part 2, we then establish Theorem \ref{c:main}. This will use both Theorem \ref{t:main} as well as a similar result to Theorem \ref{t:main} but when $T$ is much closer to a \emph{planar} cone with only $m-2$ directions of translation invariance than to any cone with higher dimensions of translational invariance \emph{and} when $T$ does not admit density ``gaps''. Again, the ideas behind this date back to the work of Simon \cite{Simon_cylindrical} and Wickramasekera \cite{W14_annals}. We stress that this case is different to Theorem \ref{t:main} as in this setting we have to assume:
\begin{itemize}
    \item[(i)] that $T$ is also much further from all $(m-1)$-invariant cones than the given $(m-2)$-invariant planar one;
    \item[(ii)] that there are no density gaps.
\end{itemize}
Ultimately, (ii) is what leads us to only being able to show countably $(m-2)$-rectifiability; in Theorem \ref{t:main}, we are actually able to show that density gaps do \emph{not} occur. 

We also establish a quantitative $\BV$-estimate of a \emph{universal frequency function}, which is defined based on Almgren's frequency function for a suitable graphical approximation of $T$ (which varies with the radial scale), in a similar manner to previous work of the first and third authors \cite{DLSk1}. We then use this with methods of Naber--Valtorta \cite{NV_varifolds} to establish countable $(m-2)$-rectifiability of the set of flat singular points in a similar manner to \cite{DLSk2}. This part of the work follows closely our previous works \cite{DLSk1, DLSk2, DMS} on area minimizing currents in arbitrary codimension.

\subsection*{Acknowledgments}
This research was conducted during the period P.M. served as a Clay
Research Fellow.

\subsection{Notation} {\Small
\begin{align*}
 	& r, \rho, s, t && \text{typically denote radii}\\
    & i,j,k && \text{indices}\\
 	& \alpha, \beta, \pi && \text{$m$-dimensional planes}\\
    & \varpi && \text{$(m+\bar{n})$-dimensional plane} \\
        &\varepsilon, \delta, \eta && \text{small numbers, with $\varepsilon$ the smallest in hierarchy}\\
        &\gamma,\kappa,\mu && \text{exponents}\\
        &\varsigma, \sigma, \tau, \varkappa && \text{parameters}\\
        &\phi,\theta,\vartheta && \text{angles}\\
        &\varphi,\psi,\chi && \text{test functions}\\
        & f,g,h,u,v,w && \text{functions, with $f$, $u$, $v$ and $w$ typically denoting multi-valued approximations}\\
        & \Psi, \Sigma && \text{$\Psi$ the parameterization of the ambient manifold $\Sigma$}\\
        & \Sigma_{p,r} &&\text{the rescaled manifold $\iota_{p,r}(\Sigma)$} \\
        & S,T && \text{currents}\\
        & p,x && \text{points in $\R^{m+n}$} \\
        & y,z,\xi,\zeta && \text{variables (typically in $m$-dimensional subspaces of $\R^{m+n}$)} \\
        &\mathbf{p}, \ \mathbf{p}^\perp &&\text{orthogonal projection, projection to orthogonal complement, respectively} \\        
        & \mathbf{1}_E && \text{indicator function of the set $E$}\\
        & \mathbf{A} && \text{the $L^\infty$ norm of the second fundamental form of $\Sigma$}\\
        &\Theta(T,x) &&\text{the $m$-dimensional density of $T$ at a point $x$} \\
        & \Sing(T), \Sing^f (T) && \text{singular sets of $T$, with $\Sing^f (T)$ the flat singularities}\\
        & \Sing^f_Q (T) && \text{flat singularities of $T$ where the density of $T$ is $Q$}\\
        & L, \ell(L) && L \text{ a cube, $\ell(L)$ half the side length}\\
        & A,B && \text{linear maps}\\
        & X && \text{vector field}\\
	& \mathbf{S} && \text{open book, cf. Definition \ref{d:open-book}} \\
    & N && \text{natural number, typically denoting the number of pages in the open book $\Sbf$}\\
        & V && \text{spines of cones} \\
    & \mathscr{P} && \text{set of $m$-dimensional planes}\\
    & \mathscr{B} && \text{set of open books (cf. Definition \ref{d:open-book})}\\
 	& \mathbf{E}^p(T,\Bbf) && \text{planar excess of $T$ in the $(m+n)$-dimensional ball $\Bbf$} \\          
	& \hat{\mathbf{E}}(T,\Sbf,\Bbf), \hat{\mathbf{E}}(\Sbf,T,\Bbf) && \text{one-sided $L^2$ conical excess in $\Bbf$ ($T$ close to $\Sbf$, $\Sbf$ close to $T$, resp.)} \\
	& \mathbb{E}(T,\Sbf,\Bbf) && \text{double-sided conical excess in $\Bbf$} \\
    & \mathbf{E}(T,\Bbf) && \text{oriented tilt excess of an integral current $T$ in $\Bbf$} \\ 
 	& B_a(V) && \text{fixed tubular neighbourhood of radius $a$ of the spine $V$ being removed from $\Bbf_1$} \\  
 	& \boldsymbol{\sigma}(\Sbf) && \text{minimal pairwise Hausdorff distance between the reflected pages of $\Sbf$ in $\Bbf_1$} \\ 
 	& \boldsymbol{\zeta}(\Sbf) && \text{maximal Hausdorff distance of the pages of $\Sbf$ in $\Bbf_1$ from the closest $m$-plane} \\
	& \Acal_Q(\R^n) && \text{the space of $Q$-tuples of vectors in $\R^n$ (cf. \cite{DLS_MAMS})} \\
    & \Rscr_m(E) && \text{the class of $m$-dimensional integer rectifiable currents supported}\\
    &&& \text{in a relatively closed set $E\subset \R^{m+n}$} \\
    & \Fscr_m(E) && \text{the class of $m$-dimensional integral flat chains supported}\\
    &&& \text{in a relatively closed set $E\subset \R^{m+n}$} \\
    & \Fscr^q_K && \text{the mod$(q)$ flat metric restricted to a given compact set $K$} \\
    & [T] && \text{the mod$(q)$ equivalence class of $T\in \Rscr_m$} \\
    & \partial^q && \text{the boundary operator defined on mod$(q)$ equivalence classes} \\
	& \iota_{z,r} &&\text{the scaling map $p \mapsto \iota_{z,r}(p)\coloneqq \frac{p-z}{r}$ around the center $z$} \\
    & \tau_z && \text{the translation map $p\mapsto \tau_z(p)\coloneqq p+z$}\\
  	& T_{q,r} &&\text{the rescaled current $(\iota_{q,r})_\sharp(T)$} \\
\end{align*}
}

\part{Structure of singularities near $(m-1)$-invariant cones}\label{pt:1}

\section{Preliminaries}

\subsection{Main decay theorem}\label{s:decay}
The main tool of our analysis leading towards the conclusions of Theorem \ref{t:main} is a decay theorem which we state in this section. The theorem states that provided $T$ is at some scale much closer to a cone with exactly $m-1$ directions of translation invariance than it is to any $m$-dimensional plane, then $T$ decays to another cone of this form at a slightly smaller scale. An excess decay theorem of this type (namely, assuming one excess is significantly smaller than a planar excess) was first utilised in the setting of codimension one stable minimal hypersurfaces in the work of Wickramasekera \cite{W14_annals}. We begin by recalling the notions of an \emph{area-minimizing cone mod$(q)$} and an \emph{open book} (also known as a \emph{classical cone}, see e.g. \cite{MW}), as defined in \cite{DLHMSS}.

\begin{definition}[Area-minimizing cones mod$(q)$, \cite{DLHMSS}*{Definition 3.3}]
    Let $S\in \Rscr_m(\R^{m+n})$ be an $m$-dimensional representative mod$(q)$. We say that $S$ is an \emph{area-minimizing cone mod$(q)$} if the following properties hold:
    \begin{itemize}
        \item[(a)] $S$ is locally area-minimizing mod$(q)$ in $\R^{m+n}$;
        \item[(b)] $\partial S=0\,$mod$(q)$;
        \item[(c)] $(\iota_{0,r})_\sharp S = S$ for every $r>0$.
    \end{itemize}
    The linear subspace $V$ of vectors $z$ such that $(\tau_z)_\sharp S = S$ is the \emph{spine} of $S$.
\end{definition}

In this article we focus on area-minimizing cones whose spines have dimension $m-1$. Note that when the spine $V$ is $m$-dimensional, then $\spt^q (S)\subset V$ and, upon giving $V$ the appropriate orientation, $S = \Theta (S, 0) \llbracket V \rrbracket$.  Likewise, cones with $(m-1)$-dimensional spines $V$ can be classified as union of half-spaces meeting at $V$, to which we assign appropriate multiplicities. 

\begin{definition}[Open books, \cite{DLHMSS}*{Definition 4.1}]\label{d:open-book}
An $m$-dimensional \emph{half-plane} (or briefly \emph{half-plane}) $\mathbf{H}$ is any set given by 
\[
\mathbf{H}:=\{x\in \pi : v\cdot x \geq 0\}
\]
for any choice of an $m$-dimensional linear subspace $\pi$ and any element $v\in \pi \cap \partial \Bbf_1$. The $(m-1)$-dimensional linear subspace $V= \{x\in \pi: x\cdot v =0\}$ will be called the \emph{boundary} of $\mathbf{H}$.

For every fixed integer $q\geq 2$ we refer to \emph{open books}, denoted by $\mathscr{B}^q$, as those subsets $\mathbf{S}$ of $\mathbb R^{m+n}$ which are unions of $N\leq q$ $m$-dimensional half-planes $\Hbf_1, \ldots, \Hbf_N$ (often called {\em pages} of the book $\mathbf{S}$) satisfying the following properties:
\begin{itemize}
    \item[(i)] Each $\mathbf{H}_i$ has the same boundary $V$;
    \item[(ii)] Each half-plane $\Hbf_i$ is contained in the same $(m+\bar n)$-dimensional plane $\varpi$.
\end{itemize}
If $x\in \Sigma$, then $\Bscr^q(x)$ will denote the subset of $\Bscr^q$ for which $\varpi = T_x \Sigma$. 

$\mathscr{P}$ and $\mathscr{P} (x)$ will denote the subset of those elements of $\Bscr^q$ and $\Bscr^q (x)$ respectively which consist of a single plane (namely, a half-space $\Hbf_i$ and its reflection across $V$). For $\mathbf{S}\in \Bscr^q\setminus \mathscr{P}$, the plane $V$ in (i) above is referred to as the {\em spine} of $\mathbf{S}$ and will often be denoted by $V (\mathbf{S})$.
\end{definition}

\begin{proposition}[\cite{DLHMSS}*{Proposition 3.5}]\label{p:classification-cones}
    Suppose that $S$ is an $m$-dimensional area-minimizing mod$(q)$ cone in $\R^{m+n}$ with an $(m-1)$-dimensional spine $V$. Then $\spt (S)$ is an open book with spine $V$. In fact we have the following more accurate description. 
    
    There exist distinct $m$-dimensional oriented half-planes $\Hbf_1,\dots,\Hbf_N$ with $\Hbf_i\cap \Hbf_j=V$ for each $i<j$ and $\spt(S) = \Hbf_1\cup\cdots\cup \Hbf_N$, such that for the unit vectors $e_i \in \Hbf_i\cap V^\perp$, the following holds.
    
    There are positive integers $Q_i < \frac{q}{2}$ such that
    \[
        S=\sum_{i=1}^N Q_i \llbracket \Hbf_i \rrbracket,
    \]
    and
    \[
        \sum_{i=1}^N Q_i e_i = 0.
    \]
    Moreover, $\sum_{i=1}^N Q_i = q$, and thus $\Theta(S, x) = \frac{q}{2}$ for each $x\in V$.
\end{proposition}

We do not include the proof here, and simply refer the reader to \cite{DLHMSS}. We note however an important fact. Since $q\geq 3$, a simple consequence of the relations $\sum_i Q_i = q$ and $1\leq Q_i < \frac{q}{2}$ is that the number $N$ of pages $\Hbf_i$ is at least $3$. In particular $\spt (S) \in \mathscr{B}^q\setminus \mathscr{P}$. 

We subsequently recall the \emph{conical $L^2$ height excess} between $T$ and elements in $\Bscr^q$, also used in \cite{DLHMSS} (defined therein via slightly different notation).
\begin{definition}[c.f. \cite{DLHMSS}*{Definition 4.3}]\label{def:L2_height_excess}
	Given a ball $\Bbf_r(x) \subset \R^{m+n}$ and a cone $\mathbf{S}\in \Bscr^q$, we define the \emph{one-sided conical $L^2$ height excess of $T$ relative to $\Sbf$}, denoted $\hat{\Ebf}(T, \mathbf{S}, \Bbf_r(x))$, by
	\[
		\hat{\Ebf}(T, \mathbf{S}, \Bbf_r(x)) \coloneqq \frac{1}{r^{m+2}} \int_{\Bbf_r (x)} \dist^2 (y, \mathbf{S})\, d\|T\|(y).
	\]
At the risk of abusing notation, we additionally define the corresponding \emph{reverse one-sided excess} (cf. \cite{DMS}) as
 \[
\hat{\Ebf} (\mathbf{S}, T, \Bbf_r (x)) \coloneqq \frac{1}{r^{m+2}}\int_{\Bbf_r (x)\cap \mathbf{S}\setminus \Bbf_{ar} (V (\mathbf{S}))}
\dist^2 (y, {\rm spt}\, (T))\, d\mathcal{H}^m (y)\, ,
\]
where $a=a(m)$ is a dimensional constant, to be specified later. We subsequently define the \emph{two-sided conical $L^2$ height excess} as 
\[
    \mathbb{E} (T, \mathbf{S}, \Bbf_r (x)) :=
\hat{\Ebf} (T, \mathbf{S}, \Bbf_r (x)) + \hat{\Ebf} (\mathbf{S}, T, \Bbf_r (x))\, .
\]
We finally introduce the \emph{planar $L^2$ height excess} which is given by
\[
\Ebf^p (T, \Bbf_r (x)) = \min_{\pi\in \mathscr{P} (x)} \hat{\Ebf} (T, \pi, \Bbf_r (x))\, .
\]
\end{definition}
We are now in a position to state our main excess decay result for this part of the paper; it is analogous to \cite{DLHMSS}*{Theorem 4.5}, only with the codimension of $T$ being higher than one, and allowing for a collapsed scenario, where $T$ may be close to an element of $\Pscr$, but will nevertheless be much closer to an element of $\Bscr^q\setminus \Pscr$. Before stating it we isolate a set of assumptions which will be used often throughout Part \ref{pt:1}.

\begin{assumption}\label{a:main}
	$T\in \Rscr_m(\Sigma)$ is an $m$-dimensional representative mod$(q)$ in $\Sigma\cap\Bbf_{7\sqrt{m}}$, where 
 $\Sigma$ is a $C^{3,\alpha_0}$ $(m+\bar{n})$-dimensional Riemannian submanifold of $\R^{m+n}\equiv \R^{m+\bar n+l}$ with $\alpha_0 \in (0,1)$. $q \geq 3$ is a fixed integer and $T$ is area-minimizing mod$(q)$ with $\partial^q [T] = 0$. 
	
    For each $x\in \Sigma \cap \Bbf_{7\sqrt{m}}$, $\Sigma\cap\Bbf_{7\sqrt{m}}(x)$ is the graph of a $C^{3,\alpha_0}$ function $\Psi_x : \Trm_x\Sigma \cap \Bbf_{7\sqrt{m}} \to \Trm_x\Sigma^\perp$. In addition,
	\[
	\boldsymbol{c}(\Sigma\cap\Bbf_{7\sqrt{m}})\coloneqq\sup_{y \in \Sigma \cap \Bbf_{7\sqrt{m}}}\|D\Psi_y\|_{C^{2,\alpha_0}} \leq \bar{\eps},
	\]
	where the small constant $\bar\eps\in (0,1]$ will be determined later. This in particular gives us the following uniform control on the second fundamental form $A_\Sigma$ of $\Sigma$:
	\[
	\Abf \coloneqq \|A_\Sigma\|_{C^0(\Sigma\cap \Bbf_{7\sqrt{m}})} \leq C_0\boldsymbol{c}(\Sigma\cap\Bbf_{7\sqrt{m}})\leq C_0 \bar\eps.
	\]
\end{assumption}

\begin{theorem}[Excess Decay Theorem]\label{t:decay}
Let $m\geq 2,n\geq \bar{n}\geq 2$ and $q\geq 3$ be positive integers, let $\varsigma>0$ and let $Q=\frac{q}{2}$. There are positive constants $\varepsilon_0 = \varepsilon_0(q,m,n,\bar n, \varsigma) \leq \frac{1}{2}$, $r_0 = r_0(q,m,n,\bar n, \varsigma) \leq \frac{1}{2}$ and $C = C(q,m,n,\bar n)$ with the following property. Assume that 
\begin{itemize}
\item[(i)] $T$ and $\Sigma$ are as in Assumption \ref{a:main};
\item[(ii)] $\|T\| (\Bbf_1) \leq (Q+\frac{1}{4}) \omega_m$;
\item[(iii)] There is $\mathbf{S}\in \mathscr{B}^q (0)\setminus\Pscr(0)$ such that 
\begin{equation}\label{e:smallness}
\mathbb{E} (T, \mathbf{S}, \Bbf_1) \leq \varepsilon_0^2 \mathbf{E}^p (T, \Bbf_1)\, 
\end{equation}
\item[(iv)] $\mathbf{A}^2 \leq \varepsilon_0^2 \mathbb{E} (T, \mathbf{S}', \Bbf_1)$ for {\em any} $\mathbf{S}'\in \mathscr{B}^q (0)$.
\end{itemize}
Then there is a $\mathbf{S}'\in \mathscr{B}^q (0) \setminus \mathscr{P} (0)$ such that 
\begin{enumerate}
    \item [\textnormal{(a)}] $\mathbb{E} (T, \mathbf{S}', \Bbf_{r_0}) \leq \varsigma \mathbb{E} (T, \mathbf{S}, \Bbf_1)\,$ \label{e:decay} \\
    \item [\textnormal{(b)}] $\dfrac{\mathbb{E} (T, \mathbf{S}', \Bbf_{r_0})}{\mathbf{E}^p (T, \Bbf_{r_0})} 
\leq 2 \varsigma \dfrac{\mathbb{E} (T, \mathbf{S}, \Bbf_1)}{\mathbf{E}^p (T, \Bbf_1)}$ \\
    \item [\textnormal{(c)}] $\dist^2 (\Sbf^\prime \cap \Bbf_1,\Sbf\cap \Bbf_1) \leq C \mathbb{E} (T, \mathbf{S}, \Bbf_1)$\label{e:cone-change}
    \item[\textnormal{(d)}] $\dist^2 (V (\mathbf{S}) \cap \Bbf_1, V (\mathbf{S}')\cap \Bbf_1) \leq C \dfrac{\mathbb{E}(T,\mathbf{S},\Bbf_1)}{\Ebf^p(T,\Bbf_1)}$\, .\label{e:spine-change}
    \end{enumerate}
\end{theorem}
Our proof of Theorem \ref{t:decay} will follow closely in structure to our previous work \cite{DMS} on the interior regularity of area-minimizing currents in arbitrary codimension, and indeed we will call upon many of the ideas in that work. In turn, \cite{DMS} builds upon on many ideas related to area-minimizing currents first developed by Almgren \cite{Almgren_regularity} and revisited in \cites{DLS14Lp,DLS16blowup,DLS16centermfld}, as well as the foundational ideas of Simon \cite{Simon_cylindrical} and Wickramasekera \cite{W14_annals} (see also \cites{MW, Min, Min2, DLHMSS, DLHMSS-excess-decay, BK} for related works). Given a suitable improved height bound, the technical aspects of the argument are close to the ideas originally in \cite{W14_annals}.

The most important difference when compared to \cite{DMS} (other than the fact that the spine of the cone in question is codimension one relative to the current) is that the mod$(q)$ setting implies a ``no-gap'' condition at the spine of the open book, see Theorem \ref{t:no-gaps}, which is in turn the main reason why the power law decay holds whenever the current is sufficiently close to an open book. The fact that the mod$(q)$ structure implies a ``no-gap'' condition was first observed in \cite{DLHMSS} in the codimension $1$ case, but an important contribution of this paper is that it holds in any codimension. Another key difference in the present work when compared to \cite{DMS} is in understanding the boundary behaviour of blow-ups we construct, particularly in the ``degenerate'' case when the objects in question are converging to a single plane with multiplicity. The reason for this is that now we get Dir-minimizers in the interiors of half-planes, rather than on full-planes, and so naturally we are left with a boundary problem. For this, we use the equivalent estimates of Wickramasekera \cite{W14_annals}*{Section 12} in our setting to understand the boundary behaviour of the blow-ups, and also adapting the excess decay arguments from \cite{W14_annals}*{Section 16} in the non-collapsed case (see also \cite{DLHMSS}) and \cite{W14_annals}*{Section 9}, \cite{MW}*{Theorem 3.1} in the collapsed case (see also \cite{DLHMSS-excess-decay}).

\subsection{Identifying $T$ with integral currents}

In this section we collect some essential facts which follow from the regularity theory developed in \cite{DLHMS}. The first one is that, if the density of $T$ is strictly below $\frac{q}{2}$, then in fact $T$ can be identified with an area-minimizing integral cycle.

\begin{proposition}\label{p:integrality}
Assume $q$, $T$, $\Sigma$, and $U$ are as in Theorem \ref{t:main}. If $\Theta (T, x) < \frac{q}{2}$ for every $x\in U$, then $\partial T \res U =0$ and $T$ is an area-minimizing integral current in $\Sigma\cap U$. 
\end{proposition}

\begin{proof} By \cite{DLHMS}*{Theorem 7.2} the singular set ${\rm Sing}\, (T)$ has Hausdorff dimension at most $m-2$. It follows from the definition of regular set that around every regular point $x$, namely any point in ${\rm Reg}\, (T) := \spt (T) \setminus {\rm Sing}\, (T)$, there is a neighborhood $U'$ of $x$ and a regular oriented $m$-dimensional submanifold $\Lambda$ of $U'\cap \Sigma$ such that 
\[
T\res U' = \Theta (T, x) \llbracket \Lambda \rrbracket\, .
\] 
In particular, since $\Theta (T,x ) < \frac{q}{2}$, ${\rm Reg}\, (T)$ can be written as the union of disjoint orientable regular $m$-dimensional submanifolds $\Lambda_1, \ldots , \Lambda_k$ of $U\setminus {\rm Sing}\, (T)$ such that 
\begin{equation}\label{e:regular-parts}
T \res (U\setminus {\rm Sing}\, (T)) = \sum_{i=1}^k i \llbracket \Lambda_i\rrbracket\, ,
\end{equation}
where $k< \frac{q}{2}$. It follows that $\partial T \res U$ must be supported in ${\rm Sing}\, (T)$. Since $\partial T \res U$ is an $(m-1)$-dimensional flat chain, by a well-known theorem of Federer \cite{Federer}*{Theorem 4.2.14} if one knows that $\mathcal{H}^{m-1} ({\rm Sing}\, (T))=0$ then $\partial T \res U$ must vanish identically. The fact that $T$ is area-minimizing in $U$ as an integral current follows immediately. 
\end{proof}

The above proposition has a very simple corollary.

\begin{corollary}\label{c:integrality}
Assume $q, T, \Sigma$ and $U$ are as in Proposition \ref{p:integrality} and suppose that we have a sequence of 
\begin{itemize}
\item[(i)] $C^{3,\alpha_0}$ regular submanifolds $\Sigma_k$ which converge (locally in $C^{3,\alpha_0}$) to $\Sigma$;
\item[(ii)] Open sets $U_k\uparrow U$;
\item[(iii)] Representatives mod$(q)$ $T_k$ which are area-minimizing flat chains mod$(q)$ in $\Sigma_k \cap U_k$ with $\partial T_k =0 \,$mod$(q)$ and such that $T_k \toweakstar T$ as flat chains mod$(q)$.
\end{itemize}
Then, for every $U'\subset\subset U$ there is $k_0 = k_0 (U')$ such that the conclusions of Proposition \ref{p:integrality} apply to $T_k$ and $U'$ in place of $T$ and $U$ for every $k\geq k_0$. Moreover $T_k \toweakstar T$ in the sense of integral currents.
\end{corollary}
\begin{proof} First of all we recall that $\|T_k\|\toweakstar \|T\|$ in the sense of varifolds, see for instance \cite{DLHMS}. 
It then suffices to show that $\Theta (T_k, y) < \frac{q}{2}$ for every $y\in U'$ and every $k$ sufficiently large. In fact this would imply that $\partial T_k=0$ in $U'$ and, by Federer's compactness theorems, that $T_k \toweakstar S$ for some integral current $S$ with $\partial S \res U' =0$. We immediately conclude that $T$ and $S$ are congruent mod$(q)$. On the other hand it also follows that $\|T_k\|\toweakstar \|S\|$ in the sense of varifolds, in particular $\|S\|=\|T\|$ and hence $\spt (S)=\spt (T)$. Therefore it suffices to show that, for every regular point of $T$ there is a neighborhood $W$ where $T$ and $S$ coincide. We can choose a neighborhood $W$ where $T\res W= i \llbracket \Lambda \rrbracket$ for some oriented connected regular surface $\Lambda$ and for some positive integer $i$. Since $\Theta (T, \cdot)< \frac{q}{2}$, we necessarily have $1\leq i< \frac{q}{2}$, and since $\spt (S)=\spt (T)$, the constancy theorem implies that $S\res W = i' \llbracket\Lambda \rrbracket$ for some (not necessarily positive) integer $i'$. However, $|i'|= i$ because $\|T\|=\|S\|$ and moreover $q$ divides the integer $i'-i$ because $S-T\equiv 0$ mod$(q)$. Given that $|i'-i|\leq |i'|+|i|=2|i|<q$, it must necessarily be that $i=i'$. 

As for showing that $\Theta (T_k, y) < \frac{q}{2}$ for every $y\in U'$ and every $k$ sufficiently large, we argue by contradiction. In particular we would find a sequence of points $\{y_k\}\subset U'$ such that $\Theta (T_k, y_k)\geq \frac{q}{2}$. Because $U'\subset\subset U$ we can assume that $y_k$ converges to an element $y\in U$. On the other hand Allard's monotonicity formula would imply that $\Theta (T, y) \geq \limsup_k \Theta (T_k, y_k) \geq \frac{q}{2}$. 
\end{proof}

\subsection{Oriented tilt-excess}
Having identified $T$ with an integral current locally in regions where the density is strictly smaller than $\frac{q}{2}$, we recall
the definition of the \emph{oriented tilt-excess} of an $m$-dimensional integral current $T$ in a cylinder $\Cbf_r(p,\pi_0)$ relative to an $m$-dimensional oriented plane $\pi$:
\[
\mathbf{E} (T, \mathbf{C}_r (p,\pi_0), \pi) := \frac{1}{2\omega_m r^m} \int_{\mathbf{C}_r (p,\pi_0)} |\vec{T} (x)- \vec{\pi} (x)|^2\, d\|T\| (x).
\]
We in turn define
\[
\mathbf{E} (T, \mathbf{C}_r (p,\pi_0)) := \min_{\pi\subset T_p \Sigma} \mathbf{E} (T, \mathbf{C}_r (p,\pi_0), \pi)\, 
\]
where the minimum is taken over all $m$-dimensional oriented planes $\pi\subset T_p\Sigma$ (identified with their corresponding planes in $\mathbb{R}^{n+m}$). The oriented tilt-excess of $T$ in $\Bbf_r(p)$ relative to an $m$-dimensional oriented plane $\pi$ and the optimal oriented tilt-excess in $\Bbf_r(p)$, denoted respectively by $\Ebf(T,\Bbf_r(p),\pi)$ and $\Ebf(T,\Bbf_r(p))$, are defined analogously.

\subsection{$L^2-L^\infty$ height bound} 
In this section, we recall the $L^2-L^\infty$ height bound derived in \cite{DMS}*{Part I, Section 3} for integral currents, which will be a key aspect of the proof of Theorem \ref{t:decay}. We refer the reader to the proofs therein.

Bearing in mind Proposition \ref{p:integrality}, which will allow us to identify $T$ with an integral current away from a neighbourhood of $V(\Sbf)$, we make the following assumption throughout this section.

\begin{assumption}\label{a:height-main}
    $\Sigma$ and $\Abf$ are as in Assumption \ref{a:main}. $T$ is an $m$-dimensional integral current in $\Sigma\cap \Bbf_{7\sqrt{m}}$ with $\partial T\mres \Bbf_{7\sqrt{m}} = 0$. For some oriented $m$-dimensional plane $\pi_0\subset\R^{m+\bar n}$ passing through the origin and some positive integer $Q$, we have 
    \[
    (\mathbf{p}_{\pi_0})_\sharp T\res \Cbf_2 (0, \pi_0) = Q \llbracket B_2(\pi_0)\rrbracket\, ,
    \]
    and $\|T\|(\Bbf_2) \leq (Q+\frac{1}{2})\omega_m 2^m$.
\end{assumption}

\begin{theorem}[$L^\infty$ and tilt-excess estimates, \cite{DMS}*{Theorem 3.2}]\label{t:heightbd} For every $1\leq r < 2$, $Q$, and $N$, there is a positive constant $\bar{C} = \bar{C} (Q,m,n,\bar n,N,r)>0$ with the following property. Suppose that $T$, $\Sigma$, $\Abf$ and $\pi_0$ are as in Assumption \ref{a:height-main}, let $p_1, \ldots , p_N \in \pi_0^\perp$ be distinct points, and set $\boldsymbol{\pi}:= \bigcup_i p_i+\pi_0$. 
Let
\begin{equation}\label{e:L2-excess}
E := \int_{\mathbf{C}_2} \dist^2 (p, \boldsymbol{\pi}) \, d\|T\| (p)\, .
\end{equation}
Then
\begin{equation}\label{e:tilt-estimate}
\Ebf(T,\Cbf_r, \pi_0)\leq \bar{C} (E + \mathbf{A}^2)\, 
\end{equation}
and, if $E\leq 1$,
\begin{equation}\label{e:Linfty-estimate}
\spt (T) \cap \mathbf{C}_r \subset \{p : \dist (p, \boldsymbol{\pi})\leq \bar{C} (E^{1/2} + \mathbf{A})\}\, .
\end{equation}
\end{theorem}

We additionally state the following important consequence of Theorem \ref{t:heightbd}, also contained in \cite{DMS}.

\begin{corollary}[\cite{DMS}*{Corollary 3.3}]\label{c:splitting-0}
For each pair of positive integers $Q$ and $N$, there is a positive constant $\delta = \delta (Q,m,n, \bar n, N)$ with the following properties. Assume that:
\begin{itemize}
    \item[(i)] $T$, $\Sigma$, and $\Abf$ are as in Assumption \ref{a:height-main};
    \item[(ii)] $T$ is area-minimizing in $\Sigma$ and for some positive $r \leq \frac{1}{4}$ and $q\in \spt(T)\cap\Bbf_1$ we have $\partial T \res \mathbf{C}_{4r} (q) = 0$, $(\mathbf{p}_{\pi_0})_\sharp T = Q \llbracket B_{4r} (q)\rrbracket$, and $\|T\| (\mathbf{C}_{2r} (q)) \leq \omega_m (Q+\frac{1}{2}) (2r)^m$;
    \item[(iii)] $p_1, \ldots, p_N\in \R^{m+n}$ are distinct points with $\mathbf{p}_{\pi_0} (p_i)= q$ and $\varkappa:=\min \{|p_i-p_j|: i<j\}$;
    \item[(iv)] $\pi_1, \ldots, \pi_N$ are oriented planes passing through the origin with 
    \begin{equation}\label{e:smallness-tilted-pi-1}
    \tau \coloneqq \max_i |\pi_i-\pi_0|\leq \delta \min\{1, r^{-1} \varkappa\}\, ;
    \end{equation}
    \item[(v)] Upon setting $\boldsymbol{\pi} = \bigcup_i (p_i+ \pi_i)$, we have 
    \begin{align}
    & (r\mathbf{A})^2 + (2r)^{-m-2} \int_{\mathbf{C}_{2r} (q)} \dist^2 (p, \boldsymbol{\pi}) d\|T\| 
    \leq \delta^2 \min \{1, r^{-2}\varkappa^2\}\, .\label{e:smallness-tilted-pi-2}
    \end{align}
\end{itemize}
Then $T\res \mathbf{C}_r (q) = \sum_{i=1}^N T_i$ where
\begin{itemize}
\item[(a)] Each $T_i$ is an integral current with $\partial T_i \res \mathbf{C}_r (q) = 0$;
\item[(b)] $\dist (q, \boldsymbol{\pi}) = \dist (q, p_i+\pi_i)$ for each $q\in \spt (T_i)$;
\item[(c)] $(\mathbf{p}_{\pi_0})_\sharp T_i = Q_i \llbracket B_r (q)\rrbracket$ for some non-negative integer $Q_i$.
\end{itemize}
\end{corollary}

\subsection{Separation between planes and halfplanes}\label{s:distance-pages}
Fist of all we introduce the following quantity for every open book $\mathbf{S} = \mathbf{H}_1\cup \cdots \cup \mathbf{H}_N$. We denote by $\pi_i$ the planes that are the extensions of the half-planes $\mathbf{H}_i$ across $V(\Sbf)$ via reflection and set 
\begin{equation}\label{e:mu}
    \boldsymbol{\zeta}(\Sbf) \coloneqq \max_{i<j} \dist(\pi_i\cap\Bbf_1,\pi_j\cap\Bbf_1). 
\end{equation}
Note that $\boldsymbol{\zeta} (\mathbf{S})$ measures how close $\mathbf{S}$ is to a plane. In particular it is not difficult to see that 
\begin{equation}\label{e:comparison-1}
C^{-1} \min_i \dist (\mathbf{S}\cap \Bbf_1, \pi_i\cap \Bbf_1)
\leq \boldsymbol{\zeta} (\mathbf{S})
\leq C \min_i \dist (\mathbf{S}\cap \Bbf_1, \pi_i\cap \Bbf_1) 
\end{equation}
while
\begin{equation}\label{e:comparison-2}
\min_{\pi\in \mathscr{P}} \dist (\mathbf{S}\cap \Bbf_1, \pi\cap \Bbf_1) \leq
\min_i \dist (\mathbf{S}\cap \Bbf_1, \pi_i\cap \Bbf_1) \leq C
\min_{\pi\in \mathscr{P}} \dist (\mathbf{S}\cap \Bbf_1, \pi\cap \Bbf_1) 
\end{equation}
for some constant $C(m,n)$.

The quantity $\boldsymbol{\zeta} (\mathbf{S})$ will play an analogous role to that of $\boldsymbol{\mu} (\mathbf{S})$ in \cite{DMS}, but we want to emphasize one difference: the cones $\mathbf{S}$ considered in \cite{DMS} are union of {\em planes} and so there is no need to introduce the reflection along the spine $V (\mathbf{S})$. For this reason we have opted to use a different notation here. Instead we use the same notation for $\boldsymbol{\sigma} (\mathbf{S})$, which is given by
\begin{equation}\label{e:sigma}
\boldsymbol{\sigma} (\mathbf{S}) = \min_{i<j}
\dist (\mathbf{H}_i \cap \Bbf_1, \mathbf{H}_j \cap \Bbf_1)\, .
\end{equation}
The following very elementary observation will be particularly useful in many places.

\begin{lemma}\label{l:trivial}
Let $\mathbf{H}_1$ and $\mathbf{H}_2$ be two $m$-dimensional halfplanes with a common boundary $V$, denote by $\pi_i$ the planes obtained by reflecting them along $V$, and let $v_i\in \mathbf{H}_i\cap \partial \Bbf_1$ be orthogonal to $V$. Then
\begin{align}
\dist (\pi_1 \cap \Bbf_1, \pi_2\cap \Bbf_1) &= \min \{|v_1-v_2|, |v_1+v_2|\}\, .\label{e:trivial-1}\\
\end{align}
\end{lemma}

\subsection{Alignment of spines, shifting}
We now record some results which are analogous to those in \cite{DMS}*{Section 7.3}. Therein, the majority of the results are stated for unions of $m$-dimensional planes intersecting in an $(m-2)$-dimensional subspace, with comparability of the two \emph{Morgan angles} that determine the pairwise separation of the planes. Here, we instead have unions of $m$-dimensional half-planes intersecting in an $(m-1)$-dimensional subspace, with only one angle parameter determining the pairwise separation of the half-planes, so the situation is in fact even simpler. Nevertheless, for clarity, we include the proofs that differ.

\begin{lemma}\label{l:spine-comparison}
	For each $M>0$ and $m,n,q\in \N$, there exists a constant $\bar{C}=\bar{C}(M,m,n,q)>0$ such that the following holds.  Suppose that $\Sbf$ and $\Sbf'$ are open books in $\Bscr^q$, consisting of $2 \leq N, N' \leq q$ $m$-dimensional half-planes $\Hbf_1,\dots,\Hbf_N$ and $\Hbf'_1,\dots,\Hbf'_{N'}$, meeting in the $(m-1)$-dimensional subspaces $V(\Sbf)$ and $V(\Sbf')$ respectively. Then 
	\begin{equation}
		\dist(V(\Sbf)\cap\Bbf_1,V(\Sbf')\cap\Bbf_1) \leq\bar C \frac{\dist(\Sbf\cap\Bbf_1,\Sbf'\cap\Bbf_1)}{\boldsymbol{\zeta} (\mathbf{S})}
	\end{equation}
\end{lemma}

\begin{proof} First of all observe that, without loss of generality, we can assume that $\dist (\mathbf{S}\cap \Bbf_1, \mathbf{S}'\cap \Bbf_1) \leq \gamma \boldsymbol{\zeta} (\mathbf{S})$ for some fixed positive $\gamma$. In fact we have otherwise that 
\[
\frac{\dist(\Sbf\cap\Bbf_1,\Sbf'\cap\Bbf_1)}{\boldsymbol{\zeta} (\mathbf{S})} \geq \gamma
\]
and, given that $\dist (V(\Sbf)\cap\Bbf_1,V(\Sbf')\cap\Bbf_1)\leq 1$, the desired inequality would be trivially true if the constant $C$ is taken to be larger than $\gamma^{-1}$. 

Next, we have by the triangle inequality
\[
    \boldsymbol{\zeta}(\Sbf) \leq 2\dist(\Sbf\cap \Bbf_1, \Sbf'\cap\Bbf_1) + \boldsymbol\zeta(\Sbf'),
\]
and thus if $\gamma<1/4$, $\boldsymbol{\zeta} (\mathbf{S}') \geq \frac{1}{2} \boldsymbol{\zeta} (\mathbf{S})>0$.

Next, let $\pi_i$ and $\pi'_i$ be the $m$-dimensional planes which are the extensions of $\Hbf_i$ and $\Hbf_i'$ across $V(\Sbf)$ and $V(\Sbf')$, respectively, by reflection. If we set $\bar{\mathbf{S}} := \cup_i \pi_i$ and $\bar{\mathbf{S}}' := \cup_i \pi'_i$, it is obvious that:
\begin{itemize}
\item $\boldsymbol{\zeta} (\mathbf{S}) =
\boldsymbol{\zeta} (\bar{\mathbf{S}})$ and 
 $\boldsymbol{\zeta} (\mathbf{S}') =
\boldsymbol{\zeta} (\bar{\mathbf{S}}')$;
\item neither $\bar{\mathbf{S}}$ nor $\bar{\mathbf{S}}'$ is planar, in particular $V (\bar{\mathbf{S}})$ and $V (\bar{\mathbf{S}}')$ are both well-defined and they equal, respectively, $V (\mathbf{S})$ and $V (\mathbf{S}')$;
\item $\dist (\bar{\mathbf{S}}\cap \Bbf_1, \bar{\mathbf{S}}'\cap \Bbf_1) 
\leq \dist (\mathbf{S} \cap \Bbf_1, \mathbf{S}'\cap \Bbf_1)$.
\end{itemize}
On the other hand, it follows from the very same argument of \cite{DMS}*{Lemma 7.12} that 
\[
\dist (V (\bar{\mathbf{S}}) \cap \Bbf_1, V (\bar{\mathbf{S}}')\cap \Bbf_1) \leq C 
\frac{\dist (\bar{\mathbf{S}}\cap \Bbf_1, \bar{\mathbf{S}}'\cap \Bbf_1)}{\boldsymbol{\zeta} (\bar{\mathbf{S}})}\, .
\]
In fact the proof of \cite{DMS}*{Lemma 7.12} handles a more complicated situation because the spines have codimension 2 in the planes, a fact which requires the additional assumption \cite{DMS}*{Lemma 7.12(iii)}.
\end{proof}

We additionally require a lemma which gives control the shifting of a given open book $\Sbf$. Before we state the lemma, let us introduce the following definition.

\begin{definition}\label{d:rot-inv}
A set $\Omega\subset \mathbb R^{m+n}$ is said to be \textit{invariant under rotation around} a linear subspace $V$ if $R (\Omega)=\Omega$ for any rotation $R$ of $\mathbb R^{m+n}$ which fixes $V$.
\end{definition}

\begin{lemma}\label{l:shift}
	Let $M\geq 1$, $m,n,q\in \N$, $2\leq N\leq q$, let $V$ be an $(m-1)$-dimensional subspace of $\R^{m+n}$ and let $U\subset\Bbf_1$ be a non-empty open set which is invariant under rotation around $V$. Then there exists a constant $\bar C = \bar C(M,m,n,N,U)>0$ for which the following holds. Suppose that $\Sbf = \Hbf_1\cup\cdots \cup\Hbf_N \in \Bscr^q$ and $V(\Sbf) = V$ for each $i < j$. Let $p\in\Bbf_{1/2}$ and $\pi_i$ be the $m$-dimensional plane which is the union of $\Hbf_i$ and its reflection through $V$. Then there exists $j_0\in \{1,\dots,N\}$ and $\Omega\subset\Hbf_{j_0}\cap U$ with $\Hcal^m(\Omega) \geq \bar C^{-1}$ and
	\begin{equation}
		|\mathbf{p}_{\pi_i}^\perp(p)| + \boldsymbol{\zeta}(\Sbf)|\mathbf{p}_{V^\perp\cap\pi_i}(p)| \leq \bar C \dist(z,p+\Sbf) \qquad \forall z\in \Omega, \ \forall i\, .
	\end{equation}
\end{lemma}
\begin{remark}
In fact the set $\Omega$ is relatively open in $\mathbf{H}_{j_0}\cap U$, even though this fact will not play an important role.    
\end{remark}
\begin{proof} First of all, we claim, as in \cite{DMS}*{Lemma 7.16}, that
\begin{equation}\label{e:DMS-7.16}
|\mathbf{p}_{\pi_i}^\perp(p)| + \boldsymbol{\zeta}(\Sbf)|\mathbf{p}_{V^\perp\cap\pi_i}(p)| 
\leq 6 \max_k |\mathbf{p}_{\pi_k}^\perp (p)|\qquad \forall i\, .
\end{equation}
The argument is much simpler here and we give a direct proof. Without loss of generality we may just prove the claim for $i=1$, so we just need to show that
\begin{equation}\label{e:DMS-7.16-2}
\boldsymbol{\zeta}(\Sbf)|\mathbf{p}_{V^\perp\cap\pi_1}(p)| 
\leq 5 \max_k |\mathbf{p}_{\pi_k}^\perp (p)|\, .
\end{equation}
Let $j_0$ be a maximizer of the right hand and then pick $j$ which maximizes $\dist (\pi_j \cap \Bbf_1, \pi_{j_0} \cap \Bbf_1)$, so that in particular
the triangle inequality yields
\[
\dist (\pi_{j_0} \cap \Bbf_1, \pi_{j} \cap \Bbf_1)\geq \frac{1}{2} \boldsymbol{\zeta} (\mathbf{S})\, .
\]
Let $v_j$ be the vector in $\partial \Bbf_1 \cap \Hbf_j$ orthogonal to $V$. Observe that $\mathbf{p}_{V^\perp\cap \pi_j} (p) = (p\cdot v_j) v_j$ and that 
\[
\dist (\pi_{j_0}\cap \Bbf_1, \pi_j \cap \Bbf_1) = |\mathbf{p}_{\pi_{j_0}}^\perp (v_j)|\, .
\]
In particular, since $\mathbf{p}_{V^\perp\cap \pi_j} = \mathbf{p}_V^\perp - \mathbf{p}_{\pi_j}^\perp$ and  $\mathbf{p}_{\pi_{j_0}}^\perp \circ \mathbf{p}_V^\perp =\mathbf{p}_{\pi_{j_0}}^\perp$ (because $V\subset\pi_{j_0}$), we have
\begin{align*}
\boldsymbol{\zeta}(\mathbf{S}) |\mathbf{p}_{V^\perp\cap \pi_j} (p)|
&\leq 2 \dist (\pi_{j_0}\cap \Bbf_1, \pi_j \cap \Bbf_1) |\mathbf{p}_{V^\perp\cap \pi_j} (p)|\\
& \leq 2
 (|\mathbf{p}_{\pi_{j_0}}^\perp (\mathbf{p}_V^\perp (p))| + |\mathbf{p}_{\pi_j}^\perp (p)|)\leq 2 (|\mathbf{p}_{\pi_{j_0}}^\perp (p)| +|\mathbf{p}_{\pi_j}^\perp (p)|)
 \leq 4 |\mathbf{p}_{\pi_{j_0}}^\perp (p)|
\, .
\end{align*}
Moreover, since $\mathbf{p}_{V^\perp\cap \pi_1}\circ\mathbf{p}_{\pi_j} = \mathbf{p}_{V^\perp\cap \pi_1}\circ\mathbf{p}_{V^\perp\cap\pi_j}$ and by the definition of $\pi_{j_0}$, we have
\begin{align*}
|\mathbf{p}_{V^\perp \cap \pi_1} (p)| &\leq |\mathbf{p}_{V^\perp \cap \pi_1} (\mathbf{p}_{V^\perp\cap \pi_j} (p))| + |\mathbf{p}_{V^\perp\cap \pi_1} (\mathbf{p}_{\pi_j}^\perp (p))|\\
&\leq |\mathbf{p}_{V^\perp\cap \pi_j} (p)| + |\mathbf{p}_{\pi_j}^\perp (p)|
\leq |\mathbf{p}_{V^\perp\cap \pi_j} (p)| + |\mathbf{p}_{\pi_{j_0}}^\perp (p)|\, .
\end{align*}
This shows \eqref{e:DMS-7.16-2} and hence \eqref{e:DMS-7.16}.

\medskip

It remains to check that $\max_k|\mathbf{p}_{\pi_k}^\perp(p)|\leq C \dist(z,p+\Sbf)$ for any $z\in\Omega$. Fix again $j_0$ such that $|\mathbf{p}_{\pi_{j_0}}^\perp (p)|$ is maximal. We thus aim at showing the existence of a subset $\Omega\subset \mathbf{H}_{j_0} \cap U$ with $\mathcal{H}^m (\Omega) \geq \bar{C}^{-1}$ with the property that
\[
|\mathbf{p}_{\pi_{j_0}}^\perp(p)| \leq \bar{C} \dist (z, p + \mathbf{S}) \qquad \forall z\in \Omega\, .
\]
We follow the argument of \cite{DMS}*{Proof of Lemma 7.14}, albeit in a much simpler setting. First we notice that $\dist (z, p+\mathbf{S}) \geq \min_i |\mathbf{p}_{\pi_i}^\perp (p-z)|$ and we therefore aim at showing that 
\begin{equation}\label{e:DMS-7.14}
|\mathbf{p}_{\pi_{j_0}}^\perp (p)| \leq \bar{C} \min_i |\mathbf{p}_{\pi_i}^\perp (p-z)|\, .
\end{equation}
We let $e\in \Hbf_{j_0}\cap \partial \Bbf_1$ be such that $e\perp V$ and observe that, since $U$ is non-empty, open, and rotationally invariant around $V$, there must be an open $W\subset V$ and an interval $[a,b]$ with the property that 
\[
w+\lambda e \in U \qquad \forall w\in W, \forall \lambda \in [a,b]\, .
\]
Observe that for such $w\in W$ and $\lambda\in [a,b]$, 
$|\mathbf{p}_{\pi_i}^\perp (p-(w+\lambda e))| = |\mathbf{p}_{\pi_i}^\perp (p-\lambda e)|$.

Consider the finite collection of $N+1$ pairwise disjoint subintervals $I_1, \ldots, I_{N+1}$ contained in $[a,b]$, of length $\frac{b-a}{2(N+1)}$ with $\dist(I_i,I_j)\geq\frac{b-a}{2(N+1)}$ for all $i\neq j$, where $N$ is the number of half-planes of the open book, and we reduce to showing that there is one $I_\ell$ for which
\begin{equation}\label{e:DMS-7.14-1}
|\mathbf{p}_{\pi_{j_0}}^\perp (p)| \leq \bar{C} \min_i |\mathbf{p}_{\pi_i}^\perp (p-\lambda e)| \qquad \forall \lambda\in I_\ell\, .
\end{equation}
Once we have established this, the proof is complete. So, suppose this were not to hold. Then, given $\bar C$ sufficiently large (to be determined), for each $k=1,\dotsc,N+1$ we may find points $\lambda_k\in I_k$ for which \eqref{e:DMS-7.14-1} fails for this choice of $\bar C$. This means that for at least one pair of distinct $\lambda_k$, the minimum on the right hand side is achieved by one same index $i$. For this $i$ we must have two numbers $\lambda$ and $\mu$ with $|\lambda-\mu|\geq\frac{b-a}{2(N+1)}$ and
\begin{align*}
|\mathbf{p}_{\pi_i}^\perp (p-\lambda e)|
&< \bar{C}^{-1} |\mathbf{p}_{\pi_{j_0}}^\perp (p)|\\
|\mathbf{p}_{\pi_i}^\perp (p-\mu e)|
&< \bar{C}^{-1} |\mathbf{p}_{\pi_{j_0}}^\perp (p)|\, .
\end{align*}
In particular, using the linearity of $\mathbf{p}_{\pi_i}^\perp$, the triangle inequality and the fact that $\lambda,\mu \leq 1$, we easily conclude that 
\begin{align}
|\mathbf{p}_{\pi_i}^\perp (e) &|< \frac{4(N+1)}{\bar{C} (b-a)}
|\mathbf{p}_{\pi_{j_0}}^\perp (p)|\label{e:DMS-7.14-2}\\
|\mathbf{p}_{\pi_i}^\perp (p)| &< \frac{4(N+1)}{\bar{C} (b-a)}
|\mathbf{p}_{\pi_{j_0}}^\perp (p)|\label{e:DMS-7.14-3}
\end{align}
Given that $V\oplus \R e = \pi_{j_0}$, we easily conclude that 
\begin{equation}\label{e:DMS-7.14-4}
|\mathbf{p}_{\pi_i}^\perp (z)|
< \frac{4(N+1)}{\bar{C} (b-a)}
|\mathbf{p}_{\pi_{j_0}}^\perp (p)| \qquad \forall z\in \Bbf_1\cap \pi_{j_0}\, .
\end{equation}
Observe that, since $p\in \Bbf_{1/2}$, \eqref{e:DMS-7.14-4} further implies that $|\pi_i-\pi_{j_0}|$ is small, provided $\bar{C}$ is large enough.
Consider now $p+\pi_i^\perp$: if $|\pi_i-\pi_{j_0}|$ is smaller than a geometric constant, then $p+\pi_i^\perp$ intersects $\pi_{j_0}$ at some point $z\in \pi_{j_0}\cap\Bbf_1$. Now $p-z$ is orthogonal to $\pi_i$ and therefore 
\[
\mathbf{p}^{\perp}_{\pi_i} (p-z) = p-z
\]
But since $z\in \pi_{j_0}$, 
\[
|\mathbf{p}^{\perp}_{\pi_i} (p-z)| = |p-z|\geq |\mathbf{p}^{\perp}_{\pi_{j_0}} (p)|\, .
\]
On the other hand, because of \eqref{e:DMS-7.14-3} and \eqref{e:DMS-7.14-4} we also know that 
\[
|\mathbf{p}^{\perp}_{\pi_i} (p-z)| < \frac{8(N+1)}{\bar{C} (b-a)}
|\mathbf{p}_{\pi_{j_0}}^\perp (p)|\, .
\]
If we choose $\bar{C}$ large enough so that $\frac{8(N+1)}{\bar{C} (b-a)}<1$, we would then reach a contradiction. This completes the proof.
\end{proof}

\section{No gaps in highest multiplicity points}\label{s:no-gaps}

In order to prove Theorem \ref{t:decay} we require the following theorem, which tells us that under the hypotheses (i), (ii) of Theorem \ref{t:decay} and a suitable smallness condition on the ratios $(\mathbf{A}^2 + \mathbb{E} (T, \Sbf, \Bbf_1))/ \boldsymbol{\zeta} (\Sbf)^2$, there are no large gaps in the multiplicity $Q \geq \frac{q}{2}$ points of $T$ near $V(\Sbf)$.

\begin{theorem}\label{t:no-gaps}
For every $\varrho>0$, $\eta>0$, and $\rho \geq 5\varrho$, there exists $\eps=\eps(q,m,n,\bar n,\varrho, \eta, \rho )>0$ such that the following holds. Suppose that $T$, $\Sigma$ and $\Sbf$ satisfy the hypotheses (i), (ii) of Theorem \ref{t:decay} and 
\begin{equation}\label{e:small-conical-max-sep-ratio}
    \Ebb(T,\Sbf,\Bbf_1) + \Abf^2 \leq \eps^2\boldsymbol{\zeta}(\Sbf)^2\, .
\end{equation}
Then
\begin{equation}\label{e:no-gaps}
    \Bbf_{\varrho} (\xi) \cap \left\{p: \Theta (T,p)\geq \frac{q}{2}\right \}\neq \emptyset \qquad \forall \xi \in V (\mathbf{S})\cap \Bbf_{1/2}\, ,
\end{equation}
and 
\begin{equation}\label{e:density-drop}
	\Theta(T,x)<\frac{q}{2} \qquad \text{for every $x\in \Bbf_{1-\eta/8} \setminus \Bbf_{\rho/4}(V)$.}
\end{equation}
\end{theorem}

\begin{proof}[Proof of Theorem \ref{t:no-gaps}]
    We argue by contradiction and consider two cases depending on whether \eqref{e:no-gaps} or \eqref{e:density-drop} fails. \\
    \underline{Case 1: \eqref{e:no-gaps} fails.} In this case, there exists $\varrho>0$ for which there exists a sequence $T_k$, $\Sigma_k$ and $\Sbf_k$ satisfying the hypotheses (i), (ii) of Theorem \ref{t:decay} and \eqref{e:small-conical-max-sep-ratio} with $\eps = \eps_k \downarrow 0$, with
    \begin{equation}\label{e:gaps}
        \Bbf_{\varrho} (\xi_k) \cap \left\{p: \Theta (T_k,p)\geq \frac{q}{2}\right\}=\emptyset \qquad \text{for some $\xi_k \in V (\mathbf{S}_k)\cap \Bbf_{1/2}$}\, .
    \end{equation}
Without loss of generality, after extracting subsequences, scaling and translating, we can assume that:
\begin{itemize}
\item $\xi_k =0$;
\item $\varrho=\frac{1}{2}$;
\item $\Sigma_k$ converge to an $(m+\bar n)$-dimensional subspace $\varpi$;
\item $\mathbf{S}_k\cap\Bbf_1$ converge to an open book $\mathbf{S}\cap\Bbf_1\subset \varpi$ in Hausdorff distance;
\item $T_k$ converge to an area-minimizing current mod$(q)$ for which we fix a mod$(q)$ representative $T$. 
\end{itemize}
Note that by \cite{DLHMS}*{Proposition 4.2} $\|T_k\|$ converges to $\|T\|$ in the sense of varifolds and that, moreover, $\spt (T) \cap \mathbf{B}_{1/2}= \mathbf{S} \cap \mathbf{B}_{1/2}$. We now distinguish two cases. 

\medskip

{\bf $\mathbf{S}$ is not supported in an $m$-dimensional plane.} By Proposition \ref{p:integrality}, $\partial T_k \res \mathbf{B}_{1/2} =0$ and $T_k\res \mathbf{B}_{1/2}$ is an area-minimizing integral current and thus, by the Federer-Fleming compactness theorem, $T_k \res \mathbf{B}_{1/2}$ converges, up to subsequences, to an area-minimizing integral current $S$ with $\partial S \res \mathbf{B}_{1/2} = 0$. This also implies that $\|T_k\|\res \mathbf{B}_{1/2} \rightharpoonup^* \|S\| \res \mathbf{B}_{1/2}$. In particular $\|T\| \res \mathbf{B}_{1/2} = \|S\|\res\mathbf{B}_{1/2}$. Thus $\spt (T) \cap \mathbf{B}_{1/2}= \spt (S) \cap \mathbf{B}_{1/2}=\mathbf{S} \cap \mathbf{B}_{1/2}$, where the latter identity follows from the hypothesis \eqref{e:small-conical-max-sep-ratio}. But this would imply that $S$ is an integral area-minimizing cone with an $(m-1)$-dimensional spine, which is not possible.

\medskip

{\bf $\mathbf{S}$ is supported in an $m$-dimensional plane $\pi$.} In this case $\boldsymbol{\zeta}(\Sbf_k)$ converges to $0$. In particular, \eqref{e:small-conical-max-sep-ratio} and the constancy theorem for area minimizing currents mod$(q)$ tells us that, up to choosing the right orientation for $\pi$, $T_k \to Q' \llbracket \pi \rrbracket$ for some positive integer $Q' \leq \frac{q}{2}$. Moreover, if we consider planes $\pi_k$ containing $V$ which minimize $\dist (\mathbf{S}_k\cap \Bbf_1, \pi\cap \Bbf_1)$, we easily conclude that 
$\dist (\mathbf{S}_k\cap \Bbf_1, \pi_k \cap \Bbf_1) \leq \boldsymbol{\zeta} (\mathbf{S}_k)$.
Therefore 
\begin{equation}\label{e:planar-excess-max-sep}
    \mathbf{E}^p (T_k, \Bbf_1) \leq \hat{\mathbf{E}} (T_k, \pi_k, \Bbf_1)  \leq C\hat\Ebf(T,\Sbf,\Bbf_1) + C\boldsymbol{\zeta}(\Sbf_k)^2  \leq C(1+ \eps_k^2)\boldsymbol{\zeta}(\Sbf_k)^2 \to 0 \, .
\end{equation}

Without loss of generality, after applying a rotation we can assume that $\pi_k = \pi$. We can then use \cite{DLS14Lp}*{Theorem 2.4} (which can be applied because the tilt excess in a slightly smaller ball is controlled by the planar $L^2$ excess, c.f. \cite{DMS}), together with \eqref{e:planar-excess-max-sep}, to approximate $T_k \res \mathbf{B}_{1/4} \cap \mathbf{C}_{c_0} (\pi, 0)$ with the graph of a Lipschitz map $f_k:B_{c_0}(\pi,0) \to \mathcal{A}_{Q'} (\mathbb R^n)$ with 
\[
\int_{B_{c_0}}  (|f_k|^2 + |Df_k|^2) \leq C_0 \boldsymbol{\zeta} (\Sbf_k)^2 \, .
\]
where $c_0 \leq \frac{1}{2}$ is a dimensional constant and $C_0= C_0(m,n,Q')$. Up to subsequences, we also know that the rescaled functions
\[
\frac{f_k}{\boldsymbol{\zeta}(\Sbf_k)}
\]
converge strongly in $L^2$ to a map $\bar{f}$, which is Dir-minimizing in light of \cite{DLHMS}*{Theorem 5.2} and \eqref{e:planar-excess-max-sep}.
Consider now the sequence of open books $\mathbf{S}'_k$ which are obtained $\mathbf{S}_k$ in the following way. Writing every point $x\in\mathbb R^{m+n}$ as $x=y+z$ with $y\in \pi$ and $z\in \pi^\perp$ we denote by $L_k:\mathbb R^{m+n}\to \mathbb R^{m+n}$ the linear function which maps $x$ into $y+ (\boldsymbol{\zeta} (\Sbf_k))^{-1} z$. We then set $\mathbf{S}'_k = L_k (\Sbf_k)$. Up to subsequences we can assume that $\Sbf'_k$ converges to an open book $\mathbf{S}'$, which is necessarily non-planar due to the normalization by $\boldsymbol{\zeta}(\Sbf_k)$. 

However, using \cite{DLS14Lp}*{Theorem 2.4} and our assumption that $\mathbb{E} (T_k, \mathbf{S}_k, \Bbf_1)$ is infinitesimal compared to $\boldsymbol{\zeta}(\Sbf_k)^2$, we conclude that the support of the graph of $\bar f$ coincides with $\mathbf{S}'$. But then $\bar{f}$ would be a $Q'$-valued Dir-minimizing map on $B_{c_0}(\pi,0)$ with an $(m-1)$-dimensional singular set, which is a contradiction to Almgren's regularity theory, cf. \cite{DLS_MAMS}.

\underline{Case 2: \eqref{e:density-drop} fails.} Now, we suppose there exists $\rho > 0$ for which there exists sequences $T_k$, $\Sigma_k$ and $\Sbf_k$ satisfying hypotheses (i), (ii) of Theorem \ref{t:main} and \eqref{e:small-conical-max-sep-ratio} with $\eps = \eps_k \downarrow 0$, but for which $\Theta(T_k,x_k) \geq \frac{q}{2}$ for some $x_k\in \Bbf_{1-\eta/8}\setminus \Bbf_{\rho/4}(V)$. We may extract subsequences to assume without loss of generality that
	\begin{itemize}
		\item $\Sigma_k$ converges to an $(m+\bar n)$-dimensional subspace $\varpi$;
		\item $\mathbf{S}_k\cap\Bbf_1$ converges to an open book $\mathbf{S}\cap\Bbf_1\subset \varpi$ in Hausdorff distance;
		\item $T_k$ converges to an area-minimizing current mod$(q)$ for which we fix a mod$(q)$ representative $T$, with $\spt(T)\cap\Bbf_{1/2}=\Sbf\cap\Bbf_{1/2}$, and $\|T_k\|$ converges to $\|T\|$ in the sense of varifolds;
		\item $x_k$ converges to $x\in \bar{\Bbf}_{1-\eta/8}\setminus \Bbf_{\rho/4} (V)$ with $\Theta(T,x)\geq \frac{q}{2}$;
	\end{itemize}
	Once again, we have two cases; either $\Sbf$ is supported in an $m$-dimensional plane, or not.
	
	{\bf $\mathbf{S}$ is not supported in an $m$-dimensional plane.} In this case, $T$ is a mod$(q)$ area-minimizing cone with an $(m-1)$-dimensional spine. By the classification theorem for these cones, we know that $\Theta(T,y)= \frac{q}{2}$ for all $y\in V$, while $\Theta (T,y) < \frac{q}{2}$ for all $y\not \in V$. Since $x\in \bar{\Bbf}_{1/8}\setminus \Bbf_{1/16} (V)$, $x$ cannot belong to the spine $V$, on the other hand it must belong to it because $\Theta (T,x)\geq \frac{q}{2}$, yielding a contradiction.
	
	{\bf $\mathbf{S}$ is supported in an $m$-dimensional plane $\pi$.} In this case, our hypotheses imply that $T_k\to \frac{q}{2}\llbracket \pi\rrbracket$ in $\Bbf_1$ for an $m$-dimensional plane $\pi$ (in particular, $q$ must be even). One may now proceed via an analogous compactness argument to that in Case 1 above, using the Lipschitz approximation of \cite{Almgren_regularity}*{Corollary 3.11} for stationary integral varifolds (or that for mod$(q)$ area-minimizing currents in \cite{DLHMS}*{Theorem 15.1}, but the former works more generally). Thus, we deduce the existence of a $Q$-valued blow-up map $\bar{f}$ on $\Bbf_{1-\eta/16}(0,\pi)$ whose graph is supported on a non-planar open book $\Sbf'$ with an $(m-1)$-dimensional spine, which in fact is the plane $V$. However, the assumption that $\Theta(T_k,x_k) \geq Q$ for some point $x_k \in \Bbf_{1-\eta/8}\setminus\Bbf_{\rho/4}(V)$ for each $k$, combined with a suitable persistence of $Q$-points argument (for example, based on the Hardt--Simon inequality, which works without any minimizing assumption, or alternatively \cite{DLHMS}*{Theorem 23.1} which does), gives that there must be a $Q$-point of $\bar{f}$ in $\overline{B}_{1-\eta/8}\setminus \mathbf{B}_{\rho/4}(V)$, which contradicts the structure of $\bar{f}$. This completes the proof. Note that for this part one can form an argument which does not require $T_k$ to be area-minimizing mod$(q)$; they merely need to be stationary integral varifolds.
\end{proof}

\section{Graphical approximations}\label{s:approx}
This section is dedicated towards approximating $T$ effectively by multi-valued graphs over the half-planes $\Hbf_i$ in an open book $\Sbf$, away from a neighbourhood of the spine $V(\Sbf)$. This procedure is analogous to the ones of Simon \cite{Simon_cylindrical} and Wickramasekera \cite{W14_annals} in the non-collapsed and collapsed cases, respectively. Here, we follow the equivalent constructions and notations of these procedures in our previous work \cite{DMS} to match the notation.

\subsection{Pruning and layer subdivision}
We begin with the following \emph{pruning lemma}, which is an analogue of the one in \cite{DMS}*{Section 8}. A key purpose of this lemma is to throw away some of the planes in a given open book, yielding a new open book with the same maximal pairwise angle between the planes that are the extensions of the half-planes across the common spine, but for which the two-sided excess of $T$ to this new book is small relative to the minimal pairwise angle between the half-planes of the new open book. Since we are now dealing with half-planes in place of planes, while the maximal separation we consider is for the planes formed from extending them across their common axis, we repeat the proofs here. Throughout this section, for a given open book $\Sbf\in\Bscr^q$, we let $\boldsymbol{\zeta}(\Sbf)$ be as in \eqref{e:mu}.

\begin{lemma}[Pruning Lemma]\label{l:pruning}
	Let $2\leq N\leq q$, $D>0$ and $0<\delta\leq 1$. Let $\Gamma\coloneqq \delta^{2-N}(N-1)!$ and $\eps\coloneqq (\Gamma + 1)^{-1}\delta$. If $\Sbf = \Hbf_1 \cup \cdots \cup \Hbf_N \in \Bscr^q$ with
	\begin{equation}\label{e:pruning-hyp}
		D\leq \eps \boldsymbol{\zeta}(\Sbf),
	\end{equation}
	then there exists a subcollection $I\subset\{1,\dots, N\}$ with $\# I \geq 2$ satisfying the following properties for the planes $\pi_i$ that are the extensions of $\Hbf_i$ by reflecting across $V(\Sbf)$:
	\begin{align}
		\max_j\min_{i\in I}\dist(\Hbf_i\cap\Bbf_1,\Hbf_j\cap\Bbf_1) &\leq \Gamma D \label{e:pruning-1} \\
		D+\max_j\min_{i\in I}\dist(\Hbf_i\cap\Bbf_1,\Hbf_j\cap\Bbf_1) &\leq\delta \min_{i,j\in I: i<j}\dist(\Hbf_i\cap\Bbf_1,\Hbf_j\cap\Bbf_1) \label{e:pruning-2} \\
		\max_{i,j\in I:i<j} \dist(\pi_i\cap\Bbf_1,\pi_j\cap\Bbf_1) &= \boldsymbol{\zeta}(\Sbf).\label{e:pruning-3}
	\end{align}
\end{lemma}

\begin{proof}
 We use a variant of the algorithm explained in \cite{DMS}*{Lemma 8.2}, which iteratively constructs $\{1, \ldots, N\} = I (0)\subset I (1) \subset \cdots \subset I (s) = I$ by removing one element from each $I (k)$ until some stopping step $s$. First of all we adopt the same rule as in \cite{DMS}*{Lemma 8.1} to decide when we stop; namely we stop at the first $s$ such that 
 \begin{equation}\label{e:stop} 
 D + \max_j \min_{i\in I (s)} \dist (\Hbf_i\cap \Bbf_1, \Hbf_j\cap \Bbf_1) \leq \delta \min_{i,j\in I(s), i<j} \dist (\Hbf_i\cap \Bbf_1, \Hbf_j\cap \Bbf_1)\, .
\end{equation}
Then, analogously to \cite{DMS}*{Lemma 8.2}, at each step $k$ before we stop, the set $I (k+1)$ is obtained from $I(k)$ by removing an index $\ell$ such that 
\begin{align}
\min_{i \in I (k)} \dist (\Hbf_i\cap \Bbf_1, \Hbf_\ell \cap \Bbf_1)
&= \min_{i, j\in I (k), i<j} \dist (\Hbf_i\cap \Bbf_1, \Hbf_j \cap \Bbf_1)\label{e:min-OK}\\
\max_{i,j \in I (k)\setminus \{\ell\}} \dist (\pi_i \cap \Bbf_1, \pi_j \cap \Bbf_1)
&= \max_{i,j\in I (k)} \dist (\pi_i \cap \Bbf_1, \pi_j \cap \Bbf_1)\, .\label{e:max-OK}
\end{align}
By following the computations in \cite{DMS}*{Lemma 8.2}, we can see that, as long as the existence of such an $\ell$ is guaranteed, namely we can perform the task of the algorithm, the inequality 
\begin{equation}\label{e:pruning-4}
\max_j\min_{i\in I (k)}\dist(\Hbf_i\cap\Bbf_1,\Hbf_j\cap\Bbf_1) \leq \Gamma D 
\end{equation}
is also guaranteed for all $k\in \{0,\dots,s\}$. In particular, it follows directly that, provided we can keep choosing $\ell$ satisfying \eqref{e:min-OK} and \eqref{e:max-OK}, at the stopping step the set $I = I(s)$ certainly satisfies all the inequalities \eqref{e:pruning-1}, \eqref{e:pruning-2}, \eqref{e:pruning-3}, and $\# I (s)\geq 2$. For the latter conclusion, observe that, either $s=0$, and hence obviously $\# I (s) = N \geq 2$, or $s>0$, and the existence of the discarded index $\ell$ at step $s-1$ guarantees that $I (s-1) \geq 3$. 

For the existence of $\ell$ at some step $k<s$, observe first that, under our assumption, the stopping condition \eqref{e:stop} is not fulfilled at that particular step and therefore
\[
\min_{i,j\in I (k), i<j} \dist (\Hbf_i\cap \Bbf_1, \Hbf_j\cap \Bbf_1) < \delta^{-1}\bigl(D + \max_j \min_{i\in I (k)} \dist (\Hbf_i\cap \Bbf_1, \Hbf_j\cap \Bbf_1)\bigr)\, .
\]
Since however \eqref{e:pruning-4} is also valid at step $k$ we infer
\begin{equation}\label{e:key-upper-bound}
\min_{i,j\in I (k), i<j} \dist (\Hbf_i\cap \Bbf_1, \Hbf_j\cap \Bbf_1)
< (\Gamma + 1) \delta^{-1} D\, .
\end{equation}
Pick a pair $\{\ell_1, \ell_2\}\subset I(k)$ which maximizes $\dist (\pi_i\cap \Bbf_1, \pi_j \cap \Bbf_1)$ for $i,j\in I(k)$. If this pair does not minimize
$\dist (\Hbf_i\cap \Bbf_1, \Hbf_j\cap \Bbf_1)$, then the existence of $\ell$ satisfying \eqref{e:min-OK} and \eqref{e:max-OK} is obvious. If the pair does minimize $\dist (\Hbf_i\cap \Bbf_1, \Hbf_j\cap \Bbf_1)$, then we can use \eqref{e:key-upper-bound} and \eqref{e:pruning-hyp} to get
\begin{align*}
\boldsymbol{\zeta}(\Sbf) &= \max_{i,j\in I (k)} \dist (\pi_i\cap \Bbf_1, \pi_j \cap \Bbf_1)
\leq \dist (\Hbf_{\ell_1}\cap \Bbf_1, \Hbf_{\ell_2}\cap \Bbf_1)\\
&< (\Gamma + 1) \delta^{-1} D \leq (\Gamma+1) \delta^{-1} \varepsilon \boldsymbol{\zeta} (\mathbf{S}) \leq \boldsymbol{\zeta} (\mathbf{S})\, ,
\end{align*}
which gives a contradiction. 
\end{proof}

A consequence of iteratively applying Lemma \ref{l:pruning} with $D$ being the minimal separation between the planes in the previous open book, is the following \emph{layer subdivision lemma}, which is the analogue of \cite{DMS}*{Lemma 8.3}. As we inductively move from one ``layer" to the next, we remove those half-planes $\Hbf_i$ for which the following property holds: the distance between $\Hbf_i$ and the other half-planes $\Hbf_j$, $j\neq i$, is comparable to the minimal pairwise separation between the half-planes of the former open book. We stop the iterative procedure as soon as we arrive at an open book for which the minimal pairwise separation between the half-planes is comparable to the maximal separation between the planes $\pi_i$ which are the extensions to its half-planes across their common spine.

\begin{lemma}[Layer subdivision]\label{l:layers}
	For all integers $2\leq N\leq q$ and every $0 < \delta \leq 1$, there exists $\eta=\eta(\delta,N)>0$ such that the following holds. Let $\Sbf = \Hbf_1 \cup \cdots \cup \Hbf_N \in \Bscr^q$ and let $\pi_i$ be the extensions of $\Hbf_i$ across $V(\Sbf)$. Then there exists $\kappa\in \N\cup \{0\}$ and subcollections $\{1,\dots,N\} = I(0) \supsetneq I(1) \supsetneq \cdots I(\kappa)$ with $\# I(\kappa) \geq 2$, for which the numbers
	\begin{align*}
		m(k) &\coloneqq \min_{i, j \in I(k) : i < j} \dist(\Hbf_i\cap\Bbf_1, \Hbf_j\cap\Bbf_1) \\
		d(k) &\coloneqq \max_{i\in I(0)}\min_{j\in I(k)} \dist(\Hbf_i\cap\Bbf_1,\Hbf_j\cap\Bbf_j) \\
		M(k) &\coloneqq \max_{i,j\in I(k): i < j} \dist(\pi_i\cap\Bbf_1,\pi_j\cap\Bbf_1)
	\end{align*}
	satisfy
	\begin{itemize}
		\item[(i)] $M(\kappa) = M(0)$,
		\item[(ii)] $\eta M(\kappa) \leq m(\kappa)$,
		\item[(iii)] $d(k) \leq \delta m(k)$ and $\eta d(k) \leq m(k-1)$ for every $1\leq k\leq \kappa$,
		\item[(iv)] $m(k-1) \leq\delta m(k)$ for each $1\leq k\leq \kappa$.
	\end{itemize}
\end{lemma}

\begin{proof}
    Fix $\delta\in (0,1]$. Let $\Gamma$ be as in Lemma \ref{l:pruning}, corresponding to $\delta/N$ in place of $\delta$, and let $\eps=(\Gamma +1)^{-1}\delta N^{-1}$ be the associated value of $\eps$ therein. 
    Fix $\eta \in (0,\eps]$.

    If $\eta M(0) \leq m(0)$, set $\kappa=0$; the conclusions (iii) and (iv) are vacuous in this case, and (i), (ii) trivially hold. Otherwise, we inductively produce nested subcollections $I(s)\subset I(s-1)\subset\dots\subset I(0)$, by applying Lemma \ref{l:pruning} to the open book formed from the indices in $I(s)$ with $D=m(s-1)$ and $\delta/N$ in place of $\delta$ as above; keep going as long as $\eta M(s-1) > m(s-1)$. Let $\kappa$ be the final index when the inductive procedure terminates; clearly $\# I(\kappa) \geq 2$. Then, by construction, clearly (ii) holds. The conclusion \eqref{e:pruning-3} of Lemma 4.1 guarantees (i). The remaining conclusions follow entirely analogously to that in the proof of \cite{DMS}*{Lemma 8.3}, with planes replaced by half-planes. We refer the reader to the arguments therein for the details.
\end{proof}

\subsection{Crude graphical approximations}
We are now in a position to approximate $T$ by multi-valued graphs over the half-planes in $\Sbf$, outside of a neighbourhood of $V(\Sbf)$. We begin with some crude approximation results. Given $\Sbf= \Hbf_1\cup\cdots\cup\Hbf_N \in \Bscr^q$, recall the notation
\[
\boldsymbol{\sigma}(\Sbf)\coloneqq \min_{i<j} \dist(\Hbf_i\cap\Bbf_1,\Hbf_j\cap\Bbf_1).
\]

We begin with the following crude splitting lemma.
\begin{lemma}[Crude splitting]\label{l:crude}
	Let $q,m,n,\bar{n}\in \N$, let $\rho,\eta>0$ and let $Q=\frac{q}{2}$. There exist constants $\delta=\delta(q,m,n,\bar n,\rho,\eta)>0$ and $\varrho=\varrho(q,m,n,\bar n,\rho,\eta)>0$ such that the following holds. Let $T,\Sigma$ and $\Abf$ be as in Assumption \ref{a:main} with $\|T\|(\Bbf_4)\leq (Q+\frac{1}{4})\omega_m 4^m$. Suppose that $2\leq N\leq q$ and that $\Sbf=\Hbf_1\cup\cdots\cup\Hbf_N\in\Bscr^q$ with $\cap_i \Hbf_i=V$ and
	\begin{equation}\label{e:crude-hyp}
		\int_{\Bbf_4\setminus \Bbf_{\rho}(V)}\dist^2(p,\Sbf)\,d\|T\|(p) + \Abf^2\leq \delta^2\boldsymbol{\sigma}(\Sbf)^2 \eqqcolon \delta^2\sigma^2.
	\end{equation}
	Then the following properties hold:
	\begin{itemize}
		\item[(a)] The sets $W_i\coloneqq (\Bbf_4\setminus \overline{\Bbf}_{\rho}(V))\cap\{\dist(\,\cdot\,,\Hbf_i)<\varrho\sigma\}$ are pairwise disjoint;
		\item[(b)] $\spt(T)\cap\Bbf_{4-\eta}\setminus \overline{\Bbf}_{\rho+\eta}(V)\subset \bigcup_i W_i$;
        \item[(c)] $T\mres \Bbf_{4-\eta/2}\setminus \overline{\Bbf}_{\rho+\eta/2}(V)$ identifies with an area-minimizing integral current.
	\end{itemize}
\end{lemma}

For an $m$-dimensional half-plane $\Hbf$, we abuse notation slightly by letting $\mathbf{p}_{\Hbf}$, and $\Hbf^\perp$, respectively denote the orthogonal projection to, and the orthogonal complement to, the $m$-dimensional plane that is the extension of $\Hbf$. We have the following crude approximation result, which is a consequence of Lemma \ref{l:crude} and Almgren's strong Lipschitz approximation theorem (see e.g. \cite{DLS14Lp}*{Theorem 1.4}).
\begin{proposition}\label{p:crude-approx}
	Let $\rho,\eta>0$ and let $q,\delta,\varrho, T,\Sigma, \Abf,W_i$ and $\Sbf$ be as in Lemma \ref{l:crude}. For each $i\in \{1,\dots,N\}$, let $\Omega_i\coloneqq (\Bbf_{4-2\eta}\cap\Hbf_i)\setminus \overline{\Bbf}_{\rho+\eta}(V)$, let $\boldsymbol{\Omega}_i\coloneqq \Bbf_{4-\eta}\cap\mathbf{p}_{\Hbf_i}^{-1}(\Omega_i)$ and let $T_i\coloneqq T\mres(W_i\cap\boldsymbol{\Omega}_i)$. Define
	\[
	E_i\coloneqq \int_{\Bbf_{4}\setminus \Bbf_{\rho}(V)} \dist^2(p,\Hbf_i)\, d\|T_i\|(p).
	\]
    Then there exists non-negative integers $Q_1,\dots,Q_N$ with $\sum_i Q_i \leq q$ for which the following properties hold, for some $\gamma=\gamma(q,m,n,\bar n)>0$ and $C=C(q,m,n,\bar n,\rho,\eta)>0$:
	\begin{itemize}
		\item[(a)] $\partial T_i \mres \boldsymbol\Omega_i = 0$ (as an integral current);
		\item[(b)] For a suitable orientation of $\Hbf_i$, $(\mathbf{p}_{\Hbf_i})_\sharp T_i = Q_i \llbracket \Omega_i\rrbracket$;
		\item[(c)] For each $q\in \spt(T_i)\cap\boldsymbol\Omega_i$ we have $\dist^2(p,\Hbf_i) \leq C (E_i +\Abf^2)$;
		\item[(d)] For each $i$ with $Q_i \geq 1$, there exist Lipschitz multi-valued maps $u_i:\Omega_i \to \Acal_{Q_i}(\Hbf_i^\perp)$ and closed sets $K_i\subset \Omega_i$ with $\gr(u_i)\subset\Sigma$, $T_i\mres\mathbf{p}_{\Hbf_i}^{-1}(K_i)=\Gbf_{u_i}\mres\mathbf{p}_{\Hbf_i}^{-1}(K_i)$, for which
		\begin{align}
			\|u_i\|^2_{L^\infty} + \|Du_i\|^2_{L^2} &\leq C(E_i + \Abf^2) \\
			\Lip(u_i) &\leq C(E_i + \Abf^2)^\gamma \\
			|\Omega_i\setminus K_i| + \|T\|(\boldsymbol\Omega_i\setminus \mathbf{p}_{\Hbf_i}^{-1}(K_i)) &\leq C(E_i + \Abf^2)^{1+\gamma};
		\end{align}
	\item[(e)] $Q_i=0$ if and only if $T_i=0$;
	\item[(f)] If we additionally have the reverse excess estimate
	\[
		\int_{\Bbf_{4-2\eta}\cap\Sbf\setminus\Bbf_{\rho+2\eta}(V)} \dist^2(p,\spt(T)) d\Hcal^m(p) \leq \delta^2\sigma^2,
	\]
	then $Q_i\geq 1$ for each $i$.
	\end{itemize}
\end{proposition}
\begin{remark}
    We will usually take $\eta = \rho = \rho_*$ where $\rho_*$ is a small constant depending only on $m,q$ as detailed in Lemma \ref{l:matching-Q} (c.f. \cite{DMS}).
\end{remark}
Note that unlike in \cite{DMS}, we need not isolate the version of Proposition \ref{p:crude-approx} in the case where $\Sbf = \Hbf_1\cup(-\Hbf_1)$ consists of a single plane formed from a half-plane and its reflection, since here, there is still a canonical choice of spine and $\boldsymbol{\sigma}(\Sbf)$ is well-defined in this case.

The validity of Lemma \ref{l:crude} follows from Theorem \ref{t:no-gaps}, Proposition \ref{p:integrality} in a similar manner to that seen in \cite{DMS}*{Lemma 8.5}. However, some aspects of the compactness argument in the proof contained in \cite{DMS} rely on closeness to balanced superpositions of planes meeting in an $(m-2)$-dimensional spine, in place of open books herein, which affects the behavior of the limiting object at the spine of the cone. Thus, we repeat the relevant details here for clarity. Meanwhile, the conclusions of Proposition \ref{p:crude-approx} are an immediate consequence of Lemma \ref{l:crude}(c) and the proof of the analogous statement \cite{DMS}*{Proposition 8.6}; notice that taking half-planes in place of planes (but with $\sigma$ defined for the corresponding full planes) does not affect any of the arguments leading to these conclusions.

However, let us first state the following lemma, which is the analogue of \cite{DMS}*{Lemma 8.8}, which gives sufficient conditions in order to guarantee that $\sum Q_i = q$ in the conclusions of Proposition \ref{p:crude-approx}.

\begin{lemma}\label{l:matching-Q}
There exists $\rho_* = \rho_*(m,q)>0$ such that the following is true. Suppose that $T$ satisfies the assumptions of Lemma \ref{l:crude} with $\rho\leq \rho_*$, let $Q=\frac{q}{2}$, and suppose that in addition we have either:
\begin{itemize}
    \item[(a)] $\{\Theta^q (T, \cdot) \geq Q\} \cap \Bbf_{\varepsilon} \neq \emptyset$ for a sufficiently small $\varepsilon = \varepsilon (Q, m,n, \bar n)$; or
    \item[(b)] for some $C_*>0$, $\rho_*$ is allowed to depend on $C_*$ also, and there is a closed set $\Omega\subset \Bbf_4$ with non-empty interior that is invariant under rotation around $V$ (see Definition \ref{d:rot-inv}), for which $\|T\| (\Omega) \geq (q-\frac{1}{2}) \mathcal{H}^m (\Hbf_1 \cap \Omega)$ and $\mathcal{H}^m(\Hbf_1\cap\Omega)\geq C_*$.
\end{itemize}
Then, if $Q_i$ is as in Proposition \ref{p:crude-approx} and $\delta$ is sufficiently small (with the same dependencies as before), we have $\sum_i Q_i =q$. 
\end{lemma}

We will in fact only need to apply Lemma \ref{l:matching-Q}(b) to very specific choices of $\Omega$ for which the property $\mathcal{H}^m(\Hbf_1\cap\Omega)\geq C_*$ will hold for an appropriate choice of $C_* = C_*(q,m)$, meaning that a choice of $\rho_* = \rho_*(q,m)$ can be made so that in the alternative (b) of Lemma \ref{l:matching-Q}, we no longer need to make any assumptions involving $C^*$.

\begin{proof}[Proof of Lemma \ref{l:crude}]
    Fix $\rho,\eta > 0$. Moreover, fix $\varrho$ small enough (depending on $m,n, \rho$) to ensure the validity of the conclusion (a). We proceed to argue by contradiction to verify (b) and (c) for a choice of $\delta$ sufficiently small. Suppose that we have a sequence $T_k$, $\Sigma_k$ and $\Abf_k$ as in Assumption \ref{a:main}, together with open books $\Sbf_k = \Hbf^k_1\cup\cdots \cup\Hbf^k_{N(k)} \in \Bscr^q$ satisfying
    \begin{itemize}
        \item[(i)] $\|T_k\|(\Bbf_4) \leq (Q+\frac{1}{4})\omega_m 4^m,$
        \item[(ii)] $\Hbf^k_i\cap \Hbf^k_j=V_k$ for each $i<j\leq N(k)\leq q,$
        \item[(iii)] for $\sigma_k\coloneqq \boldsymbol{\sigma}(\Sbf_k)$ and
        \[
            \Ebf_k \coloneqq \int_{\Bbf_4\setminus \Bbf_{\rho}(V)}\dist^2(p,\Sbf_k)d\|T_k\|(p),
        \]
        we have $\frac{\Ebf_k + \Abf_k^2}{\sigma_k^2} \to 0$,
        \item[(iv)] there exist $p_k \in \spt(T_k)\cap (\Bbf_{4-\eta}\setminus \Bbf_{\rho+\eta}(V_k))$ with $\dist(p_k,\Hbf_i) \geq \varrho \sigma_k$.
    \end{itemize}
    Up to rotating and extracting a subsequence, we may thus in addition assume the following:
    \begin{itemize}
        \item[(v)] $V_k \equiv V$ is a fixed $(m-1)$-dimensional subspace and $N(k) \equiv N \leq Q$ is a fixed integer;
        \item[(vi)] $\Sbf_k$ converges locally in Hausdorff distance, to $\Sbf = \Hbf_1\cup\cdots \cup\Hbf_{N'}\in \Bscr^q$ with $1\leq N'\leq N$ and $\Hbf_i \cap \Hbf_j = V$ for all $i<j$ (note that if $N'=1$ then this last condition is vacuous); 
        \item[(v)] $T_k$ converges in the mod$(q)$ flat topology to an area-minimizing representative mod$(q)$ in $\Bbf_4$, which we denote by $T$ and which satisfies $\partial T=0\,\text{mod}(q)$;
        \item[(vi)] $\spt(T) \cap (\Bbf_4\setminus \Bbf_\rho(V)) \subset \Sbf$.
    \end{itemize}
    Now, applying Theorem \ref{t:no-gaps} to $(\iota_{0,4})_\sharp(T_k)$, and in turn applying Proposition \ref{p:integrality} to $(\iota_{0,4})_\sharp(T_k)$ and $U = \Bbf_{1-\eta/8}\setminus \overline{\Bbf}_{(\rho+\eta/2)/4}(V)$, for $k$ sufficiently large we may identify each $T_k$ with an integral current in $\Bbf_{4-\eta/2}\setminus\overline{\Bbf}_{\rho+\eta/2}(V)$ (not relabelled); this establishes (c). Moreover, (iii) above tells us that $\spt(T)\cap (\Bbf_{4}\setminus\overline{\Bbf}_{\rho}(V)) \subset \Sbf$. Applying the mod$(q)$ version of the Constancy Theorem (see for instance \cite{DLHMS}*{Lemma 7.4}), we may conclude that there exist integers $\tilde{Q}_i$ such that
    $$T\res (\Bbf_{4}\setminus\overline{\Bbf}_{\rho}(V)) = \sum^N_{i=1}\tilde{Q}_i\llbracket \Hbf_i\rrbracket \res (\Bbf_{4}\setminus\overline{\Bbf}_{\rho}(V))$$
    with $-Q\leq\tilde{Q}_i\leq Q$. Up to changing the orientation of $\Bbf_i$ we can assume that all the $\tilde{Q}_i$ are non-negative. The remainder of the proof of (b) then follows via the exact same reasoning as that of \cite{DMS}*{Lemma 8.5}, and so we omit it here and refer the reader to the argument therein.
\end{proof}

We have already justified Proposition \ref{p:crude-approx}, so we move onto Lemma \ref{l:matching-Q}.

\begin{proof}[Proof of Lemma \ref{l:matching-Q}]
    It suffices to demonstrate that in the compactness argument in the contradiction proof of Lemma \ref{l:crude}, the limiting mod$(q)$ representative $T$, which obeys
    \[
    T\res(\Bbf_{4}\setminus\overline{\Bbf}_{\rho}(V)) =\sum_{i=1}^N \tilde{Q}_i\llbracket \Hbf_i \rrbracket \res (\Bbf_{4}\setminus\overline{\Bbf}_{\rho}(V))
    \]
    satisfies $\frac{1}{2}\sum_i \tilde{Q}_i = Q$ for any $\rho \leq \rho_*(q,m)$ sufficiently small, when we suppose that one of the hypotheses (a) or (b) holds. 

    First, note that we have from monotonicity of mass ratios and the weak convergence that $\|T\|(\Bbf_3) \leq (Q+\frac{1}{4})\omega_m3^m$. For sufficiently small $\rho$, this evidently implies that we must have $\frac{1}{2}\sum_i \tilde{Q}_i\leq Q$. Thus, we just need to show that $\frac{1}{2}\sum_i\tilde{Q}_i\geq Q$ when we additionally suppose one of the hypotheses (a) or (b) hold.

    Note first that we can cover $\Bbf_3\cap V$ by $C\rho^{-(m-1)}$ balls of radius $\rho$; if we double the radius of each ball, we may then without loss of generality assume that they cover $\Bbf_3\cap \Bbf_\rho(V)$ also. But then the monotonicity formula for $T$ gives for any such ball $\Bbf_i$ in this cover,
    $$\|T\|(\Bbf_i) \leq C\rho^m\|T\|(\Bbf_{7/2}) \leq C(q,m)\rho^m$$
    and so
    \begin{equation}\label{e:spine-nhd}
        \|T\|(\Bbf_3\cap \Bbf_\rho(V))\leq C\rho^m\cdot \rho^{-(m-1)} = C\rho.
    \end{equation}

    Now consider the case (a). Taking a sequence $\eps_k\downarrow 0$ for the sequence of currents $T_k$ from the proof of Lemma \ref{l:crude}, we have a sequence of points $p_k\in \Bbf_{\eps_k}$ with $\Theta(T_k,p_k)\geq Q$. Upper semi-continuity of the density guarantees that the limiting current satisfies $\Theta(T,0)\geq Q$, so $\|T\|(\Bbf_3)\geq 3^m\cdot Q\omega_m$, and hence $\|T\|(\Bbf_3\setminus\overline{\Bbf}_\rho(V))\geq 3^m\cdot Q\omega_m - C\rho$. But then this directly implies that
    $$\frac{1}{2}\sum_i\tilde{Q}_i(3^m\omega_m-3^{m-1}\omega_1\rho)\geq 3^m\cdot Q\omega_m - C\rho$$
    and so since the $\tilde{Q}_i$ are non-negative integers, if $\rho \leq \rho_* = \rho_*(q,m)$ sufficiently small, this evidently implies that $\frac{1}{2}\sum_i\tilde{Q}_i\geq Q$ in this case.

    If instead (b) holds, then for $T_k$, $\Sbf_k$ as in the proof of Lemma 4.3, we have $\|T_k\|(\Omega_k)\geq (q-\frac{1}{2})\mathcal{H}^m(\Hbf^k_1\cap {\Omega_k})$ for each $k$ and closed sets $\Omega_k$ with non-empty interior that are invariant under rotation around $V$ (recall that we are assuming the latter is fixed by a rotation of coordinates). Note that we are assuming that $T_k$ and $T$ are representatives mod$(q)$. Again, recall that we can apply \cite{DLHMS}*{Proposition 5.2} to conclude the weak-$*$ convergence of the masses $\|T_k\|$ to $\|T\|$. Since $\mathcal{H}^m(\Hbf_1^k\cap \Omega_k)\geq C_*$, we may pass to a subsequence to ensure that $\mathcal{H}^m(\Hbf_1^k\cap \Omega_k)\to \tilde{C}_*\in  [C_*, 4^m\omega_m]$. Combining this with the structure of $T$ and the rotational invariance of $\Omega$ about $V$, we get
    $$(q-\tfrac{1}{2})\tilde{C}_* \leq \lim_{k\to\infty}\|T_k\|(\Omega_k) \leq \lim_{k\to\infty}\|T_k\|(\Omega_k\setminus\Bbf_\rho(V)) + C\rho \leq \sum^N_{i=1}\tilde{Q}_i \tilde{C}_* + C\rho.$$
    Note that in the second inequality we use a mass bound analogous to the one in \eqref{e:spine-nhd}, which one may observe still holds for $T_k$ in $\Omega_k\cap B_\rho (V)$ since $\Omega_k$ is closed and contained in $\Bbf_4$. Thus, as $\sum^N_{i=1}\tilde{Q}_i$ is an integer, provided $\rho \leq \rho_* = \rho_*(q,m,C_*)$ is sufficiently small, we get a contradiction and thus complete the proof.
\end{proof}

Before providing a more refined Lipschitz approximation result, we have the following useful lemma, which is a consequence of the preceding results in this section.

\begin{lemma}\label{l:planar-excess-zeta}
    Let $\bar\delta > 0$. Suppose that $T$, $\Sigma$ and $\Abf$ are as in Assumption \ref{a:main}, let $\Sbf=\Hbf_1\cup\cdots\cup \Hbf_N\in \Bscr^q$ for $N\geq 2$, and let $V=V(\Sbf)$. Then there exist $C=C(q,m,n,\bar{n})>0$, $\bar C=\bar C(q,m,n,\bar{n},\bar\delta)>0$ and $\eps_3=\eps_3(q,m,n,\bar{n},\bar\delta)>0$ such that the following holds. Suppose that
    \[
        \Abf^2 \leq \eps_3^2 \Ebb(T,\Sbf,\Bbf_1) \leq \eps_3^4 \Ebf^p(T,\Bbf_1).
    \]
    Then there exists $\Sbf'=\Hbf_{i_1}\cup\cdots\cup\Hbf_{i_k}\in \Bscr^q(0)$ for a subcollection $\{i_1,\dots,i_k\}\subset\{1,\dots,N\}$ with $k\geq 2$, such that
    \begin{itemize}
        \item[(a)] $C^{-1} \Ebf^p (T, \Bbf_1) \leq \boldsymbol{\zeta} (\Sbf)^2 = \boldsymbol{\zeta} (\Sbf')^2 \leq C \Ebf^p (T, \Bbf_1)$;
        \item[(b)] $\Ebb(T,\Sbf',\Bbf_1)\leq \bar C\Ebb(T,\Sbf,\Bbf_1)$;
        \item[(c)] $\dist^2(\Sbf\cap\Bbf_1,\Sbf'\cap\Bbf_1)\leq C\Ebb(T,\Sbf,\Bbf_1)$.
        \item[(d)] $\Abf^2 + \Ebb(T,\Sbf',\Bbf_1) \leq \bar\delta^2\boldsymbol{\sigma}(\Sbf')^2$.
    \end{itemize}
\end{lemma}

\begin{proof}
    Let us begin with the conclusion (a). In proving this, we will also show (d) as a byproduct. First of all, since
    \[
        \Ebf^p(T,\Bbf_1) \leq C\hat\Ebf(T,\Sbf,\Bbf_1) + C\boldsymbol{\zeta}(\Sbf)^2 \leq C\eps_3^2\Ebf^p(T,\Bbf_1) + C\boldsymbol{\zeta}(\Sbf)^2,
    \]
    for some constant $C=C(m,n)>0$, the bound
    \begin{equation}\label{e:planar-leq-max-sep}
        \Ebf^p(T,\Bbf_1) \leq \bar{C}\boldsymbol{\zeta}(\Sbf)^2
    \end{equation}
    follows immediately for $\bar{C}=\bar{C}(m,n)>0$, provided that $\eps_3$ is below a sufficiently small dimensional constant, which in turn yields
    \[
        \Ebb(T,\Sbf,\Bbf_1) \leq \bar{C}\eps_3^2\boldsymbol{\zeta}(\Sbf)^2.
    \]
    Now fix $\delta > 0$, to be determined later, and let $\Gamma= \Gamma(\delta,N)$ be as in the Pruning Lemma \ref{l:pruning}. Consequently, let $\eps_3 = \bar{C}^{-1/2} (1+\Gamma)^{-1/2}\delta$. Thus, letting $D\coloneqq \Ebb(T,\Sbf,\Bbf_1)^{1/2}$, this choice of $D$ satisfies the hypotheses of the Pruning Lemma. Hence, applying the lemma gives a subcollection $I=\{i_1,\dots,i_k\}\subset \{1,\dots,N\}$ satisfying the properties stated therein. Let $\Sbf'\coloneqq \Hbf_{i_1}\cup\cdots\cup\Hbf_{i_k}$ and let $\pi_{i_j}$ denote the $m$-dimensional planes that are the extensions of the half-planes $\Hbf_{i_j}$ by reflection across $V$. Since $\boldsymbol{\zeta}(\Sbf) = \boldsymbol{\zeta}(\Sbf')$, to conclude (a) it suffices to demonstrate that
    \begin{equation}\label{e:max-sep-planar-exc}
        \boldsymbol{\zeta}(\Sbf')^2 \leq C\Ebf^p(T,\Bbf_1),
    \end{equation}
    for $C=C(q,m,n,\bar n)>0$. First of all, given an arbitrary $m$-dimensional plane $\varpi$, observe that there must exist a plane $\pi_{i_j}$ corresponding to some half-plane $\Hbf_{i_j}$ in $\Sbf'$, which up to relabelling we may assume is $\pi_1$, with
    \[
        \dist(\pi_1\cap\Bbf_1,\varpi\cap\Bbf_1) \geq \frac{1}{2}\boldsymbol{\zeta}(\Sbf').
    \]
    It follows that there exists $p\in\pi_1$, $r=r(m,n)>0$ and $C=C(m,n)>0$ such that $\Bbf_r(p)\subset \Bbf_{3/4}\setminus \Bbf_{1/4}(V)$ and
    \begin{equation}\label{e:split}
        \dist(x,\varpi)\geq C^{-1} \boldsymbol{\zeta}(\Sbf') \qquad\forall x\in B_r(p,\pi_1);
    \end{equation}
    see the proof of (d) of \cite{DMS}*{Proposition 9.1} for the details. We will now proceed to pass this to an analogous estimate for $x\in \Bbf_r(p)\cap\spt(T)$.
    
    Now, for a choice of $\eps_3>0$ sufficiently small (with the claimed dependencies), we claim that we may apply Lemma \ref{l:crude} and Proposition \ref{p:crude-approx} to $T_{0,4}$ and $\Sbf'$, with $\eta=\rho=\frac{1}{32}$. Indeed, \eqref{e:pruning-2} tells us that
    \begin{equation}\label{e:pruning-exc-min-sep}
        \Ebb(T,\Sbf, \Bbf_1) + \max_j \min_{i\in I} \dist^2(\Hbf_i\cap\Bbf_1,\Hbf_j\cap\Bbf_1) \leq 2\delta^2\boldsymbol{\sigma}(\Sbf')^2.
    \end{equation}
    Thus, for $C=C(q,m)>0$ we have 
    \begin{align}
        \hat\Ebf(T,\Sbf',\Bbf_1) \leq 2
        \hat\Ebf(T,\Sbf,\Bbf_1) + C \max_j \min_{i\in I}\dist^2(\Hbf_i\cap\Bbf_1,\Hbf_j\cap\Bbf_1) \leq C\delta^2\boldsymbol{\sigma}(\Sbf')^2.\label{e:conical-excess-pruned}
    \end{align}
    On the other hand, since $\hat\Ebf(\Sbf',T,\Bbf_1) \leq \hat\Ebf(\Sbf,T,\Bbf_1)$, \eqref{e:pruning-exc-min-sep} gives
    \[
        \hat\Ebf(\Sbf',T,\Bbf_1) \leq 2\delta^2 \boldsymbol{\sigma}(\Sbf')^2.
    \]
    In summary, we have demonstrated that
    \[
        \Ebb(T,\Sbf',\Bbf_1) \leq C\delta^2 \boldsymbol\sigma(\Sbf')^2.
    \]
    On the other hand, the assumption $\Abf^2 \leq \eps_3^2 \Ebb(T,\Sbf,\Bbf_1)$ combined with \eqref{e:pruning-exc-min-sep} yields
    \[
        \Abf^2\leq 2\eps_3^2\delta^2\boldsymbol\sigma(\Sbf')^2.
    \]
    Thus, the hypothesis \eqref{e:crude-hyp} of Lemma \ref{l:crude} indeed holds with parameter $\bar\delta$ (previously denoted $\delta$ therein), for a suitably small choice of $\delta=\delta(q,m,n,\bar n,\bar\delta)>0$. This is exactly the conclusion (d). Recall that this in turn determines how small we must take $\eps_3$.

    In particular, by Proposition \ref{p:crude-approx}(f) and (c), for $\varrho>0$ as in Lemma \ref{l:crude} and a choice of $C=C(q,m,n,\bar{n})>0$, the current
    \[
        T'\coloneqq T\mres\Bbf_r(p)\cap\{\dist(\cdot,\pi_1)\leq\varrho\boldsymbol{\sigma}(\Sbf')\}
    \]
    is non-zero and satisfies
    \[
        \dist(x,\pi_1) \leq C\hat\Ebf(T,\Sbf',\Bbf_1)^{1/2} + C\Abf \leq C\eps_3^2\boldsymbol{\zeta}(\Sbf') \qquad \forall x\in\spt(T'),
    \]
    where the final estimate is a consequence of \eqref{e:planar-leq-max-sep}. Combining this with \eqref{e:split}, we deduce that, up to possibly further decreasing $\eps_3$ (still with the same dependencies), we have
    \[
        \dist(x,\varpi)\geq C^{-1}\boldsymbol\zeta(\Sbf') \qquad \forall x\in \spt(T'),
    \]
    as desired. We now square this and integrate with respect to $d\|T'\|$. When combined with the monotonicity formula for mass ratios (as $T'\neq 0$), this yields the desired conclusion \eqref{e:max-sep-planar-exc}.

    Now let us demonstrate the conclusion (b). Notice that \eqref{e:pruning-1} from the conclusions of the pruning lemma, together with the first inequality in \eqref{e:conical-excess-pruned}, in fact yields
    \[
        \hat\Ebf(T,\Sbf',\Bbf_1) \leq 2\hat\Ebf(T,\Sbf,\Bbf_1) + C\Gamma^2\Ebb(T,\Sbf,\Bbf_1).
    \]
    Once again combining this with the observation that $\hat\Ebf(\Sbf',T,\Bbf_1) \leq \hat\Ebf(\Sbf,T,\Bbf_1)$, we conclude (b). Meanwhile, conclusion (c) follows immediately from \eqref{e:pruning-1} in the pruning lemma.
\end{proof}

\subsection{Refined graphical approximations}
We are now in a position to carry out a more refined graphical approximation procedure, analogous to that in \cite{DMS}*{Section 8.5}. We begin with the following assumption that will be used throughout this section.

\begin{assumption}\label{a:refined}
    Let $q,n,\bar{n}\in \N$, $m\in \N_{\geq 2}$ and let $Q=\frac{q}{2}$. Suppose that $T$, $\Sigma$ and $\Abf$ satisfy Assumption \ref{a:main} with $T(\Bbf_4)\leq (Q+\frac{1}{4})\omega_m$. Let $\Sbf=\Hbf_1\cup\cdots\cup\Hbf_N\in \Bscr^q\setminus \Pscr$ with $2\leq N \leq q$ and let $V=V(\Sbf)$ denote the spine of $\Sbf$. Let $\pi_i$ denote the $m$-dimensional plane which is the extension of $\Hbf_i$ as before. For a sufficiently small choice of $\eps = \eps(q,m,n,\bar{n})>0$, whose choice will be determined in Assumption \ref{a:parameters} below, smaller than the $\eps$-threshold in Theorem \ref{t:no-gaps} for appropriate choices of $\varrho,\eta,\rho$ therein, suppose that we have the two-sided excess bound
    \begin{equation}\label{e:refined-emall-excess-and-A}
        \Ebb(T,\Sbf,\Bbf_4) + \Abf^2 \leq \eps^2\boldsymbol{\sigma}(\Sbf)^2.
    \end{equation}
\end{assumption}

\subsubsection{Whitney decomposition}
We begin by setting up a family of dyadic cubes in the spine $V$, which will in turn be used to define a Whitney decomposition towards $V$. The procedure is completely analogous to that in \cite{DMS}*{Section 8.5} and we adopt the notation and terminology therein. We recall the latter here for clarity. Let $L_0$ be the $(m-1)$-dimensional closed unit cube contained in $V$ of side length $\frac{2}{\sqrt{m-1}}$ centered at the origin. Let $R$ denote the ``punctured cylinder"
\[
    R\coloneqq \{p:\mathbf{p}_V(p)\in L_0, \quad 0<|\mathbf{p}_{V^\perp}(p)|\leq 1\}.
\]
For each $\ell\in \N$, let $\Gcal_\ell$ be the family of $(m-1)$-dimensional dyadic closed cubes that are formed by subdividing $L_0$ into $2^{\ell(m-1)}$ cubes of side length $\frac{2^{-(\ell-1)}}{\sqrt{m-1}}$ with mutually disjoint interiors; note that $\Gcal_{\ell+1}\subset \Gcal_\ell$. Let $\Gcal\coloneqq \bigcup_\ell \Gcal_\ell$. We will denote by $L$ the cubes in $\Gcal$, and if needed, we will write $\ell(L)$ for the integer $\ell$ such that $L\in\Gcal_\ell$. Given $L\in \Gcal_\ell$, we refer to the unique $L'\in \Gcal_{\ell-1}$ with $L\subset L'$ as the \emph{parent} of $L$, while $L$ is called a \emph{child} of $L'$. In general, any $H\in \Gcal$ with $L\subset H$ will be referred to as an \emph{ancestor} of $L$, and $L$ is then said to be a \emph{descendent} of $H$ (note that $L$ is therefore an ancestor and descendent of itself). Given $L\in\Gcal_\ell$, define
\[
    R(L)\coloneqq \{p: \mathbf{p}_V(p)\in L, \quad 2^{-\ell-1}\leq |\mathbf{p}_{V^\perp} (p)|\leq 2^{-\ell}\}.
\]
For $L\in \Gcal_\ell$, we denote its center by $y_L \in V$ and we let $\Bbf(L)$ denote the ball $\Bbf_{2^{2-\ell(L)}}(y_L)$ and we let $\Bbf^h(L)\coloneqq \Bbf(L)\setminus \Bbf_{\rho_* 2^{-\ell(L)}}(V)$, where $\rho_*$ is the constant from Lemma \ref{l:matching-Q}. Moreover, given $\lambda \in [1,\frac{3}{2}]$ and $L\in \Gcal_\ell$, we will denote by $\lambda L$ the $(m-2)$-dimensional subcube of $V$ that is concentric cube to $L$ but with side length $\frac{\lambda 2^{-(\ell-1)}}{\sqrt{m-1}}$, while $\lambda R(L)$ is defined by
\[
    \lambda R(L) \coloneqq \{p:\mathbf{p}_V(p)\in \lambda L, \quad \lambda^{-1}2^{-(\ell+1)}\leq |\mathbf{p}_{V^\perp}(p)|\leq\lambda 2^{-\ell}\}.
\]
For each half-plane $\Hbf_i$, we in turn let
\[
    L_i\coloneqq R(L)\cap \Hbf_i, \qquad \lambda L_i\coloneqq \lambda R(L)\cap \Hbf_i.
\]
At the risk of abusing terminology, when referring to the interior of $L_i$, we implicitly mean the relative interior within $V$. 

Indeed the above construction yields a Whitney decomposition for the collection $\{R(L):L\in \Gcal\}$:
\begin{lemma}[\cite{DMS}*{Lemma 8.11}]
For $\Gcal$ as constructed above, the following properties hold.
\begin{itemize}
\item[(i)] Given any pair of distinct $L, L'\in \mathcal{G}$ the interiors of $R(L)$ and $R(L')$ are pairwise disjoint and $R(L)\cap R(L') \neq \emptyset$ if and only if $L\cap L' \neq \emptyset$ and $|\ell(L)-\ell(L')|\leq 1$, while the interiors of $L$ and $L'$ are disjoint if $\ell (L) \leq \ell (L')$ and $L'$ is not an ancestor of $L$.
\item[(ii)] The union of $R(L)$ ranging over all $L\in\mathcal{G}$ is the whole set $R$.
\item[(iii)] The diameters of the sets $L$, $R (L)$, $\lambda L$, $\lambda R (L)$, $L_i$, $\lambda L_i$, and $\mathbf{B}^h (L)$ are all comparable to $2^{-\ell (L)}$ and, with the exception of $L,\lambda L$, all comparable to the distance between an arbitrarily element within them and $V$; more precisely, any such diameter and distance is bounded above by $C 2^{-\ell (L)}$ and bounded below by $C^{-1} 2^{-\ell (L)}$ for some constant $C$ which depends only on $m$, $n$ and $q$.
\item[(iv)] There is a constant $C = C(m,n,q)$ such that, if {$\Bbf^h (L) \cap \Bbf^h (L') \neq \emptyset$, then $|{\ell (L)} - \ell (L')|\leq C$ and $\dist (L, L') \leq C 2^{-\ell (L)}$. In particular, for every $L\in \mathcal{G}$, the subset of $L'\in \mathcal{G}$ for which $\Bbf^h (L)$ and $\Bbf^h (L')$} have nonempty intersection is bounded by a constant. 
\item[(v)] $\sum_{L\in \mathcal{G}_\ell} \mathcal{H}^{m-1} (L) = C(m)$ for any $\ell$ and therefore, {for any $\kappa>0$},
\begin{equation}\label{e:geometric}
\sum_{L\in \mathcal{G}} 2^{-(m-1+\kappa) \ell (L)} \leq C (\kappa, m)\, .
\end{equation}
\end{itemize}
\end{lemma}
The proof of this is elementary and so we leave the details to the reader.

\subsubsection{Layers and selection of parameters}
In order to build more refined graphical approximations for our final blow-up procedure, we proceed as follows. Let $\bar\delta\in (0,1]$ (determined below in Assumption \ref{a:parameters}). For $\Sbf$ as in Assumption \ref{a:refined}, apply Lemma \ref{l:layers} to produce a nested family of sub-cones $\Sbf=\Sbf_0\supset \Sbf_1\supset\cdots\supset \Sbf_\kappa$, where $\Sbf_k\coloneqq \bigcup_{j\in I(k)} \Hbf_j$. Now we define $\bar\kappa$ as follows, distinguishing between two possibilities:
\begin{itemize}
    \item[(a)] if $\max_{i,j\in I(\kappa):i<j} \dist(\pi_i\cap\Bbf_1,\pi_j\cap\Bbf_1)<\bar\delta$, let $\Sbf_{\kappa+1}\in \Bscr^q$ consist of a single plane $\pi_{i_0}$ for $i_0\in I(\kappa)$ and let $\bar\kappa=\kappa+1$;
    \item[(b)] if $\max_{i,j\in I(\kappa):i<j} \dist(\pi_i\cap\Bbf_1,\pi_j\cap\Bbf_1)\geq \bar\delta$, let $\bar\kappa=\kappa$.
\end{itemize}

Let us now outline our selection of parameters herein.
\begin{assumption}[Hierarchy of parameters]\label{a:parameters}
    Let $\delta^*=\delta^*(q,m,n,\bar{n})>0$ be the minimum of the parameters $\delta$ in Lemma \ref{l:crude} and Lemma \ref{l:matching-Q} applied to all sub-cones $\Sbf'\subset \Sbf$. Given a small constant $c=c(q,m,n,\bar{n})>0$ (to be determined in Lemma \ref{l:regions}), fix $\tau=\tau(q,m,n,\bar{n})\in (0,c\delta^*]$ and consequently fix $\bar\delta=\bar\delta(q,m,n,\bar n) \in (0,c\tau]$. Finally, fix $\eps=\eps(q,m,n,\bar n,\delta^*,\bar\delta)\in (0,c\bar\delta]$.
\end{assumption}

\subsubsection{Regions and local approximations}
We are now in a position to construct sub-regions of $R$ which will determine which sub-cone $\Sbf_k$ we locally construct graphical approximations for $T$ over. The set up is entirely analogous to that in \cite{DMS}*{Section 8.5.3}, but we re-introduce it here for the purpose of clarity.

Given $L\in \Gcal$ and $k\in \{0,\dots,\bar\kappa\}$, let
\[
    \Ebf(L,k)\coloneqq 2^{(m+2)\ell(L)}\int_{\Bbf^h(L)}\dist^2(p,\Sbf_k)d\|T\|(p),
\]
and
\[
    \mathbf{s}(k)\coloneqq \min_{i,j\in I(k):i<j} \dist(\Hbf_i\cap\Bbf_1,\Hbf_j\cap\Bbf_1).
\]
\begin{definition}\label{d:regions}
Let $L\in \mathcal{G}$. We say that:
\begin{itemize}
\item[(i)] $L$ is an \emph{outer cube} if $\mathbf{E} (L^\prime,0) \leq \tau^2 \mathbf{s} (0)^2$
for every ancestor $L^\prime$ of $L$ (including $L$ itself).
\item[(ii)] $L$ is a \emph{central cube} if it is not an outer cube and if 
$\min_k \mathbf{E} (L^\prime, k)/{\mathbf{s}(k)^2} \leq \tau^2$
for every ancestor $L'$ of $L$ (including $L$).
\item[(iii)] $L$ is an \emph{inner cube} if it is neither an outer nor a central cube, but its parent is an outer or a central cube. 
\end{itemize}
The corresponding families of cubes will be denoted by $\mathcal{G}^o$, $\mathcal{G}^c$, and $\mathcal{G}^{in}$, respectively. Observe that any cube $L\in \mathcal{G}$ is either an outer cube, or a central cube, or an inner cube, or a descendant of an inner cube.
\end{definition}
We correspondingly define three subregions of $R$:
\begin{itemize}
    \item The \emph{outer region}, denoted $R^o$, is the union of $R (L)$ for $L$ varying over elements of $\mathcal{G}^o$.
    \item The \emph{central region}, denoted $R^c$, is the union of $R (L)$ for $L$ varying over elements of $\mathcal{G}^c$.
    \item The \emph{inner region}, denoted $R^{in}$, is the union of $R(L)$ for $L$ ranging over elements of $\mathcal{G}$ which are neither outer nor central cubes, or equivalently ranging over $L\in \Gcal^{in}$ and their descendants.
\end{itemize}
We refer the reader to \cite{DMS}*{Figure 2} for a depiction of the regions defined above. Let us begin with the following key lemma.

\begin{lemma}\label{l:regions}
    There exists $c=c(q,m,n,\bar{n})>0$ such that the following holds. Suppose that $T$ and $\Sbf$ satisfy Assumption \ref{a:refined} and suppose that the parameters $\bar\delta, \ \delta^*, \ \tau$ and $\eps$ are fixed as in Assumption \ref{a:parameters} (arbitrarily) with this choice of $c$. Then
    \begin{itemize}
        \item[(i)] $L_0\in\Gcal^o$ and for any $\ell\in \N$ there exists $\bar c=\bar c(q,m,n,\bar n,\tau,\ell)>0$ such that if $\eps\leq \bar c$ then $\Gcal_\ell\subset \Gcal^o$;
        \item[(ii)] For each $L\in \Gcal^c$ there exists $k(L)\in \{0,\dots,\bar\kappa\}$ such that for each $k\in\{k(L),\dots,\bar\kappa\}$ we have
        \begin{equation}\label{e:small-excess-sep-ratio}
            \Ebf(L,k)\leq\tau^2\mathbf{s}(k)^2,
        \end{equation}
        while for each $k\in\{0,\dots,k(L)-1\}$ we have
        \begin{equation}\label{e:large-excess-sep-ratio}
            \Ebf(L,k)>\tau^2\mathbf{s}(k)^2;
        \end{equation}
        \item[(iii)] For each $L\in \Gcal^o$, \eqref{e:small-excess-sep-ratio} holds for every $k\in \{0,\dots\bar\kappa\}$ (and so we define $k(L)=0$ for such $L$), while for each $L\in \Gcal^{in}$, \eqref{e:large-excess-sep-ratio} holds for every $k\in \{0,\dots,\bar\kappa\}$;
        \item[(iv)] There exists $\bar C=\bar C(q,m,n,\bar n,\bar\delta, \delta^*,\tau)>0$ such that
        \begin{align}
            \Ebf(L,k(L)) &\leq \bar{C}\Ebf(L,0) \qquad \forall L \in \Gcal^c, \\
            1 &\leq \bar{C} \Ebf(L,0) \qquad \forall L \in \Gcal^{in};
        \end{align}
        \item[(v)] For each $L\in \Gcal^{o}$, one may apply Proposition \ref{p:crude-approx} with $\eta=\rho=\frac{1}{32}$ to the rescaled current $T_{y_L,2^{-\ell(L)}}$ and the open book $\Sbf_0$, while for each $L\in \Gcal^c$ one may apply Proposition \ref{p:crude-approx} to $T_{y_L,2^{-\ell(L)}}$ and the open book $\Sbf_{k(L)}$.
    \end{itemize} 
\end{lemma}

For simplicity we will henceforth adopt the notation $\Ebf(L)\coloneqq \Ebf(L,k(L))$ for $L\in \Gcal^o\cup \Gcal^c$. The proof of this is completely analogous to that of \cite{DMS}*{Lemma 8.14}, with the application of Lemma 8.1 therein replaced with Lemma \ref{l:layers}, and the application of Propositions 8.6 and 8.9 therein replaced with Proposition \ref{p:crude-approx}. The only minor modification to the argument verifying the validity of (iv) involves the fact that now we do not have a case where $I(\bar\kappa)$ consists of a single element. Instead, we have the possibility that the alternative (a) above holds for $\bar\kappa$, implying that $\Sbf_{\bar\kappa}=\pi_{i_0}$ for some $m$-dimensional plane $\pi_{i_0}$. Thus, in this case we trivially have $\mathbf{s}(\bar\kappa) = 1 \geq \delta$. We thus omit the details here.

We are now in a position to use Lemma \ref{l:regions}(v) to construct Lipschitz approximations for $T_{y_L,2^{-\ell(L)}}$ in $(\Bbf_2\setminus \bar{\Bbf}_{2\rho_*}(V))\cap \Sbf_{k(L)}$ for all cubes $L\in \Gcal^o\cup\Gcal^c$. Letting
\[
    \Omega(L)\coloneqq (\Bbf_{2^{1-\ell(L)}}(y_L)\setminus \bar{\Bbf}_{\rho_*2^{1-\ell(L)}}(V))\cap\Sbf_{k(L)},
\]
and $\Omega_i(L)\coloneqq \Omega(L)\cap\Hbf_i$ for $i\in I(k(L))$, this yields corresponding local $Q_{L,i}$-valued Lipschitz approximations $u_{L,i}$ for $T$ over $\Omega_i(L)$. We recall the notation for the sets $\boldsymbol{\Omega}_i = \Bbf_{127/32}\cap\mathbf{p}_{\Hbf_i}^{-1}(\Omega_i)$ for $i\in I(k(L))$ from Proposition \ref{p:crude-approx} (recall $\eta=\rho_*$), as well as the closed sets $K_i$ therein, which now depend also on $L$, so we denote them by $K_{L,i}$. We in turn let
\[
    \boldsymbol{\Omega}_i(L)\coloneqq 2^{-\ell(L)}\boldsymbol{\Omega}_i + y_L, \qquad K_i(L) \coloneqq 2^{-\ell(L)}K_{L, i} + y_L.
\]
Note that although all of the above are defined only for indices $i\in I(k(L))$ for a given cube $L\in \Gcal$, we may extend this to all indices $i\in \{1,\dots,N\}$ by simply setting $Q_{L,i}=0$ for $i\notin I(k(L))$. The following proposition describes the key properties of the local graphical approximations $u_{L,i}$.

\begin{proposition}\label{p:refined}
    Let $T$, $\Sigma$ and $\Sbf$ be as in Assumption \ref{a:refined}. Suppose that the parameters $\bar\delta$, $\delta^*$, $\tau$ and $\eps$ are as in Lemma \ref{l:regions} and that $\gamma$ is as in Proposition \ref{p:crude-approx}. Then there exists $\lambda=\lambda(m)\in (1,\frac{3}{2}]$ and $\bar C=\bar C(q,m,n,\bar n,\delta^*)>0$ (but not depending on $\bar\delta$, $\tau$ or $\eps$) for which the following holds:
    \begin{itemize}
        \item[(i)] For each $i\in \{1,\dots,N\}$, we have $Q_{L,i} = Q_{L',i}$ for any $L,L'\in \Gcal^o$;
        \item[(ii)] $\sum_{i=1}^N Q_{L,i} = q$ for every $L\in \Gcal^o\cup \Gcal^c$;
        \item[(iii)] For each $L\in \Gcal^o\cup\Gcal^c$ we have $\spt(T)\cap \lambda R(L)\subset \bigcup_i \boldsymbol{\Omega}_i(L)$ and
        \begin{equation}\label{e:refined-1}
            2^{2\ell(L)}|p-\mathbf{p}_{\alpha_i}(p)|^2 \leq \bar{C}(\Ebf(L) + 2^{-2\ell(L)}\Abf^2) \qquad \forall p \in \spt(T)\cap\boldsymbol{\Omega}_i(L);
        \end{equation}
        \item[(iv)] For each $L\in \Gcal^o\cup \Gcal^c$ and $i\in \{1,\dots,N\}$, the currents
        \[
            T_{L,i} \coloneqq T\res \boldsymbol{\Omega}_i(L)\cap \{\dist(\cdot,\Hbf_i)<\bar{C} 2^{-\ell(L)}(\Ebf(L)+2^{-2\ell(L)}\Abf^2)^{1/2}\}
        \]
        satisfy
        \[
            T_{L,i}\mres \mathbf{p}_{\Hbf_i}^{-1}(K_i(L)) = \Gbf_{u_{L,i}}\mres \mathbf{p}_{\Hbf_i}^{-1}(K_i(L)),
        \]
        while $\gr(u_{L,i})\subset \Sigma$ and
        \begin{align}
            2^{2\ell(L)}\|u_{L,i}\|^2_{C^0} + 2^{m\ell(L)}\|Du_{L,i}\|^2_{L^2} & \leq \bar{C}(\Ebf(L)+2^{-2\ell(L)}\Abf^2); \label{e:refined-2} \\
            \Lip(u_{L,i}) & \leq \bar{C} (\Ebf(L)+2^{-2\ell(L)}\Abf^2)^\gamma; \label{e:refined-3} \\
            |\Omega_i(L)\setminus K_i(L)| + \|T_{L,i}\|(\boldsymbol{\Omega}_i(L) \setminus \mathbf{p}_{\Hbf_i}^{-1}(K_i(L))) & \leq \bar{C} 2^{-m\ell(L)}(\Ebf(L) + 2^{-2\ell(L)}\Abf^2)^{1+\gamma}. \label{e:refined-4}
        \end{align}
        \item[(v)] For each $L\in\Gcal^o\cup \Gcal^c$, we have $\Theta(T,\cdot)\leq \max_i Q_{L,i} + \frac{1}{2}$ in $\lambda R(L)$. In particular, when $L\in \Gcal^o$, the density satisfies $\Theta(T,\cdot) \leq \frac{q-1}{2}$.
    \end{itemize}
\end{proposition}

The proof of this is entirely analogous to that of \cite{DMS}*{Lemma 8.15}, replacing the application of \cite{DMS}*{Lemma 8.5}, \cite{DMS}*{Proposition 8.6} (or \cite{DMS}*{Proposition 8.9}), \cite{DMS}*{Lemma 8.8} and \cite{DMS}*{Lemma 8.14} therein with Lemma \ref{l:crude}, Proposition \ref{p:crude-approx}, Lemma \ref{l:matching-Q} and Lemma \ref{l:regions} respectively. Furthermore, the application of \cite{DMS}*{Lemma 8.16} therein is replaced with the following lemma, which is its analogue here. 

\begin{lemma}\label{l:curved}
Under the assumptions of Lemma \ref{l:regions}, consider $L\in \mathcal{G}^o \cup \mathcal{G}^c$, let $U\subset \lambda R(L)$ for $\lambda$ as in Proposition \ref{p:refined} be a set invariant under rotations around $V$ whose cross-sections $U_i = U\cap \Hbf_i$ are Lipschitz open sets or the closures of Lipschitz open sets, and let
\[
    \tilde{U} = \bigcup_i \{p:\mathbf{p}_{\Hbf_i}(p)\in U_i \quad \text{and} \quad |p-\mathbf{p}_{\Hbf_i}(p)(p)| \leq \bar{C} 2^{-\ell(L)}(\Ebf(L) + 2^{-2\ell(L)} \Abf^2)^{1/2}\},
\]
for $\bar{C}$ as in Proposition \ref{p:refined}. Then
\begin{align}
\|T\| (U\setminus \tilde{U}) + \|T\|(\tilde{U}\setminus U) &\leq C \left(\mathbf{E} (L) + 2^{-2\ell(L)}\mathbf{A}^2\right) 2^{-\ell (L)} \mathcal{H}^{m-1} (\partial U_i)\, \nonumber\\
&\qquad \hspace{5em} + C 2^{-m \ell (L)}\left(\mathbf{E} (L) + 2^{-2\ell(L)}\mathbf{A}^2\right)^{1+\gamma} \label{e:curved}
\end{align}
where $\gamma$ is as in Proposition \ref{p:crude-approx} and $C$ depends on $q,m,n,\bar{n},\bar\delta$ and the Lipschitz regularity of the boundary of $2^{\ell (L)} U_i$. 
\end{lemma}

The proof is exactly the same of \cite{DMS}*{Lemma 8.16}. Note indeed that it suffices to replace the inductive hypothesis (A) in the argument of \cite{DMS} with the following analogue: 
\begin{itemize}
    \item[(A)] If $\sum_i Q_{L',i}=q$ for $L'\in \mathcal{G}^c\cup \mathcal{G}^o$, then $\sum_i Q_{L,i} =q$ for every child $L\in \mathcal{G}^c\cup \mathcal{G}^o$ of $L^\prime$.
\end{itemize}

\subsection{Coherent outer approximation and first blow-up}
Let us now construct a single multi-valued approximation defined over $\bigcup_i L_i$ for cubes $L\in \Gcal^o$, from the local Lipschitz approximations in Proposition \ref{p:refined}. We begin by introducing the following notation.

\begin{definition}\label{d:coherent}
    Let $T$, $\Sigma$ and $\Sbf$ be as in Assumption \ref{a:refined}. Let $L\in \Gcal^o$. Then we define
    \[
        \Nscr(L)\coloneqq \{L'\in \Gcal^o : R(L)\cap R(L')\neq \emptyset \},
    \]
    and
    \[
        \bar\Ebf(L) \coloneqq \max\{\Ebf(L') : L'\in \Nscr(L)\}.
    \]
    Moreover, for each $i\in \{1,\dots,N\}$, let
    \[
        R_i^o \coloneqq \bigcup_{L\in \Gcal^o} L_i \equiv \Hbf_i \cap \bigcup_{L\in \Gcal^o} R(L).
    \]
\end{definition}

We are now in a position to state the coherent outer approximation result, which follows from Proposition \ref{p:refined} in the same way as \cite{DMS}*{Proposition 8.19} follows from \cite{DMS}*{Proposition 8.15}.

\begin{proposition}[Coherent outer approximation]\label{p:coherent}
    Let $T$, $\Sigma$, $\Abf$ and $\Sbf$ be as in Assumption \ref{a:refined}, let $T_{L,i}$ be as in Proposition \ref{p:refined}, let $\gamma$ be as in Proposition \ref{p:crude-approx} and let $Q_i\coloneqq Q_{L_0,i}$ for $i\in \{1,\dots,N\}$ . Then there exist Lipschitz maps $u_i: R_i^o\to \Acal_{Q_i}(\Hbf_i^\perp)$ and closed subsets $\bar{K}_i(L)\subset L_i$ such that
    \begin{itemize}
        \item[(i)] $\gr(u_i)\subset \Sigma$ and $T_{L,i}\mres \mathbf{p}_{\Hbf_i}^{-1}(\bar{K}_i(L))=\Gbf_{u_i}\mres \mathbf{p}_{\Hbf_i}^{-1}(\bar{K}_i(L))$ for each $L\in \Gcal^o$;
        \item[(ii)] One has estimates analogous to \eqref{e:refined-2}-\eqref{e:refined-4}, namely
        \begin{align}
            2^{2\ell(L)}\|u_{i}\|^2_{C^0(L_i)} + 2^{m\ell(L)}\|Du_{i}\|^2_{L^2(L_i)} & \leq \bar{C}(\Ebf(L)+2^{-2\ell(L)}\Abf^2); \\
            \Lip(u_{i}) & \leq \bar{C} (\Ebf(L)+2^{-2\ell(L)}\Abf^2)^\gamma; \\
            |L_i\setminus \bar{K}_i(L)| + \|T_{L,i}\|(\mathbf{p}_{\Hbf_i}^{-1}(L_i\setminus \bar{K}_i(L))) & \leq \bar{C} 2^{-m\ell(L)}(\Ebf(L) + 2^{-2\ell(L)}\Abf^2)^{1+\gamma}.
        \end{align}
    \end{itemize}
\end{proposition}

We may now collect together the local estimates in Proposition \ref{p:coherent} to yield a global Dirichlet energy estimates on $R_i$ for the maps $u_i$, which are key for our final blow-up procedure. The proof is exactly the same as that for \cite{DMS}*{Proposition 8.20}; note that \cite{DLS14Lp}*{Theorem 2.6} still applies here since the maps $u_i$ are $\Acal_{Q_i}$-valued (rather than special $Q$-valued), and $T\mres R^o$ identifies with an integral current without boundary.

\begin{proposition}[First blow-up]\label{p:first-blowup}
    Let $T$, $\Sigma$, $\Abf$ and $\Sbf$ be as in Assumption \ref{a:refined}. Let the parameters $\bar\delta$, $\delta^*$, $\tau$ and $\eps$ be fixed as in Lemma \ref{l:regions}. Then for every $\sigma,\varsigma>0$ there exist $C=C(q,m,n,\bar{n},\delta^*,\bar\delta,\tau)>0$ and $\eps=\eps(q,m,n,\bar{n},\delta^*,\bar\delta,\tau,\sigma,\varsigma)>0$ such that
    \begin{itemize}
        \item[(i)] We have the inclusion $R\setminus \Bbf_\sigma(V) \subset R^o$;
        \item[(ii)] The maps $u_i$ in Proposition \ref{p:coherent} satisfy
        \begin{equation}\label{e:first-blowup}
            \int_{R_i} |Du_i|^2 \leq C (\sigma^{-2} \hat{\Ebf}(T,\Sbf,\Bbf_4) + \Abf^2),
        \end{equation}
        where $R_i\coloneqq (R\setminus \Bbf_\sigma(V))\cap\Hbf_i$;
        \item[(iii)] If in addition $\Abf^2 \leq \eps^2 \hat{\Ebf}(T,\Sbf,\Bbf_4)$ then for the normalizations $v_i\coloneqq \hat{\Ebf}(T,\Sbf,\Bbf_4)^{-1/2} u_i$, there exist Dir-minimizing maps $w_i : R_i \to \Acal_{Q_i}(\Hbf_i^\perp)$ with
        \[
            d_{W^{1,2}}(v_i,w_i) \leq \varsigma,
        \]
        for the $W^{1,2}$-distance $d_{W^{1,2}}$ as in \cite{DLS_MAMS}.
    \end{itemize}
\end{proposition}
Note that the constant $C$ in Proposition \ref{p:first-blowup} is independent of $\sigma$ and $\varsigma$. Moreover, the estimate \eqref{e:first-blowup} is suboptimal since it blows up near the spine of $\Sbf$; we will proceed to improve it to a $\sigma$-independent estimate in Section \ref{s:blowup}.

\section{Reduction of Theorem \ref{t:decay}}
Here, we reduce Theorem \ref{t:decay} to a weaker decay result, which concludes that the decay to a new $(m-1)$-invariant cone occurs at one of two possible scales; see Theorem \ref{t:two-scale-decay} below. This type of ``multi-scale excess decay'' naturally arises in our setting when the cones could be degenerating, and has been used in other such contexts (see \cites{W14_annals, MW, Min}). First of all, recall that for an open book $\Sbf=\Hbf_1\cup\cdots\cup \Hbf_N\in \Bscr^q$, with $\pi_i$ the $m$-dimensional planes that are the extensions of $\Hbf_i$ as before, we denote
\[
    \boldsymbol{\sigma}(\Sbf)\coloneqq \min_{i<j} \dist(\Hbf_i\cap\Bbf_1,\Hbf_j\cap\Bbf_1), \qquad \boldsymbol{\zeta}(\Sbf) \coloneqq \max_{i<j} \dist(\pi_i\cap\Bbf_1,\pi_j\cap\Bbf_1).
\]
\begin{theorem}\label{t:two-scale-decay}
    Let $q, Q, m,n,\bar{n}$ be as in Assumption \ref{a:refined} and fix $\varsigma_1>0$. Then there exist $\eps_1=\eps_1(q,m,n,\bar{n},\varsigma_1)\in (0,\frac{1}{2}]$ and $\bar{r}_i=\bar{r}_i(q,m,n,\bar{n},\varsigma_1)\in (0,\frac{1}{2}]$ for $i=1,2$, such that the following holds. Suppose that
    \begin{itemize}
        \item[(i)] $T$ and $\Sigma$ are as in Assumption \ref{a:main};
        \item[(ii)] $\|T\|(\Bbf_1)\leq (Q+\frac{1}{4})\omega_m$;
        \item[(iii)] There exists an open book $\Sbf\in \Bscr^q(0)$ with
        \begin{equation}\label{e:smallness-min-sep}
            \Ebb(T,\Sbf,\Bbf_1) \leq \eps_1^2 \boldsymbol{\sigma}(\Sbf)^2;
        \end{equation}
        \item[(iv)] $\Abf^2 \leq \eps_1^2 \Ebb(T,\tilde\Sbf,\Bbf_1)$ for each $\tilde\Sbf\in \Bscr^q(0)$.
    \end{itemize}
    Then there exists an open book $\Sbf'\in \Bscr^q(0)\setminus \Pscr(0)$ such that for some $i\in \{1,2\}$ we have
    \begin{equation}
        \Ebb(T,\Sbf',\Bbf_{\bar{r}_i})\leq \varsigma_1\Ebb(T,\Sbf,\Bbf_1).
    \end{equation}
\end{theorem}

In order to prove that Theorem \ref{t:two-scale-decay} implies Theorem \ref{t:decay}, we first require the following intermediate proposition. First of all the latter allows us to reduce the hypothesis \eqref{e:smallness} to the one in \eqref{e:smallness-min-sep}, up to removing some of the half-planes in $\Sbf$: such an operation may increase the $L^2$ conical excess but by no more than a dimensional constant. Secondly, in Proposition \ref{p:multi-decay}  the decay is achieved at one of finitely many scales. This can be compared with the reduction from \cite{DMS}*{Theorem 10.2} to \cite{DMS}*{Proposition 10.3}.

\begin{proposition}\label{p:multi-decay}
	Let $q,n,\bar{n} \in \N$, let $m\in \N_{\geq 2}$ and let $Q=\frac{q}{2}$. Let $\bar{N}=q(q-1)$ and fix $\varsigma_2>0$. Then there exist $\eps_2, r_1,\dots,r_{\bar N} \leq \frac{1}{2}$, depending on $q,m,n,\bar{n},\varsigma$, such that the following holds. Suppose that $T$, $\Sigma$ and $\Abf$ are as in Theorem \ref{t:decay} with $\eps_2$ in place of $\eps_0$, namely:
    \begin{itemize}
        \item[(i)] $T$, $\Sigma$ and $\Abf$ are as in Assumption \ref{a:main};
        \item[(ii)] $\|T\|(\Bbf_1)\leq (Q+\frac{1}{4})\omega_m$;
        \item[(iii)] There exists $\Sbf \in \Bscr^q(0)\setminus\Pscr(0)$ with
        \[
            \Ebb(T,\Sbf,\Bbf_1) \leq \eps_2^2 \Ebf^p(T,\Bbf_1);
        \]
        \item[(iv)] $\Abf^2\leq \eps_2^2\Ebb(T,\tilde\Sbf,\Bbf_1)$ for each $\tilde\Sbf\in \Bscr^q(0)$.
    \end{itemize}
    Then there exists $\Sbf'\in \Bscr^q(0)\setminus \Pscr(0)$ and an index $i \in \{1,\dots, \bar{N}\}$ such that
\begin{equation}
	\Ebb(T,\Sbf',\Bbf_{r_i}) \leq \varsigma_2\Ebb(T,\Sbf,\Bbf_1).
\end{equation}
\end{proposition}

\begin{proof}[Proof of Proposition \ref{p:multi-decay} from Theorem \ref{t:two-scale-decay}]
 Fix $\Sbf=\Hbf_1\cup\cdots\cup \Hbf_N\in \Bscr^q(0)\setminus \Pscr(0)$. We argue in the spirit of that in \cite{DMS}*{Section 10}. Fix $k\in \{3,\dots, q\}$, and for $s>0$ define $\eps^{(k)}(s)>0$, $r^{(k)}_i(s)>0$ as follows:
 \begin{itemize}
 	\item Take $q=k$ and $\varsigma_1=s$ in Theorem \ref{t:two-scale-decay} and let $\eps^{(k)}(s)\coloneqq \eps_1(k,m,n,\bar{n},s)$ and $r^{(k)}_i(s) \coloneqq r_i(k,m,n,\bar{n},s)$ be the corresponding parameters therein, for $i=1,2$.
 \end{itemize}
Note that $k$ characterizes the number of half-planes in a given open book, while $s$ characterizes the factor by which the two-sided $L^2$ conical excess decays as in Theorem \ref{t:two-scale-decay}. 

We start observing that, an elementary argument (c.f. \cite{DMS}*{Proof of Proposition 10.3}) shows the existence of a constant $C^\dagger=C^\dagger(m,n,\bar{n})>0$ such that
\begin{equation}\label{e:C-dagger}
    \Ebf^p(T,\Bbf_1) \leq C^\dagger \boldsymbol{\zeta}(\Sbf)^2\, .
\end{equation}
In particular we can conclude that Theorem \ref{t:two-scale-decay} can be immediately applied when $N=2$ provided $\varepsilon_2\leq \varepsilon_1/\sqrt{C^\dagger}$, because for $N=2$ we trivially have $\boldsymbol{\zeta} (\Sbf) \leq \boldsymbol{\sigma} (\Sbf)$. Thus, we may assume that $N\geq 3$. Moreover, we may assume that
\begin{equation}\label{e:large-excess-min-sep-ratio}
    \Ebb(T,\Sbf,\Bbf_1) > \eps^{(N)}(\varsigma_2)^2 \boldsymbol{\sigma}(\Sbf)^2.
\end{equation}
Indeed, otherwise, as long as $\eps_2\leq\eps^{(N)}(\varsigma_2)$, we may apply Theorem \ref{t:two-scale-decay}.

Let us now proceed to prune $\Sbf$, following an argument analogous to that for the Pruning Lemma (Lemma \ref{l:pruning}). 

Observe that, because of \eqref{e:C-dagger} and \eqref{e:large-excess-min-sep-ratio}, we have 
\[
\boldsymbol{\sigma} (\mathbf{S})^2 
< \varepsilon^{(N)} (\varsigma_2)^{-2} \Ebb (T, \mathbf{S}, \Bbf_1)
\leq \varepsilon_2^2 \varepsilon^{(N)} (\varsigma_2)^{-2} \Ebf^p (T, \Bbf_1)
\leq C^\dagger \varepsilon_2^2 \varepsilon^{(N)}  (\varsigma_2)^{-2} \boldsymbol{\zeta} (\mathbf{S})^2\, .
\]
In particular, if $\varepsilon_2$ is small enough,
\begin{equation}\label{e:s<z}
\boldsymbol{\sigma} (\mathbf{S}) <
\boldsymbol{\zeta} (\mathbf{S})\, .
\end{equation}
 Up to relabelling the indices, we can assume that 
\[
    \boldsymbol{\sigma}(\Sbf) = \dist(\Hbf_1\cap\Bbf_1,\Hbf_2\cap\Bbf_1)\, .
\]
Now let
\[
    \boldsymbol{\zeta}(\Sbf)=\dist(\pi_{i_*}\cap\Bbf_1,\pi_{j_*}\cap\Bbf_1)\, ,
\]
where $\pi_i$ is the $m$-dimensional plane containing $\Hbf_i$. Obviously $\{i_*, j_*\}\neq \{1,2\}$, otherwise we would have 
$\boldsymbol{\zeta} (\mathbf{S}) 
\leq \boldsymbol{\sigma} (\mathbf{S})
< \boldsymbol{\zeta} (\mathbf{S})$.
We can therefore without loss of generality assume that $i_*, j_*\geq 2$.

Now let $\Sbf_{N-1}\coloneqq \Hbf_2\cup\cdots\cup\Hbf_N$ denote the open book with the half-plane $\Hbf_1$ removed. This satisfies $V(\Sbf_{N-1})=V(\Sbf)$ and since $i_*,j_*\neq 1$, we have $\boldsymbol{\zeta}(\Sbf)=\boldsymbol{\zeta}(\Sbf_{N-1})$. Moreover, by \eqref{e:large-excess-min-sep-ratio} we have
\[
    \dist^2(\Sbf\cap\Bbf_1,\Sbf_{N-1}\cap\Bbf_1)=\boldsymbol{\sigma}(\Sbf)^2\leq \eps^{(N)}(\varsigma_2)^{-2}\Ebb(T,\Sbf,\Bbf_1).,
\]
which in turn yields
\[
    \Ebb(T,\Sbf_{N-1},\Bbf_1) \leq C_0 (1+ \eps^{(N)}(\varsigma_2)^{-2})\Ebb(T,\Sbf,\Bbf_1),
\]
for $C_0=C_0(q,m,n,\bar n)>0$. 

We then proceed as in the proof of \cite{DMS}*{Proposition 10.3}. In particular we define the constants $C_0^*=1$ and $C_j^*=C_0\left(C_{j-1}^* + \eps^{(N-(j-1))}(\varsigma_2/C_{j-1}^*)^{-2}\right)$ for $1\leq j \leq N-2$ and produce cones $\mathbf{S}_{N-1}, \mathbf{S}_{N-2}, \ldots , \mathbf{S}_{N-k}$ by removing at each step a half plane from $\mathbf{S}_{N-j}$ to produce $\mathbf{S}_{N-j-1}$. This is done as long as at the step $j$ we cannot apply Theorem \ref{t:two-scale-decay} to the cone $\mathbf{S}_{N-j}$ with decay $\varsigma_2/C_j^*$, namely as long as 
\begin{equation}\label{e:large-excess-min-sep-ratio-ind}
   \Ebb(T,\Sbf_{N-j},\Bbf_1) > \eps^{(N-j)}(\varsigma_2/C_j^*)^2 \boldsymbol{\sigma}(\Sbf_{N-j})^2.
\end{equation}
In particular $k$ is the first step at which \eqref{e:large-excess-min-sep-ratio-ind} fails.
For the moment we assume that we have shown that at each intermediate step $j$ (i.e. at each step for which \eqref{e:large-excess-min-sep-ratio-ind} holds for $j$ and for all smaller indices) we have
\begin{equation}\label{e:s<z-ind}
\boldsymbol{\sigma} (\Sbf_{N-j}) < \boldsymbol{\zeta} (\Sbf_{N-j})\, ,
\end{equation}
so that our algorithm produces at step $j$ a cone $\Sbf_{N-j-1}$ with $\boldsymbol{\zeta} (\mathbf{S}_{N-j-1}) = \boldsymbol{\zeta} (\mathbf{S}_{N-j}) = \boldsymbol{\zeta} (\mathbf{S})$ (we will justify \eqref{e:s<z-ind} momentarily). Incidentally, showing \eqref{e:s<z-ind} also proves that the procedure stops necessarily after at most $N-2$ steps, because \eqref{e:s<z-ind} cannot be valid when the cone in question consists of two halfplanes.

Given that the procedure stops after at most $N-2$ steps, it yields a collection of radii 
$$r_1^{N,N}, r_2^{N,N}, r_1^{N,N-1}, r_2^{N,N-1},\dots,r_1^{N,2}, r_2^{N,2}$$ 
for which the conclusion of Proposition \ref{p:multi-decay} must hold, where $r_1^{N,N-k}, r_2^{N,N-k}$ are the two radii given by Theorem \ref{t:two-scale-decay} applied to a cone with $N-k$ halfplanes when imposing a decay factor $\varsigma_2/C_j^*$. As $N$ ranges over the collection $\{2,\dots, q\}$, we obtain $q(q-1)$ possible radii in total. We omit the details here, and instead refer the reader to the argument within the proof of \cite{DMS}*{Proposition 10.3}.

It now remains to check that \eqref{e:s<z-ind} holds. We claim that this is guaranteed if we impose the smallness condition
\begin{equation}
\varepsilon_2^2 \leq \min_j\,  \left[(C^\dagger C_j^*)^{-1} \varepsilon^{(N-j)} (\varsigma_j/C_j^*)^2\right]\, .
\end{equation}
Indeed let us argue by contradiction and assume that $j$ is the first step at which \eqref{e:s<z-ind} fails (we already argued that \eqref{e:s<z-ind} holds for $j=0$). In particular since it held at all the steps prior to $j$, we know that $\boldsymbol{\zeta} (\Sbf) = \boldsymbol{\zeta} (\Sbf_{N-j})$. Observe moreover that, as argued in \cite{DMS}*{Proposition 10.3}, we have the inequality
\[
\Ebb (T, \mathbf{S}_{N-j}, \Bbf_1)
\leq C_j^* \Ebb (T, \mathbf{S}, \Bbf_1)\, .
\]
In particular, combining the last two pieces of information with \eqref{e:C-dagger} and \eqref{e:large-excess-min-sep-ratio-ind}, we get 
\begin{align*}
    \boldsymbol{\zeta} (\Sbf)^2 = \boldsymbol{\zeta}(\Sbf_{N-j})^2 \leq \boldsymbol{\sigma}(\Sbf_{N-j})^2 &< \eps^{(N-j)}(\varsigma_2/C^*_j)^{-2}\Ebb(T,\Sbf_{N-j},\Bbf_1) \\
    &\leq C_j^*\eps^{(N-j)}(\varsigma_2/C^*_j)^{-2} \Ebb(T,\Sbf,\Bbf_1) \\
    &\leq C_j^*\eps^{(N-j)}(\varsigma_2/C^*_j)^{-2} \eps_2^2\Ebf^p(T,\Bbf_1) \\
    &\leq C^\dagger C_j^*\eps^{(N-j)}(\varsigma_2/C^*_j)^{-2}  \boldsymbol{\zeta}(\Sbf)^2\\ 
    &\leq \boldsymbol{\zeta}(\Sbf)^2\, ,
\end{align*}
which is a contradiction. 
\end{proof}

The proof of Theorem \ref{t:decay} assuming the validity of Proposition \ref{p:multi-decay} now follows by analogous reasoning to that in \cite{DMS}*{Section 10}. However, for the purpose of clarity, we repeat the parts of the argument that have minor differences due to our cones being open books here. We begin with some intermediate results which will be needed. The first is the following analogue of \cite{DMS}*{Corollary 10.4}, which is a consequence of the Pruning Lemma (Lemma \ref{l:pruning}) and Proposition \ref{p:refined}, which allows us to control the two-sided conical excess at smaller scales comparable to 1. The proof is exactly the same as that in \cite{DMS}, only without the requirement of balancing the open book, and with Lemma \ref{l:pruning}, Lemma \ref{l:regions} and Proposition \ref{p:refined} used in place of \cite{DMS}*{Lemma 8.2, Lemma 8.14, Proposition 8.15} respectively. We therefore omit the proof here.

\begin{lemma}\label{l:scaled-excess-control}
    Assume that $\tilde{T}$, $\tilde\Sigma$ and $\tilde\Abf = \Abf(\tilde\Sigma)$ satisfy Assumption \ref{a:main} and let $q,Q,m,n,\bar{n}$ be as in Assumption \ref{a:refined}. Let $\tilde\Sbf\in \Bscr^q(0)\setminus \Pscr(0)$ and let $\bar r \in (0,1]$. Then there exist $\tilde\eps=\tilde\eps(q,m,n,\bar{n},\bar r)>0$ and $\tilde C = \tilde C(q,m,n,\bar n, \bar r)>0$ such that the following holds. Suppose that
    \[
        \Ebb(\tilde T,\tilde \Sbf,\Bbf_1) \leq \tilde\eps^2\Ebf^p(\tilde T,\Bbf_1)
    \]
    and
    \[
        \tilde\Abf^2 \leq\tilde\eps^2\Ebb(\tilde T,\bar\Sbf,\Bbf_1) \qquad \forall \bar\Sbf\in\Bscr^q(0).
    \]
    Then there exists $\Sbf'\in \Bscr^q(0)\setminus \Pscr(0)$ with
    \begin{equation}\label{e:excess-control}
        \Ebb(\tilde T,\Sbf',\Bbf_r)\leq \tilde{C}\Ebb(\tilde T, \tilde\Sbf,\Bbf_1) \qquad \forall r\in [\bar r,1].
    \end{equation}
\end{lemma}

Furthermore, we have the following analogue of \cite{DMS}*{Lemma 10.5}.

\begin{lemma}\label{l:rescalings-decay}
    Let $\bar r \in (0,\frac{1}{2}]$ and $\eps_2\in (0,1)$. There exists $\eps_0=\eps_0(q,m,n,\bar n,\bar r,\eps_2)>0$ such that the following holds. Suppose that $T$, $\Sigma$, $\Abf$ and $\Sbf$ satisfy the hypotheses of Theorem \ref{t:decay} with this choice of $\eps_0$. Then for each $r\in [\bar r,\frac{1}{2}]$ the following holds. Suppose that there exists $\Sbf_r\in \Bscr^q(0)$ with
    \begin{itemize}
        \item[(i)] $\Ebb(T,\Sbf_r,\Bbf_r)\leq \Ebb(T,\Sbf,\Bbf_1)$,
        \item[(ii)] $r^2\Abf^2 \leq \eps_0^2\inf\{\Ebb(T,\Sbf',\Bbf_r) : \Sbf'\in \Bscr^q(0)\}$.
    \end{itemize}
    Then the rescalings $T_{0,r}$, $\Sigma_{0,r}$ and $\Sbf_r$ also satisfy the hypotheses of Theorem \ref{t:decay} (with $\eps_2$ in place of $\eps_0$ therein).
\end{lemma}

\begin{proof}
    We argue analogously to that in the proof of \cite{DMS}*{Lemma 10.5}, only now in our compactness argument, we use the Lipschitz approximation of \cite{DLHMS}*{Theorem 15.1} in place of Almgren's strong Lipschitz approximation \cite{DLS14Lp}*{Theorem 2.4}.
    
    Fix $\bar{r}\in (0,\frac{1}{2}]$ and $\eps_2\in(0,1)$. Suppose, for a contradiction, that the conclusion of the lemma fails to hold. Then there exists a sequence $\eps_0^k\downarrow 0$ and corresponding sequences $T_k$, $\Sigma_k$, $\Abf_k$, $\Sbf_k = \Hbf_1^k\cup\cdots \cup \Hbf_N^k\in \Bscr^q(0)$ (with $N$ a fixed positive integer) satisfying the hypotheses of Theorem \ref{t:decay} with $\eps_0^k$ and (i), (ii) but for which the hypotheses of Theorem \ref{t:decay} fail for the rescalings $(T_k)_{0,r_k}$, $(\Sigma_k)_{0,r_k}$ and $\Sbf_{r_k}\coloneqq (\Sbf_k)_{r_k}$ for some scales $r_k \in [\bar r,1]$. In particular
    \begin{equation}\label{e:rescale-hyp}
    	r_k^2\Abf_k^2 \leq (\eps_0^k)^2\Ebb(T_k,\Sbf_{r_k}, \Bbf_{r_k}) \leq (\eps_0^k)^2\Ebb(T_k,\Sbf_k,\Bbf_1) \leq (\eps_0^k)^4 \Ebf^p(T_k,\Bbf_1).
    \end{equation}
    but
    \begin{equation}\label{e:rescale-contradiction}
   		\frac{\Ebb(T_k,\Sbf_{r_k},\Bbf_{r_k})}{\Ebf^p(T_k,\Bbf_{r_k})} \geq \eps_2^{2}.
    \end{equation}
    In other words, we suppose that the hypothesis (iii) of Theorem \ref{t:decay} fails. Note that the validity of the remaining hypotheses follows immediately from the assumptions of the Lemma.
    
    Now, \eqref{e:rescale-hyp} and \eqref{e:rescale-contradiction} together tell us that $\Ebf^p(T_k,\Bbf_{r_k})\to 0$ as $k\to\infty$. Thus, up to subsequence, $T_k\mres\Bbf_{r_k}$ converges weakly to $q_0\llbracket \pi_\infty \rrbracket$, for some $m$-dimensional plane $\pi_\infty$ and some positive integer $q_0$. Meanwhile, up to another subsequence, we have $\Sbf_k\to \Sbf_\infty\in \Bscr^q(0)$ locally in Hausdorff distance. As $\mathbb{E}(T_k,\Sbf_{r_k},\Bbf_{r_k})\to 0$ (from \eqref{e:rescale-hyp}), this gives that $\pi_\infty = \Sbf_\infty$. In turn, as $\mathbb{E}(T_k,\Sbf_k,\Bbf_1)\to 0$ (again from \eqref{e:rescale-hyp}), we get $\spt(T_\infty) = \pi_\infty\cap \Bbf_1$ for the weak limit $T_\infty$ of $T_k\mres\Bbf_1$, taken along yet another subsequence. Thus, we deduce that in fact $\Ebf^p(T_k,\Bbf_1)\to 0$, and for $k$ sufficiently large we are in a position (up to rotation) to apply the strong Lipschitz approximation result \cite{DLHMS}*{Theorem 15.1} for $T_k\mres \Cbf_{1/2}(0,\pi_\infty)$, yielding functions $f_k:B_{1/2}(\pi_\infty) \to \Ascr_Q(\pi_\infty^\perp)$. Note that the scale reduction by factor $1/4$ therein can be replaced by factor $1/2$, up to increasing the constants in the estimates. Letting $E_k\coloneqq \Ebf^p(T_k,\Bbf_1)$, the normalizations
    \[
    	\bar{f}_k\coloneqq \frac{f_k}{E_k^{1/2}}
    \]
    converge, up to yet another subsequence, to a Dir-minimizing function $f_\infty:B_{1/2}(\pi_\infty) \to \Ascr_Q(\pi_\infty^\perp)$.
    
     In addition, \eqref{e:rescale-hyp} allows us to write $\Sbf_k$ as a union of suitable portions of graphs of linear functions $L_1^k,\dots, L_N^k:\pi_\infty\to\pi_\infty^\perp$. More precisely, we first note that  $W_k = \mathbf{p}_{\pi_\infty} (V (\mathbf{S}_k))$ is an $(m-1)$-dimensional space for $k$ large enough, given that $\Sbf_k$ converges to $\pi_\infty$ locally in the Hausdorff topology. Thus we can assume, up to a further rotation which fixes $\pi_\infty$, that $W_k$ is in fact independent of $k$: from now on we will simply denote it by $W$. Further, $W$ subdivides $\pi_\infty$ into two halfplanes $\pi^+_\infty$ and $\pi^-_\infty$, and we can subdivide the linear functions $L^k_i$ into two further (both nonempty) subcollections $L^{k}_{+,i}$ and $L^{k}_{-,i}$, so that $\mathbf{S}_k$ is the union of the graphs of $L^{k}_{\pm, i}$ (for a choice of sign $\pm$ for each $k$) over the corresponding halfplanes $\pi^\pm_\infty$.
     
    Up to one more subsequence, we can finally assume that $E_k^{-1/2} L_{\pm,i}^k$ converge to linear functions $L_{\pm, i}^\infty$ as $k\to\infty$, for each $i$. Let us denote by $\bar{\Sbf}_\infty$ the union (again with the appropriate choice of sign $\pm$ for each $k$) of their graphs over the corresponding halfplanes $\pi^\pm_\infty$. Under our assumption that $\varepsilon_0^k \downarrow 0$, we then infer that the graph of $f_\infty$ coincides with $\bar{\Sbf}_\infty \cap (B_{1/2} (\pi_\infty)\times \pi_\infty^\perp)$. Note that Lemma \ref{l:planar-excess-zeta} implies further that at least two of the linear maps $L_{\pm,i}^\infty$ are distinct, since  $\boldsymbol{\zeta} (\bar{\Sbf}_\infty) > 0$. 

    Now, given the $L^2$ convergence of $\bar f_k$ to $f_\infty$, the estimates on the difference between the graph of $f_k$ and the current $T_k$,  and the bounds $0 < \bar r<r_k\leq \frac{1}{2}$, we infer that
    \begin{equation}\label{e:planar-exc-ratio}
    C^{-1} \leq \lim_{k\to\infty} \frac{\Ebf^p(T_k,\Bbf_{r_k})}{E_k} \equiv
    \lim_{k\to\infty} \frac{\Ebf^p(T_k,\Bbf_{r_k})}{\Ebf^p(T_k,\Bbf_1)} \leq C
    \end{equation}
    for some constant $C=C(q,m,n,\bar{n})>0$. We refer the reader to the analogous argument in the proof of \cite{DMS}*{Lemma 10.6} for more details. This, combined with \eqref{e:rescale-hyp}, in turn yields
    \[
        \frac{\Ebb(T_k,\Sbf_{r_k},\Bbf_{r_k})}{\Ebf^p(T_k,\Bbf_{r_k})} \leq C \frac{\Ebb(T_k,\Sbf_{k},\Bbf_{1})}{\Ebf^p(T_k,\Bbf_{1})} \leq (\eps_0^k)^2 \to 0.
    \]
    This contradicts \eqref{e:rescale-contradiction}, which concludes the proof. 
    
    Note that the above argument does not require $T$ to be area-minimizing mod$(q)$. Indeed, it suffices to use the Lipschitz approximation for general stationary integral varifolds, together with Allard's $L^2-L^\infty$ height bound to justify \eqref{e:planar-exc-ratio}. Observe in addition that the validity of the property $\boldsymbol\zeta(\bar\Sbf_\infty)$ is a mere application of the triangle inequality, and does not require anything about the structure of $T$.
\end{proof}

\subsection{Proof of Theorem \ref{t:decay} from Proposition \ref{p:multi-decay}}
    The conclusion (a) of Theorem \ref{t:decay} follows by the exact same reasoning as in \cite{DMS}, with Lemma \ref{l:rescalings-decay} applied in place of \cite{DMS}*{Lemma 10.5}, now with
    \[
        \Ebb(T,\Bbf_r) \coloneqq \inf\{\Ebb(T,\Sbf,\Bbf_r) : \Sbf\in \Bscr^q(0)\setminus \Pscr(0)\}.
    \]
    The conclusions (b)-(d) will then follow immediately from the following lemma (which will also come in useful independently in later sections), under the assumption that the conclusion (a) holds.
    
    \begin{lemma}\label{l:spine-tilt-estimate}
        Let $r_0\in (0,1]$. Then there exist constants $\eps=\eps(q,m,n,\bar n,r_0)>0$ and $C=C(q,m,n,\bar n)>0$ (independent of $r_0$) such that the following holds. Suppose that
        \begin{itemize}
            \item[(i)] $T$, $\Sigma$ and $\Abf$ are as in Assumption \ref{a:main};
            \item[(ii)] $\|T\|(\Bbf_1) \leq (Q+\frac{1}{4})\omega_m$;
            \item[(iii)] there exists $\Sbf=\Hbf_1\cup\cdots\cup\Hbf_N\in \Bscr^q(0)\setminus \Pscr(0)$ with
            \[
                \Abf^2 + \Ebb(T,\Sbf,\Bbf_1) \leq \eps^2\Ebf^p(T,\Bbf_1).
            \]
        \end{itemize}
        Then
        \begin{equation}\label{e:1-homog}
            C^{-1}\Ebf^p(T,\Bbf_1) \leq \Ebf^p(T,\Bbf_{r_0}) \leq C\Ebf^p(T,\Bbf_1).
        \end{equation}
        Furthermore, there exists $\bar{C}=\bar C(q,m,n,\bar n, r_0) > 0$ such that, up to further decreasing $\eps$ if necessary (with the same dependencies), if there exists another open book $\Sbf'= \Hbf'_1\cup\cdots \Hbf'_{M}\in \Bscr^q(0)\setminus \Pscr(0)$ satisfying
        \[
            \Ebb(T,\Sbf',\Bbf_{r_0}) \leq \eps^2 \Ebf^p(T,\Bbf_1),
        \]
        then
        \begin{align}
            \dist^2(\Sbf\cap\Bbf_1,\Sbf'\cap\Bbf_1)&\leq \bar C (\Abf^2 + \Ebb(T,\Sbf,\Bbf_1) + \Ebb(T,\Sbf',\Bbf_{r_0})); \label{e:reduction-1} \\
            \dist^2(V(\Sbf)\cap\Bbf_1,V(\Sbf')\cap\Bbf_1) &\leq \bar C\frac{\Abf^2 + \Ebb(T,\Sbf,\Bbf_1) + \Ebb(T,\Sbf',\Bbf_{r_0})}{\Ebf^p(T,\Bbf_1)}. \label{e:reduction-2}
        \end{align}
    \end{lemma}

    Before coming to the proof of the above statement, we point out a geometric fact which will be useful and whose proof is essentially already contained in \cite{DMS}. Since however here we are dealing with halfplanes rather than full planes, we give all the details.

    \begin{lemma}\label{l:disks-planes}
    For every $c_0>0$ and $\mu> 0$ there is a constant $C = C(m,n, c_0, \mu)$ with the following property. Assume $\mathbf{H}_1$ and $\mathbf{H}_2$ are two arbitrary $m$-dimensional half-planes with $(m-1)$-dimensional spines $V_1$ and $V_2$ (passing through $0$), that $p\in \mathbf{H}_1\cap \Bbf_{1/2}$ with $\dist (p, V_1) \geq 2 c_0$, and that $E\subset \Bbf_{c_0} (p)\cap \Hbf_1$ with $\mathcal{H}^m (E) \geq \mu \mathcal{H}^m (\Bbf_{c_0} (p)\cap \Hbf_1)$. Then
    \begin{equation}\label{e:hp-close}
    \sup_{x\in \Bbf_1 \cap \Hbf_1} \dist (x, \Hbf_2) 
    \leq C \sup_{y\in E} \dist (y, \Hbf_2)\, .
    \end{equation}
    \end{lemma}
    \begin{proof} Set 
    \[
    D := \sup_{y\in E} \dist (y, \Hbf_2)\, .
    \]
    We will assume 
    \begin{equation}\label{e:smallness-D-eta}
    D \leq \eta 
    \end{equation}
    where $\eta = \eta (m,n,c_0, \mu)$ is a sufficiently small constant whose choice will be specified later. Of course if $D>\eta$ the inequality \eqref{e:hp-close} holds trivially for $C= \eta^{-1}$.
    
    Next, denote by $\pi_1$ and $\pi_2$ the two $m$-dimensional planes which contain $\Hbf_1$ and $\Hbf_2$ respectively. We first claim that
    \begin{equation}\label{e:p-close}
    \sup_{x\in \Bbf_1 \cap \pi_1} \dist (x, \pi_2) 
    \leq C(m,n, c_0, \mu) D\, .
    \end{equation}
    The simple argument is essentially contained in \cite{DMS}*{Proof of Lemma 10.6} and we repeat it for the reader's convenience. We just need to find $m$ vectors $v_1, \ldots , v_m\in E$ which satisfy a quantitative form of linear independence, namely such that 
    \begin{itemize}
    \item For every $v\in \pi_1\cap \Bbf_1$ there are $\lambda_1, \ldots, \lambda_m$ with $|\lambda_i|\leq C (m,n, c_0)$ such that $v=\sum_i \lambda_i v_i$. 
    \end{itemize}
    In order to find $v_1, \ldots, v_m$ we can argue inductively as follows. We fix $v_1= p$. We then look at the cylinder $\Gamma_1 := \{y:|y- (y \cdot \frac{v_1}{|v_1|}) \frac{v_1}{|v_1|}| < \gamma_1\}$ of radius $\gamma_1$ and axis $\{\lambda v_1: \lambda \in \mathbb R\}$. Choose $\gamma_1=\gamma_1 (c_0, m,\mu)$ so that $\mathcal{H}^m (E\setminus \Gamma_1)\geq \frac{\mu}{2} \mathcal{H}^m (\Bbf_{c_0} (p)\cap \Hbf_1)$ and we pick a point $v_2\in E\setminus \Gamma_1$. We then denote by $W$ the $2$-dimensional space generated by $v_1$ and $v_2$, then set $\Gamma_2 := |y-\mathbf{p}_W (y)|<\gamma_2\}$ to be the $\gamma_2$-neighborhood of $W$, then choose $\gamma_2=\gamma_2 (c_0,m,\mu)$ so that $\mathcal{H}^m (E\setminus \Gamma_2) \geq \frac{\mu}{4} \mathcal{H}^m (\Bbf_{c_0} (p)\cap \Hbf_1)$, and then pick a point $v_3\in E\setminus\Gamma_2$. We then repeat this process inductively to find $v_1,\dotsc,v_m$. It is easy to check that these vectors satisfy the desired quantitative form of linear independence, which in turns gives \eqref{e:p-close}.

    Having reached \eqref{e:p-close}, choose an orthonormal base $e_1, \ldots , e_{m-1}$ of $V_1$ and an orthogonal unit vector $e_m \in V_1^\perp \cap \Hbf_1$. Either $\mathbf{p}_{\pi_2} (e_i)\in \Hbf_2$, or $- \mathbf{p}_{\pi_2} (e_i) \in \Hbf_2$. Note that $-e_i\in V_1$ for all $i\leq m-1$. Hence without loss of generality we can, if needed, flip the sign of a subset of the first $m-1$ vectors so that $\mathbf{p}_{\pi_2} (e_i)\in \Hbf_2$ for every $i\in \{1, \ldots , m-1\}$. Obviously we cannot flip the sign of $e_m$, since $-e_m$ does not lie in $\Hbf_1$. However, by \eqref{e:p-close} we do find that 
    \begin{equation}\label{e:flip}
    \min \Big\{ \sup_{y\in \Hbf_1\cap \Bbf_1} \dist (y, \Hbf_2), \sup_{y\in (\pi_1\setminus\Hbf_1)\cap \Bbf_1} \dist (y, \Hbf_2) \Big\}
    \leq C D\, .
    \end{equation}
    The proof is thus complete once we show that, for a sufficiently small choice of $\eta$ in \eqref{e:smallness-D-eta} the minimum in the left hand side of \eqref{e:flip} is achieved by $\sup_{y\in \Hbf_1\cap \Bbf_1} \dist (y, \Hbf_2)$. This easily follows by contradiction. Fix indeed a sequence of $\eta_k$ converging to $0$ and pairs of half-planes $\Hbf_1^k$, $\Hbf_2^k$ for which the minimum in \eqref{e:flip} is attained by $\sup_{y\in (\pi^k_1\setminus\Hbf^k_1)\cap \Bbf_1} \dist (y, \Hbf^k_2)$. We can assume, up to a rotation, that the second plane $\Hbf^k_2$ is a fixed one $\Hbf_2$. The other members of the pair, denoted by $\Hbf_1^k$, will then have the property that their reflections along the respective spines, namely the half-planes $\pi^k_1\setminus \Hbf_1^k$, converge locally in Hausdorff distance to $\Hbf_2$. But recall that, by our assumption \eqref{e:smallness-D-eta} with $\eta = \eta_k\downarrow 0$, we have a sequence of points $p'_k\in\Hbf_1^k \cap \Bbf_1$ which are at distance at least $c_0$ from $\pi_1^k\setminus \Hbf_1^k$ and are converging to $\Hbf_2$ as well (in fact, the distance of the corresponding sets $E_k$ to $\Hbf_2$ must be converging to $0$). This is not possible and thus completes the proof.  \end{proof}

    \begin{proof}[Proof of Lemma \ref{l:spine-tilt-estimate}]
    Fix $r_0>0$. Let us begin with \eqref{e:1-homog}. The proof follows by very similar reasoning to that of Lemma \ref{l:rescalings-decay}. First of all, let $\eps\leq\eps_3$, where $\eps_3$ is the threshold in Lemma \ref{l:planar-excess-zeta}. Lemma \ref{l:planar-excess-zeta}(a) with hypothesis (iii) gives
    \begin{align*}
    \Ebf^p(T,\Bbf_{r_0}) & \leq C\hat{\Ebf}(T,\Sbf,\Bbf_{r_0}) + Cr_0^{-2}\max_{i<j}\dist^2(\pi_i\cap \Bbf_{r_0},\pi_j\cap \Bbf_{r_0})\\
    & \leq Cr_0^{-m-2}\Ebf(T,\Sbf,\Bbf_1) + C\boldsymbol{\zeta}(\Sbf)^2\\
    & \leq C(1+\eps^2r_0^{-m-2})\Ebf^p(T,\Bbf_1)
    \end{align*}
    Thus, provided that $\eps\leq r_0^{m+2}$, the right-hand inequality of \eqref{e:1-homog} is verified. To see the left-hand inequality, we argue by contradiction. Suppose that the inequality fails for some sufficiently large constant $C_*=C_*(q,m,n,\bar n)>0$ (to be determined); this yields sequences $T_k$, $\Sigma_k$ and $\Abf_k$ and $\Sbf_k=\Hbf^k_1\cup\cdots\cup\Hbf^k_{N_k} \in \Bscr^q(0)\setminus \Pscr(0)$ satisfying hypotheses (i)-(iii) with thresholds $\eps_k \downarrow 0$, such that up to subsequence we have
    \begin{equation}\label{e:1-homog-contradiction}
        \lim_{k\to\infty} \frac{\Ebf^p(T_k,\Bbf_{r_0})}{\Ebf^p(T_k,\Bbf_1)} < C_*^{-1}.
    \end{equation}
    Up to rotation, we may without loss of generality assume that the infimum in $\Ebf^p(T_k,\Bbf_1)$ is realized by the same plane $\pi_\infty$. Moreover, up to a further subsequence, we may assume that $N_k\equiv N\leq q$ for each $k$. 
    
    First of all, if $\liminf_{k\to\infty}\Ebf^p(T_k,\Bbf_1) > 0$, then up to another subsequence, we have
    \[
        T_k \toweakstar T_\infty, \quad \Sbf_k\to \Sbf_\infty\in \Bscr^q(0)\setminus \Pscr(0),
    \]
    with $\spt(T_\infty)\cap\Bbf_1=\Sbf_\infty\cap\Bbf_1$. Thus, $T_\infty$ is a non-planar area-minimizing cone mod$(q)$. This immediately implies, from the homogeneous structure, that
    \[
        \lim_{k\to\infty} \frac{\Ebf^p(T_k,\Bbf_{r_0})}{\Ebf^p(T_k,\Bbf_1)} = 1, 
    \]
    contradicting \eqref{e:1-homog-contradiction} for any choice of $C_*\geq 1$.

    Now suppose that $\Ebf^p(T_k,\Bbf_1) \to 0$; this allows us to apply the strong Lipschitz approximation theorem \cite{DLHMS}*{Theorem 15.1} for $T_k\mres\Cbf_{1/2}(0,\pi_\infty)$ for $k$ sufficiently large, to obtain Lipschitz functions $f_k: B_{1/2}(\pi_\infty)\to\Ascr_Q(\pi_\infty^\perp)$. The hypothesis (iii) allows us to do the same for the cones $\Sbf_k$, describing them as a superposition of graphs of maps $L_1^k,\dots,L_N^k:\pi_\infty\to\pi_\infty^\perp$. Normalizing by $E_k^{1/2}\coloneqq \Ebf^p(T_k,\Bbf_1)^{1/2}$ and arguing exactly as in the proof of Lemma \ref{l:rescalings-decay}, we obtain the desired contradiction.

\textbf{Proof of \eqref{e:reduction-1} and \eqref{e:reduction-2}: reduction to pruned open book.}
    We will now proceed to establish \eqref{e:reduction-1} and \eqref{e:reduction-2}. First of all, we claim that we may assume
    \begin{align}
        \Abf^2 + \Ebb(T,\Sbf,\Bbf_1)&\leq \bar\delta^2\boldsymbol\sigma(\Sbf)^2 \label{e:reduction-assumption-1} \\
        r_0^2\Abf^2 + \Ebb(T,\Sbf',\Bbf_{r_0})&\leq \bar\delta^2\boldsymbol\sigma(\Sbf')^2,\label{e:reduction-assumption-2}
    \end{align}
    for a suitably small choice of constant $\bar\delta=\bar\delta(q,m,n,\bar n, r_0)>0$. In order to do this, we will in fact find a constant $\bar C=\bar C(q,m,n,\bar n, \bar\delta)>0$ and a cone $\Sbf_1\in\Bscr^q(0)$ satisfying \eqref{e:reduction-assumption-1} such that
    \begin{itemize}
        \item[(ia)] $V(\Sbf_1) = V(\Sbf)$;
        \item[(iia)] $\Ebb(T,\Sbf_1,\Bbf_1)\leq\bar C(\Ebb(T,\Sbf,\Bbf_1) + \Abf^2)$;
        \item[(iiia)] $\dist^2(\Sbf_1\cap\Bbf_1,\Sbf\cap\Bbf_1)\leq\bar C (\Ebb(T,\Sbf,\Bbf_1) + \Abf^2)$,
    \end{itemize}
    and a cone $\Sbf'_1$ satisfying \eqref{e:reduction-assumption-2} such that (ia)-(iiia) holds for $T_{0,r_0}$ and $\Sbf'$ in place of $T$ and $\Sbf$ respectively. 

    Once we accomplish this, one may check that the assumptions of the lemma hold for $T, \Sbf_1, \Sbf'_1$ in place of $T,\Sbf,\Sbf'$ (up to a geometric constant factor), and that the validity of the conclusions \eqref{e:reduction-1} and \eqref{e:reduction-2} for $T,\Sbf_1, \Sbf'_1$ in turn imply the conclusions for $T,\Sbf,\Sbf'$.

    Now we proceed with the proof of the reduction; observe that it suffices to demonstrate the existence of $\Sbf_1$, since $\Sbf'_1$ is then found by applying the same reasoning to $T_{0,r_0}$ and $\Sbf'$ and taking $\bar\delta$ to lie below both thresholds. We wish to apply Lemma \ref{l:planar-excess-zeta} to $T,\Sigma, \Abf$ and $\Sbf$. Note, however, that the hypothesis $\Abf^2 \leq \eps_3^2\Ebb(T, \Sbf,\Bbf_1)$ is not guaranteed here (where $\eps_3=\eps_3(q,m,n,\bar n,\bar\delta)$ is the parameter in Lemma \ref{l:planar-excess-zeta}). Thus, we will first show that if $\Abf^2 > \eps_3^2\Ebb(T, \Sbf,\Bbf_1)$, then there exists $\Sbf_e\in \Bscr^q(0)$ with $V(\Sbf_e) = V(\Sbf)$ such that
    \begin{equation}\label{e:continuity}
        \Abf^2 = \eps_3^2\Ebb(T,\Sbf_e,\Bbf_1).
    \end{equation}
    Observe that the hypothesis (iii) of the statement of the Lemma guarantees that, as long as $\eps< \eps_3$, for any $\varpi\in \Pscr(0)$ containing $V(\Sbf)$ we have
    \[
        \Abf^2< \eps_3^2\hat\Ebf(T,\varpi,\Bbf_1).
    \]
    Now fix such a plane $\varpi$ and construct a 1-parameter family of open books $\Sbf(t)$, $t\in [0,1]$, with $\Sbf(0)= \Sbf$, $\Sbf(1) = \varpi$, so that $t\mapsto \Sbf (t)\cap \overline{\Bbf}_1$ is continuous in Hausdorff distance in $\Bbf_1$, $\Sbf(t)$ consists of $N$ distinct halfplanes for each $t\in [0,1)$, and $V(\Sbf(t)) = V(\Sbf)$ for all $t\in [0,1)$. In particular, we have continuity in $t$ of $t\mapsto\Ebb(T,\Sbf(t),\Bbf_1)$ on $[0,1)$, while $\liminf_{t\uparrow 1} \Ebb(T,\Sbf(t),\Bbf_1) \geq \hat\Ebf(T,\varpi,\Bbf_1)$. The intermediate value theorem then gives the existence of $\Sbf_e\coloneqq \Sbf(t_0)$ for some $t_0\in (0,1)$ satisfying \eqref{e:continuity}.

    In the case $\Abf^2 \leq \eps_3^2\Ebb(T, \Sbf,\Bbf_1)$, simply let $\Sbf_e \coloneqq \Sbf$. In either case, we then have $\Abf^2\leq \eps^2_3\mathbb{E}(T,\Sbf,\Bbf_1)$. Fix $\bar\delta > 0$ (to be determined). For this choice of $\bar\delta$, we may now choose $\eps$ small enough in the hypotheses of the present lemma so that $T,\Sigma, \Abf$ and $\Sbf_e$ satisfy the hypotheses of Lemma \ref{l:planar-excess-zeta} for the parameter $\eps_3$ therein. This produces an open book $\Sbf_1= \tilde\Hbf_{1}\cup\cdots\cup\tilde\Hbf_{N'} \subset \Sbf_e\in \Bscr^q(0)$ satisfying properties \eqref{e:reduction-assumption-1}, (ia), and (iia). Indeed, in the case where $\Sbf_e\neq \Sbf$, the latter property follows from the fact that $\Ebb(T,\Sbf_e,\Bbf_1)=\eps_3^{-2} \Abf^2$.

    It remains to check property (iiia) in the case where $\Sbf_e\neq \Sbf$ (if $\Sbf_e = \Sbf$, (iiia) also follows from Lemma \ref{l:planar-excess-zeta}). In light of the validity of \eqref{e:reduction-assumption-1} and (iia), we may now choose $\bar\delta=\bar\delta(q,m,n,\bar n)>0$ small enough such that Lemma \ref{l:crude} and Proposition \ref{p:crude-approx} apply for $\rho=\eta=\frac{1}{32}$ to $T_{0,1/4}$ and $\Sbf_1$. We will show that
    \begin{equation}\label{e:reduction-iv}
        \dist^2(\Sbf_1\cap\Bbf_1,\Sbf\cap\Bbf_1)\leq C(\Ebb(T,\Sbf_1,\Bbf_1)+\Ebb(T,\Sbf,\Bbf_1)+\Abf^2).
    \end{equation}
    To this end, we will first show that for each $i\in\{1,\dots,N'\}$, there exists $j_*\in \{1,\dots,N\}$ such that
    \begin{equation}\label{e:dist-books}
        \dist^2(\tilde\Hbf_i\cap\Bbf_1,\Hbf_{j_*}\cap\Bbf_1) \leq C(\Ebb(T,\Sbf_1,\Bbf_1)+\Ebb(T,\Sbf,\Bbf_1)+\Abf^2),
    \end{equation}
    for some $C=C(q,m,n,\bar n,\bar\delta)>0$. Firstly, notice that for $W_i\subset \Bbf_4$, $i=1,\dots,N'$, as given by Lemma \ref{l:crude}, the current $T\mres(\Bbf_{3/4}\setminus \Bbf_{1/32}(V(\Sbf_1))$ is supported in $\cup_i\tilde W_i$, where $\tilde W_i\coloneqq \iota_{0,4}(W_{i})$. Letting $T_i\coloneqq T\mres \tilde W_i$, we thus conclude from Proposition \ref{p:crude-approx} that
    \begin{equation}\label{e:reduction-heightbd}
        \dist^2(p,\tilde\Hbf_i) \leq C(\Ebb(T,\Sbf_1,\Bbf_1) +\Abf^2) \qquad \forall p \in \spt(T_i).
    \end{equation}
    Now fix $i\in\{1,\dots,N'\}$ and let $e$ denote the unique unit vector in $V(\Sbf_1)^\perp\cap\tilde\Hbf_i$. Let $\xi\coloneqq \frac{e}{4}$ and let $B_i\coloneqq B_{1/32}(\xi,\tilde\pi_i)\subset(\Bbf_{3/4}\cap\tilde\Hbf_i)\setminus \bar\Bbf_{1/64}(V(\Sbf_1))$, where $\tilde\pi_i$ is the $m$-dimensional plane that is the extension of $\tilde\Hbf_i$ in the usual manner. Now conclusions (b) and (f) of Proposition \ref{p:crude-approx} tell us that
    \[
        ((\mathbf{p}_{\tilde\Hbf_i})_\sharp T_i)\mres (\Bbf_{3/4}\cap\tilde\Hbf_i)\setminus \bar\Bbf_{1/64}(V(\Sbf_1)) = Q_i \llbracket (\Bbf_{3/4}\cap\tilde\Hbf_i)\setminus \bar\Bbf_{1/64}(V(\Sbf_1)) \rrbracket
    \]
    for an integer $Q_i\in \{1,\dots,q\}$. This in turn yields
    \begin{equation}
        \|T_i\|(\mathbf{p}^{-1}_{\tilde\Hbf_i}(B_i)) \geq c_0(m,n,\bar n) > 0.
    \end{equation}
    Proceeding via Chebyshev's inequality with the measure $\|T_i\|$, analogously to that in the proof of \cite{DMS}*{Lemma 10.6}, we deduce that for $\tilde C\coloneqq C_*/c_0$, where $C_*>1$ is to be determined, if
    \[
        E_i \coloneqq \{p\in \mathbf{p}_{\tilde\Hbf_i}^{-1}(B_i)\cap\spt(T_i) : \dist^2(p,\Sbf)\leq \tilde C\Ebb(T,\Sbf,\Bbf_1)\},
    \]
    then we have $\|T_i\|(E_i)\geq(1-1/C_*)\|T_i\|(\mathbf{p}_{\tilde\Hbf_i}^{-1}(B_i))$. This combined with Proposition \ref{p:crude-approx}(d) and (iia), we deduce that, up to further decreasing $\bar\delta$ if necessary, and for a suitable choice of $C_*$,
    \begin{equation}\label{e:measure-lb}
        \Hcal^m(\mathbf{p}_{\tilde\Hbf_i}(E_i))\geq \frac{1}{2}\Hcal^m(B_i).
    \end{equation}
    We refer the reader to \cite{DMS} for more details. Fix $w_0\in \mathbf{p}_{\tilde\Hbf_i}(E_i)$, let $p_0\in E_i$ be such that $\mathbf{p}_{\tilde\Hbf_i}(p_0)=w_0$, and let $j_*\in\{1,\dots,N\}$ be an index such that $\dist(p_0,\Hbf_{j_*}) = \dist(p_0,\Sbf)$. Then, invoking the definition of $E_i$ and \eqref{e:reduction-heightbd} we have
    \begin{align*}
        |\mathbf{p}^\perp_{\Hbf_{j_*}}(w_0)|^2 &=\dist^2(w_0,\Hbf_{j_*}) \leq 2\dist^2(w_0,p_0) + 2\dist^2(p_0,\Hbf_{j_*}) \\
        &= 2\dist^2(w_0,p_0) + 2\dist^2(p_0,\Sbf) \\
        &\leq C(\Ebb(T,\Sbf_1,\Bbf_1) + \Ebb(T,\Sbf,\Bbf_1) +\Abf^2),
    \end{align*}
    for some $C=C(q,m,n,\bar n,\bar\delta)>0$. We may now complete $w_0$ to a basis of $\tilde\pi_i$, and use the linearity of $\mathbf{p}^\perp_{\Hbf_{j_*}}$ (as $V(\Sbf_1) = V(\Sbf)$) to deduce \eqref{e:dist-books}.

    To conclude the validity of \eqref{e:reduction-iv}, it remains to check that for each $j\in \{1,\dots,N\}$, there exists $i_*\in \{1,\dots,N'\}$ such that \eqref{e:dist-books} holds with $i_*,j$ in place of $i,j_*$. Now, for $j$ fixed, we let $f$ be the unit vector in $V(\Sbf)^\perp\cap \Hbf_j$. Let $\zeta=\frac{f}{4}$ and $B_j\coloneqq B_{1/32}(\zeta,\pi_j)\subset(\Bbf_{3/4}\cap\Hbf_j)\setminus \bar\Bbf_{1/64}(V(\Sbf))$, where $\pi_j$ is the $m$-dimensional plane that is the extension of $\Hbf_j$. Again applying Chebyshev's inequality, only now with $\Hcal^m\mres B_j$, we deduce that the set
    \[
        F_j\coloneqq \{q\in B_j : \dist^2(q,\spt(T))\leq (2/\Hcal^m(B_j))\Ebb(T,\Sbf,\Bbf_1)\}
    \]
    satisfies $\Hcal^m(F_j)\geq \frac{1}{2}\Hcal^m(B_j)$. Fix $z_0\in F_j$ and let $p(z_0)\in \spt(T)$ be such that 
    \[
    \dist(z_0,\spt(T))=|z_0-p(z_0)|\, .
    \]
    The definition of $F_j$ guarantees that for $\eps$ sufficiently small (below a threshold depending only on $q,m$), $|z_0-p(z_0)|\leq \frac{1}{16}$, which in turn implies that $p(z_0)\in (\Bbf_{3/4}\cap\Hbf_j)\setminus \bar\Bbf_{1/64}(V(\Sbf_1))$ also. Thus, Lemma \ref{l:crude}(b) tells us that $p(z_0)\in \tilde W_{i_*}$ for some $i_*\in \{1,\dots,N'\}$. Proceeding analogously to above, we thus have
    \begin{align*}
        |\mathbf{p}^\perp_{\tilde\Hbf_{i_*}}(z_0)|^2 &= \dist^2(z_0,\tilde\Hbf_{i_*}) \leq 2\dist^2(z_0,p(z_0)) + 2\dist^2(p(z_0),\tilde\Hbf_{i_*}) \\
        &= 2\dist^2(z_0,\spt(T)) + C(\Ebb(T,\Sbf_1,\Bbf_1)+\Abf^2) \\
        &\leq C(\Ebb(T,\Sbf,\Bbf_1)+\Ebb(T,\Sbf_1,\Bbf_1)+\Abf^2).
    \end{align*}
    So indeed we have shown that
    \[
        \dist^2(\tilde\Hbf_{i_*}\cap\Bbf_1,\Hbf_j\cap\Bbf_1) \leq C(\Ebb(T,\Sbf,\Bbf_1)+\Ebb(T,\Sbf_1,\Bbf_1)+\Abf^2).
    \]
    Together with \eqref{e:dist-books}, we arrive at \eqref{e:reduction-iv}. When combined with (iia), this gives the conclusion of (iiia).
    
    Relabelling $\Sbf_1$, $\Sbf'_1$ to be $\Sbf$, $\Sbf'$ respectively, we may now work under the assumption that \eqref{e:reduction-assumption-1} and \eqref{e:reduction-assumption-2} hold.
    
    \textbf{Proof of \eqref{e:reduction-1} and \eqref{e:reduction-2} for pruned open book.}
    We will use analogous reasoning to that in the proof of property (iiia) above, but now we have to take into account the fact that we possibly have $V(\Sbf')\neq V(\Sbf)$. This is however compensated by the validity of both \eqref{e:reduction-assumption-1} and \eqref{e:reduction-assumption-2}, whereas before we only knew this for one of the open books in consideration.

    Notice also that by Lemma \ref{l:planar-excess-zeta}, which we applied to produce our relabelled open books $\Sbf$ and $\Sbf'$, we have
    $$C^{-1}\Ebf^p(T,\Bbf_1) \leq \boldsymbol{\zeta}(\Sbf)^2,\boldsymbol{\zeta}(\Sbf')^2\leq C\Ebf^p(T,\Bbf_1)$$
    and so \eqref{e:reduction-2} follows from \eqref{e:reduction-1} and Lemma \ref{l:spine-comparison}, and thus we may just focus on proving \eqref{e:reduction-1}.

    For $\bar\delta$ suitably small (with the same dependencies), we may apply Lemma \ref{l:crude}, Proposition \ref{p:crude-approx} and Lemma \ref{l:matching-Q} to the respective pairs $T_{0,1/4}$, $\Sbf$ and $T_{0,r_0/4}$, $\Sbf'$ with $\eta=\rho=\frac{r_0}{32}$. Relabelling the appropriate rescalings of the sets $W_i$ therein to still be denoted by $W_i$ for simplicity of notation, this gives
    \begin{itemize}
        \item[(1)] Pairwise disjoint neighborhoods $W_i$ around $(\Hbf_i\cap\Bbf_{1/2})\setminus \bar\Bbf_{r_0/64}$, corresponding currents $T_i\coloneqq T\mres W_i$ and constants $C=C(q,m,n,\bar n) > 0$, $\bar C= \bar C(q,m,n,\bar n, r_0)>0$ with
        \begin{align}
            &\spt(T)\cap\Bbf_{1/2}\setminus \Bbf_{r_0/64}(V(\Sbf)) \subset \bigcup_i W_i; \label{e:dist-control-S-1} \\
            &\dist^2(p,\Hbf_i) \leq C(\Ebb(T,{\Sbf},\Bbf_1) +\Abf^2) \notag \\
            &\qquad\qquad\qquad\leq \bar C(\Ebb(T,\Sbf,\Bbf_1) +\Abf^2 r_0^2)r_0^2 \qquad \forall p \in \spt(T_i)\cap\Bbf_{r_0}; \label{e:dist-control-S-2} \\
            &(\mathbf{p}_{\Hbf_i})_\sharp (T_i\cap \mathbf{p}_{\Hbf_i}^{-1}((\Bbf_{r_0}\cap\Hbf_i)\setminus \Bbf_{r_0/64}(V(\Sbf)))) = Q_i\llbracket (\Bbf_{r_0}\cap\Hbf_i)\setminus \Bbf_{r_0/64}(V(\Sbf))\rrbracket; \label{e:dist-control-S-3}
        \end{align}
        \item[(2)] Pairwise disjoint neighborhoods $W'_i$ around $(\Hbf'_i\cap\Bbf_{r_0/2})\setminus \bar\Bbf_{r_0/64}$, corresponding currents $T'_i\coloneqq T\mres W'_i$ and a constant $C=C(q,m,n,\bar n) > 0$ with
        \begin{align}
            &\spt(T)\cap\Bbf_{r_0/2}\setminus \Bbf_{r_0/64}(V(\Sbf')) \subset \bigcup_i W'_i; \label{e:dist-control-S'-1} \\
            &\dist^2(p,\Hbf'_i) \leq C(\Ebb(T,\Sbf',\Bbf_{r_0}) +\Abf^2 r_0^2)r_0^2 \qquad \forall p \in \spt(T'_i)\cap\Bbf_{r_0}; \label{e:dist-control-S'-2} \\
            &(\mathbf{p}_{\Hbf'_i})_\sharp (T'_i\cap \mathbf{p}_{\Hbf'_i}^{-1}((\Bbf_{r_0}\cap\Hbf_i)\setminus \Bbf_{r_0/64}(V(\Sbf')))) = Q_i\llbracket (\Bbf_{r_0}\cap\Hbf'_i)\setminus \Bbf_{r_0/64}(V(\Sbf'))\rrbracket; \label{e:dist-control-S'-3}
        \end{align}
    \end{itemize}
    Once again, we wish show that for each $i\in \{1,\dots,N\}$, there exists $j\in \{1,\dots, M\}$ such that
    \begin{equation}\label{e:dist-control-S-S'}
        \dist^2(\Hbf_i\cap\Bbf_1,\Hbf'_j\cap\Bbf_1) \leq \bar C (\Ebf(T,\Sbf,\Bbf_1) + \Ebf(T,\Sbf',\Bbf_1) + \Abf^2),
    \end{equation}
    for a constant $\bar C=\bar C(q,m,n,\bar n,r_0)>0$. This time, unlike in the corresponding argument for $\Sbf$ and $\Sbf_1$ above, note that the argument and conclusion is symmetric in $\Sbf$ and $\Sbf'$. Thus the validity of \eqref{e:dist-control-S-S'} will conclude the proof of \eqref{e:reduction-1}.

    Now fix $i$ and let $e$ be a vector of length $r_0/4$ in $\Hbf_i$ such that $\Bbf_{2\lambda_0 r_0}(e)\subset \Bbf_{r_0/2}\setminus (\Bbf_{r_0/64}(V(\Sbf))\cup\Bbf_{r_0/64}(V(\Sbf'))$, for some $\lambda_0=\lambda_0(m)\in(0,1/2]$. Proceeding as in \cite{DMS}, for each $z\in \Bbf_{\lambda_0 r_0}(e)\cap\Hbf_i$, there exists $\Hbf'_{j(z)}$ such that
    \[
        \dist^2(z,\Hbf'_{j(z)})\leq C(\Ebb(T,\Sbf,\Bbf_1) + \Abf^2 r_0^2 + \Ebb(T,\Sbf',\Bbf))r_0^2,
    \]
    for some $C=C(q,m,n,\bar n, r_0)>0$. In particular since $M\leq q$, there exists $E\subset \Bbf_{\lambda_0 r_0}(e)\cap\Hbf_i$ with $\Hcal^m(E)\geq \frac{1}{2q}\Hcal^m(\Bbf_{\lambda_0 r_0}(e)\cap\Hbf_i)$ and $j_0\in \{1,\dots, M\}$ such that for each $z\in E$ we have
    \begin{equation}
        \dist^2(z,\Hbf'_{j_0})\leq C(\Ebb(T,\Sbf,\Bbf_1) + \Abf^2 r_0^2 + \Ebb(T,\Sbf',\Bbf))r_0^2.
    \end{equation}
    Now we may apply Lemma \ref{l:disks-planes} with the planes $\Hbf_i \ni p\coloneqq e$, $\Hbf'_1$, $c_0= r_0/2$ and this choice of $E$ to conclude \eqref{e:dist-control-S-S'}. Note that the dependencies of the resulting constant will be on $q,m,n,\bar n$ and $r_0$. 
\end{proof}

\section{Estimates at the spine}\label{s:spine-est}
In this section we collect the necessary estimates required at the spine of a given open book $\Sbf\in \Bscr^q$, which are the mod$(q)$ analogues of the results in \cite{DMS}*{Section 11} and are an adaptation of Simon's work \cite{Simon_cylindrical} and Wickramasekera's work \cite{W14_annals} to this framework. This will be necessary to ensure that our blow-up sequence of graphs relative to a sequence of open books, thus far constructed away from a neighborhood of the spines, converge all the way up to the spine to a Dir-minimizer (along which it has an $(m-1)$-dimensional subspace of $Q$-points).

We omit the proofs here, since they follow in exactly the same way as the corresponding proofs in \cite{DMS}, together with our version of the pruning lemma (Lemma \ref{l:pruning}), as well as the graphical estimates in Section \ref{s:approx} taking the place of those in \cite{DMS}*{Section 8}, and Lemma \ref{l:shift} applied in place of \cite{DMS}*{Lemma 7.14}. Note that the monotonicity formula for mass ratios still applies here, to the stationary varifold associated to $T$, and since $Q=\frac{q}{2}$ the estimate \cite{DMS}*{(11.7)} is replaced by
\[
    \sum_i Q_i \Hcal^m(\Hbf_i\cap\Bbf_\rho) = Q\omega_m\rho^m \leq \omega_m\Theta(T,0)\rho^m,
\]
where $Q_i$ are the multiplicities given by Proposition \ref{p:coherent}.

We make the following assumption throughout this section.

\begin{assumption}\label{a:Simons}
    $T$, $\Sigma$, $\Abf$ and $q$ are as in Assumption \ref{a:main} and $\|T\|(\Bbf_4)\leq 4^m\omega_m(Q+\frac{1}{4})$ for $Q=\frac{q}{2}$. $\Sbf=\Hbf_1\cup\cdots\cup\Hbf_N\in\Bscr^q$ is an open book and $V=V(\Sbf)$. For $\eps=\eps(q,m,n,\bar n)>0$ smaller than the $\eps$-threshold in Assumption \ref{a:refined}, to be determined as the minimum of the thresholds in Theorem \ref{t:Simon-monotonicity-error}, Corollary \ref{c:Simon-non-conc}, and Proposition \ref{p:Simon-shift}, assume that
    \begin{equation}
        \Ebb(T,\Sbf,\Bbf_4)+\Abf^2\leq\eps^2\boldsymbol\sigma(\Sbf)^2.
    \end{equation}
\end{assumption}
We will henceforth use the following notation. Let $p\in \spt(T)$ be a point at which the approximate tangent plane $\pi(p)$ oriented by a simple $m$-vector $\vec{T}(p)$. We let $\mathbf{p}_{\vec T}(p)$ and $\mathbf{p}_{\vec T}^\perp(p)$ denote the respective projections onto $\pi(p)$ and $\pi(p)^\perp$. Moreover, for simplicity we let $p_\parallel\coloneqq \mathbf{p}_{\vec T}(p)$ and $p^\perp\coloneqq \mathbf{p}_{\vec T}^\perp(p)$.

\begin{theorem}[Simon's error and gradient estimates]\label{t:Simon-monotonicity-error}
    There exists $C=C(q,m,n,\bar n)>0$ and $\eps=\eps(q,m,n,\bar n)>0$ such that the following holds. Suppose that $T, \Sigma, \Abf, \Sbf,q$ and $Q$ are as in Assumption \ref{a:Simons} with this choice of $\eps$, let $r_*=\frac{1}{4\sqrt{m-1}}$ and suppose that $\Theta(T,0)\geq Q$. Then
    \begin{align}
        \int_{\Bbf_{r_*}}\frac{|p^\perp|^2}{|p|^{m+2}}d\|T\|(p) &\leq C(\Abf^2 +\hat\Ebf(T,\Sbf,\Bbf_4)) \label{e:Simons-gradient-1} \\
        \int_{\Bbf_{r_*}}|\mathbf{p}_{V}\circ\mathbf{p}_{\vec T}^\perp|^2 d\|T\| &\leq C(\Abf^2 +\hat\Ebf(T,\Sbf,\Bbf_4)) \label{e:Simons-gradient-2}
    \end{align}
\end{theorem}

\begin{corollary}[Simon's non-concentration estimate]\label{c:Simon-non-conc}
    There exists $\eps = \eps(q,m,n,\bar{n})>0$ such that the following holds. Let $\kappa\in (0,m+2)$. Then, there is a constant $C_\kappa = C_\kappa(q,m,n,\bar{n},\kappa)>0$ such that if $T, \Sigma, \Abf, \Sbf,q, Q$ are as in Assumption \ref{a:Simons} with this choice of $\eps$, if $r_*$ is as in Theorem \ref{t:Simon-monotonicity-error} and moreover if $\Theta(T,0)\geq Q$, then
    \begin{equation}
        \int_{\Bbf_{r_*}} \frac{\dist^2(p,\Sbf)}{|p|^{m+2-\kappa}}\, d\|T\|(p) \leq C_\kappa(\Abf^2 +\hat\Ebf(T,\Sbf,\Bbf_4))
    \end{equation}
\end{corollary}

\begin{proposition}[Simon's shift inequality]\label{p:Simon-shift}
    There exist a radius $r_* = r_*(q,m,n,\bar{n})>0$, $\eps = \eps(q,m,n,\bar{n})>0$, and $C = C(q,m,n,\bar{n})>0$ such that the following holds. For each $\kappa\in (0,m+2)$, there is a constant $\bar{C}_\kappa = \bar{C}_\kappa(q,m,n,\bar{n},\kappa)>0$ such that if $T, \Sigma, \Abf, \Sbf,q, Q$ are as in Assumption \ref{a:Simons} with this choice of $\eps$, then for any $p_0\in \Bbf_{r_*}$ with $\Theta(T,p_0)\geq Q$ we have
    \begin{align}
        \int_{\Bbf_{4r_*}(p_0)}\frac{\dist^2(p,p_0+\Sbf)}{|p-p_0|^{m+2-\kappa}}d\|T\|(p) &\leq\bar C_\kappa (\Abf^2 + \hat\Ebf(T,\Sbf,\Bbf_4)) \label{e:shift-1}\\
        |\mathbf{p}_{\pi_1}^\perp(p_0)|^2 + \boldsymbol\zeta(\Sbf)^2|\mathbf{p}_{V^\perp\cap\pi_1}(p_0)|^2 &\leq C (\Abf^2 + \hat\Ebf(T,\Sbf,\Bbf_4)), \label{e:shift-2}
    \end{align}
    where the $m$-dimensional plane $\pi_1$ is the extension of the half-plane $\Hbf_1$.
\end{proposition}

\section{Linearization}
Before completing the blow-up procedure relative to an open book $\Sbf\in \Bscr^q(0)$, we first analyze the linearized problem, which will be the key to reaching a contradiction for the proof of Theorem \ref{t:decay}. A key role will be played by the boundary regularity theory developed for Dir-minimizing $Q$-valued functions in \cite{Hi}, which will be instrumental in proving the decay Lemma \ref{l:bdry-decay} below. 

In what follows, given a Lipschitz open set $\Omega\subset\R^m$ and a map $u\in W^{1,2} (\Omega, \mathcal{A}_Q (\mathbb R^{\bar n}))$ we will say that $u$ is \emph{Dir-minimizing in $\Omega$} if 
\[
\int_\Omega |Dv|^2 \geq \int_\Omega |Du|^2
\]
for every $v\in W^{1,2} (\Omega, \mathcal{A}_Q (\mathbb R^{\bar n}))$ which has the same trace as $u$ on $\partial \Omega$ (c.f. \cite{DLS_MAMS} for the relevant definition of the trace of a multi-valued Sobolev function). 
The key properties used here are the continuity of Dir-minimizers at the boundary under the assumption that the boundary data is in a suitable H\"older class and the monotonicity of the frequency function at boundary points. Before coming to the statement of the decay lemma we introduce the following terminology.

\begin{definition}\label{d:spaces}
We denote by $H^+\subset \mathbb R^m$ the open half space $\{x: x_1>0\}$, by $V$ the hyperplane $\{x_1=0\}$, and by $B_r^+\subset \mathbb R^m$ the half ball $H^+\cap B_r = \{x\in \mathbb R^m : |x|<r, x_1>0\}$. We then let:
\begin{itemize}
\item[(a)] $\mathscr{H}$ be the space of $W^{1,2}$ maps $u: B_1^+ \to \mathcal{A}_Q (\mathbb R^{\bar n})$ which are Dir-minimizing, whose trace on $B_1\cap V$ is identically $Q\llbracket 0 \rrbracket$ and with the property that the frequency $I_{u,p}(0)= \lim_{r\downarrow 0}I_{u,p}(r)$ of $u$ (see \eqref{e:bdry-freq} below, re-centered at $p$) at every boundary point $p\in V\cap B_1$ is at least $1$. 
\item[(b)] $\mathscr{L}$ be the space of maps $u: H^+\to \mathcal{A}_Q (\mathbb R^{\bar n})$ of the form 
\[
u (x) = \sum_i \llbracket L_i (x)\rrbracket\, ,
\]
where $L_1, \ldots , L_Q : \mathbb R^m \to \mathbb R^{\bar n}$ are linear maps which vanish on the hyperplane $V$. 
\end{itemize}
\end{definition}

The key decay lemma is then the following.

\begin{lemma}\label{l:bdry-decay}
    For every $Q$, $m$, $\bar n\in\mathbb{Z}_{\geq 1}$ and $\varepsilon>0$, there exists $\rho = \rho (Q, m, \bar n, \varepsilon)>0$ such that
\begin{equation}\label{e:linear-decay}
\min_{v\in \mathscr{L}} 
\int_{B_r^+} \mathcal{G} (u,v)^2 
\leq \varepsilon r^{m+2} \int_{B_1^+} |u|^2
\end{equation}
for every $u\in \mathscr{H}$ and for every $r\in (0,\rho]$. 
\end{lemma}

\begin{proof}
Following Almgren's celebrated computations on the monotonicity formula for the frequency, we introduce the quantities 
\[
H_u (r):= \int_{\partial B_r^+} |u|^2 \qquad \mbox{and} \qquad D_u (r) := \int_{B_r^+} |Du|^2
\]
and 
\begin{equation}\label{e:bdry-freq}
I_u (r):= \frac{rD_u(r)}{H_u(r)}, .
\end{equation}
We then claim that 
\[
I_u'(r) \geq 0
\]
and 
\begin{equation}\label{e:log-derivative}
\left.\frac{d}{d\tau}\right|_{\tau=r} \left[ \ln \left(\frac{H_u(\tau)}{\tau^{m-1}}\right)\right] = \frac{2 I_u(r)}{r}\, .
\end{equation}
In order to prove these identities we can argue as in \cite{DLS_MAMS}*{Theorem 3.15 \& Corollary 3.18}. In fact the arguments in there are based on the variational identities derived in \cite{DLS_MAMS}*{Proposition 3.2}, which results from taking first variations of the Dirichlet energy along certain specific one-parameter families of deformations of the function $u$. A simple inspection of the argument for \cite{DLS_MAMS}*{Proposition 3.2} shows that these deformations keep the boundary value of $u$ unchanged, due to the zero boundary data and so they are valid choices of variation. This shows that the same identities hold in our case as well. In turn, integrating \eqref{e:log-derivative} and using the monotonicity of $I$ shows that 
\[
\frac{H_u(r)}{r^{m-1}}\leq \left(\frac{r}{t}\right)^{2I_u(r)} \frac{H_u(t)}{t^{m-1}}
\]
for every $0<r\leq t \leq 1$. 

From the latter we also conclude that 
\begin{equation}\label{e:decay-1}
\int_{B_s^+} |u|^2 \leq \left(\frac{s}{t}\right)^{m+2 I_u(s)} \int_{B_t^+} |u|^2
\leq \left(\frac{s}{t}\right)^{m+2} \int_{B_t^+} |u|^2
\end{equation}
for every $0<s\leq t\leq 1$ and every $u\in \mathscr{H}$ (given that $1\leq I_u (0) = \lim_{r\downarrow 0} I_u (r)$). 
Finally, the same computations also lead to the conclusion that, if $I_u (r)$ has a constant value $\alpha$ for each $r>0$, then $u$ has to be radially $\alpha$-homogeneous.

Fix now $\varepsilon$ as in the statement. We claim that there is a $\delta>0$ such that, if $u\in \mathscr{H}$ and $I_u(1) \leq 1+\delta$, then
\begin{equation}\label{e:freq-near-1-decay}
\min_{v\in \mathscr{L}} \int_{B_{1/2}^+} \mathcal{G} (u,v)^2 \leq \varepsilon 2^{-m-2} \int_{B_1^+} |u|^2\, .
\end{equation}
Indeed assume the latter claim were false. Then we could find a sequence $\{u_k\}_k\subset \mathscr{H}$ for which $I_{u_k}(1) \leq 1+\frac{1}{k}$ but nonetheless 
\begin{equation}\label{e:contra}
\min_{v\in \mathscr{L}} \int_{B_{1/2}^+} \mathcal{G} (u_k,v)^2 \geq \varepsilon 2^{-m-2} \int_{B_1^+} |u_k|^2\, .
\end{equation}
Observe that by normalizing we can without loss of generality assume $\int_{B_1^+} |u_k|^2 =1$. Because of the upper frequency bound and monotonicity, it is immediate to see that, for every $s<1$, $\int_{B^+_s} |Du_k|^2$ is uniformly bounded in $k$. In particular up to extraction of a subsequence we can assume that $u_k$ converges strongly in $L^2 (B_s^+)$ to a function $u$ for each $0<s<1$. As it is shown in \cite{Hi}*{Proof of Proposition 3.3} (see also \cite{DDHM}*{Lemma 4.7}), the convergence is then strong in $W^{1,2} (B^+_s)$ for every $s<1$ and the function $u$ is in fact Dir-minimizing in $B^+_s$ and attains the boundary value $Q \llbracket 0 \rrbracket$ at $\{x_1=0\}\cap B_1$ (from convergence of traces, see e.g. \cite{DLS_MAMS}*{Proposition 2.10}). Moreover the upper semi-continuity of the frequency guarantees that $I_{u,x}(0)\geq 1$ for every $x\in \{x_1=0\}\cap B_1$. In particular, again exploiting the frequency monotonicity, we conclude that $I_u (r) = 1$ for all $0<r<1$, and thus the function $u$ must be radially $1$-homogeneous around the origin. But for the same reason it must be radially $1$-homogeneous around any other point $x\in \{x_1=0\}$, in turn implying easily that $u$ belongs to $\mathscr{L}$. On the other hand we have 
\[
\int_{B_{1/2}^+} \mathcal{G} (u_k,u)^2  \geq \varepsilon 2^{-m-2}
\]
by \eqref{e:contra}, contradicting the previously reached conclusion that $u_k$ converges strongly to $u$ in $L^2(B_{1/2}^+)$. 

Having found $\delta$ as above, let $\bar{\rho}$ be such that $\bar\rho^{2\delta} \leq \varepsilon$ and set $\rho:= \frac{\bar{\rho}}{2}$. Our claim is that \eqref{e:linear-decay} holds for this choice of $\rho$. Fix $r \leq \rho$ and let $\bar {r}=2r$. We then distinguish two possibilities. 

{\bf Case 1.} $I_u (\bar r) \geq 1+\delta$. We can then apply \eqref{e:decay-1} twice to conclude 
\begin{align*}
r^{-m-2} \int_{B^+_r} |u|^2 \leq \bar{r}^{-m-2} \int_{B^+_{\bar r}} |u|^2 \leq \bar{r}^{2\delta} \int_{B_1^+} |u|^2 \leq \bar{\rho}^{2\delta} \int_{B_1^+} |u|^2 \leq \varepsilon \int_{B_1^+} |u|^2\, .
\end{align*}
In particular \eqref{e:linear-decay} holds because the function identically equal to $Q\llbracket 0 \rrbracket$ belongs to $\mathscr{L}$.

{\bf Case 2.} $I_u (\bar r)\leq 1+\delta$. We can then consider the function $u_{\bar r} (x) = \bar r^{-1} u (\bar r x)$; observe that $u_{\bar r}$ belongs to $\mathscr{H}$. We thus know from \eqref{e:freq-near-1-decay} that there exists a function $v\in \mathscr{L}$ such that 
\[
2^{m+2} \int_{B_{1/2}^+} \mathcal{G} (u_{\bar r}, v)^2 \leq
\varepsilon \int_{B_1^+} |u_{\bar r}|^2\, .
\]
Changing variables we arrive at the inequality
\[
\frac{1}{r^{m+2}} \int_{B_r^+} \mathcal{G} (u,v)^2 \leq \frac{\varepsilon}{\bar{r}^{m+2}} \int_{B_{\bar r}^+} |u|^2\, .
\]
On the other hand we can use \eqref{e:decay-1} to get 
\[
\frac{1}{\bar{r}^{m+2}} \int_{B_{\bar r}^+} |u|^2 \leq \int_{B_1^+} |u|^2\, .
\]
Combining the last two inequalities we then reach \eqref{e:linear-decay}.
\end{proof}

We will also need the following lemma to extend the Dir-minimizing property from the interior of a half-plane to the boundary of the half-plane (i.e. to get a Dir-minimizing property for competitors with the same trace).

\begin{lemma}\label{l:technical-minimizing}
For every choice of positive integers $Q, m, n$ and for every bounded Lipschitz open set $\Omega\subset \mathbb R^m$ the following holds. A map $u\in W^{1,2} (\Omega, \mathcal{A}_Q (\mathbb R^{n}))$ is Dir-minimizing if and only if the inequality
\begin{equation}\label{e:Dir-minimizing}
\int_\Omega |Dv|^2 \geq \int_\Omega |Du|^2
\end{equation}
holds for every $v\in W^{1,2} (\Omega, \mathcal{A}_Q (\mathbb R^{n}))$ which coincides with $u$ in a neighborhood of $\partial \Omega$.
\end{lemma}
\begin{proof} Let $\Omega$ be a bounded Lipschitz open set and for $N\in \N$, fix a classical function $w\in W^{1,2} (\Omega, \mathbb R^N)$
with the property that $w|_{\partial \Omega} = 0$. We then claim that there is a sequence of functions $\{w_k\}\subset W^{1,2} (\Omega, \mathbb R^N)$ with the following properties:
\begin{itemize}
\item[(i)] $w_k$ vanishes identically in a neighborhood of $\partial \Omega$;
\item[(ii)] $\|w_k-w\|_{W^{1,2}}$ converges to $0$;
\item[(iii)] $\{w_k \neq w\}$ is contained in the $\frac{1}{k}$ neighborhood of $\partial \Omega$. 
\end{itemize}
This is in fact a simple exercise in classical Sobolev-space theory and we shall return to it at the end of the proof for the reader's convenience.
Armed with it, let us now show the conclusion of the lemma using Almgren's bi-Lipschitz embedding $\boldsymbol{\xi}$ of $\mathcal{A}_Q (\mathbb R^n)$ into $\mathbb R^{N (Q,n)}$ and the Lipschitz retraction $\boldsymbol{\rho}$ of $\mathbb R^{N (Q,n)}$ onto $\boldsymbol{\xi} (\mathcal{A}_Q (\mathbb R^n))$; for their definitions and properties we refer to \cite{DLS_MAMS}*{Section 2.1}. 

One implication in the lemma is clear, so we focus on the other. So assume that $v$ is a competitor with the property that $v|_{\partial \Omega} = u|_{\partial \Omega}$ for some map $u\in W^{1,2} (\Omega; \mathcal{A}_Q (\mathbb R^n))$ which has strictly smaller Dirichlet energy than $u$ on $\Omega$. We aim at constructing a similar competitor with the property that it coincides with $u$ in a neighborhood of $\partial \Omega$. First of all we consider the classical Sobolev map $w:= \boldsymbol{\xi} (v) - \boldsymbol{\xi} (u)$. We now construct functions $w_k$ satisfying (i), (ii), and (iii) above and hence we let 
\[
v_k := \boldsymbol{\xi}^{-1} \circ \boldsymbol{\rho} \circ (\boldsymbol{\xi} (u) + w_k)\, .
\]
Since $\boldsymbol{\rho}$ is the identity on $\boldsymbol{\xi} (\mathcal{A}_Q (\mathbb R^n))$, the map $v_k$ coincides with $u$ in a neighborhood of $\partial \Omega$. We next claim that 
\[
\lim_{k\to \infty} \int_\Omega |Dv_k|^2 = \int_\Omega |Dv|^2\, ,
\]
which would suffice to conclude the proof. Let $U_k:= \{x\in\Omega: \dist (x, \partial \Omega) < \frac{1}{k}\}$ and observe that $\{v_k\neq v\}\subset U_k$. It thus suffices to show that 
\[
\lim_{k\to \infty} \int_{U_k} |Dv_k|^2 = 0\, .
\]
But by \cite{DLS_MAMS}*{Theorem 2.4} and the Lipschitz regularity of $\boldsymbol{\rho}$ we get 
\[
\int_{U_k} |Dv_k|^2 \leq C \int_{U_k} |D (\boldsymbol{\xi} (u) + w_k)|^2 
\leq C \int_{U_k} |D (\boldsymbol{\xi} (u) + w)|^2 + C \int_{U_k} |D (w_k-w)|^2\, .
\]
On the other hand, given that $|D (\boldsymbol{\xi} (u) + w)|$ is a fixed $L^2$ function and the measure of $U_k$ converges to 0, the first summand tends to zero, while the second converges to zero because $\|w_k-w\|_{W^{1,2}}$ does.  

Coming to the existence of the functions $w_k$, recall first that, because of the boundedness and Lipschitz regularity of the open set $\Omega$, the function $w$ belongs in fact to $W^{1,2}_0 (\Omega)$, namely the closure of $C^\infty_c (\Omega)$ in the strong $W^{1,2}$ topology. In particular there certainly is a sequence $\{z_j\}\subset C^\infty_c (\Omega)$ with the property that $\|z_j - w\|_{W^{1,2}(\Omega)} \to 0$. Fix now $k\geq 2$ and consider the open sets $U_k$ (as defined above) and $\overline{U}_{k+1}^c \cap \Omega$ (i.e. the complement of the closure of $U_{k+1}$ in $\Omega$). The pair forms an open cover of $\Omega$ and we can thus find a smooth partition of unity $(\varphi_k, \psi_k)$ subordinate to it, namely $\spt (\varphi_k) \subset \overline{U}_{k+1}^c \cap \Omega$, $\spt (\psi_k)\subset U_{k}$, and $\varphi_k+\psi_k =1$ on $\Omega$. We set $z_{j,k} := \varphi_k w + \psi_k z_j$. It is then immediate to see that $z_{j,k}$ is identically equal to $w$ on the complement of $U_k$ and that $\lim_{j\to \infty} \|z_{j,k}-w\|_{W^{1,2}(\Omega)} = 0$. Therefore the sequence $w_k$ is achieved as a diagonal one of the form $z_{j(k), k}$ for a suitable choice of $j(k)\uparrow \infty$. 
\end{proof} 

\section{Final blow-up}\label{s:blowup}
We are now in a position to prove the reduced version of the main decay theorem (Theorem \ref{t:two-scale-decay}), which in turn implies Theorem \ref{t:decay}.

\subsection{Two decay regimes}
The conclusion of Theorem \ref{t:two-scale-decay} follows from the validity of either Proposition \ref{p:collapsed} or Proposition \ref{p:noncollapsed} below.

\begin{proposition}[Collapsed decay]\label{p:collapsed}
    Suppose that $m\geq 2$, $n\geq\bar{n}\geq 1$ are integers, let $T$, $\Sigma$, $\Abf$ and $q$ be as in Assumption \ref{a:main} with $\|T\|(\Bbf_1)\leq \omega_m(Q+\frac{1}{4})$ for $Q=\frac{q}{2}$. For every $\varsigma_1 >0$, there exists $\eps_c=\eps_c(q,m,n,\bar n, \varsigma_1)\in (0,\frac{1}{2}]$ and $r_c=r_c(q,m,n,\bar n, \varsigma)\in (0,\frac{1}{2}]$ such that the following holds. Suppose that there exists an open book $\Sbf\in \Bscr^q(0)$ satisfying
        \[
            \Abf^2 \leq \eps_c^2\Ebb(T,\Sbf,\Bbf_1) \leq \eps_c^4 \boldsymbol\sigma(\Sbf)^2,
        \]
        and $\boldsymbol{\zeta}(\Sbf)\leq \eps_c$. Then there exists $\Sbf'\in\Bscr^q(0)\setminus \Pscr(0)$ with
    \begin{equation}
        \Ebb(T,\Sbf',\Bbf_{r_c}) \leq \varsigma_1\Ebb(T,\Sbf,\Bbf_1).
    \end{equation}
\end{proposition}

\begin{proposition}[Non-collapsed decay]\label{p:noncollapsed}
    Suppose that $m\geq 2$, $n\geq \bar{n}\geq 1$ are integers, let $T$, $\Sigma$, $\Abf$ and $q$ be as in Assumption \ref{a:main} with $\|T\|(\Bbf_1)\leq \omega_m(Q+\frac{1}{4})$ for $Q=\frac{q}{2}$. Then for every $\eps^\star_c$, $\varsigma_1 >0$, there exists $\eps_{nc}=\eps_{nc}(q,m,n,\bar n, \varsigma_1,\eps^\star_c)\in (0,\frac{1}{2}]$ and $r_{nc}=r_{nc}(q,m,n,\bar n, \varsigma_1,\eps^\star_c)\in (0,\frac{1}{2}]$ such that the following holds. Suppose that there exists an open book $\Sbf\in \Bscr^q(0)$ satisfying
        \[
            \Abf^2 \leq \eps_{nc}^2\Ebb(T,\Sbf,\Bbf_1) \leq \eps_{nc}^4 \boldsymbol\sigma(\Sbf)^2,
        \]
        and $\boldsymbol{\zeta}(\Sbf)\geq \eps^\star_c$. Then there exists $\Sbf'\in\Bscr^q(0)\setminus \Pscr(0)$ with
    \begin{equation}
        \Ebb(T,\Sbf',\Bbf_{r_{nc}}) \leq \varsigma_1\Ebb(T,\Sbf,\Bbf_1).
    \end{equation}
\end{proposition}

Clearly, one can deduce Theorem \ref{t:two-scale-decay} from the two propositions above by letting $\varsigma_1>0$ be fixed arbitrarily, then taking $\eps_c^\star=\eps_c$ in Proposition \ref{p:noncollapsed}, followed by $\eps_1=\min\{\eps_c,\eps_{nc}\}$ and $\bar{r}_1=r_c$, $\bar{r}_2=r_{nc}$. Thus to prove Theorem \ref{t:decay}, we now just need to establish Proposition \ref{p:collapsed} and Proposition \ref{p:noncollapsed}.

Analogously to that in \cite{DMS}, we will proceed to verify Proposition \ref{p:collapsed} and Proposition \ref{p:noncollapsed} by a contradiction blow-up argument. The difference between the two blow-up regimes is that in Proposition \ref{p:collapsed}, we will take a sequence of parameters $\eps^k_c$ tending to zero, thus implying that the corresponding open books $\Sbf_k$ converge to a flat plane (with multiplicity). This means that for $k$ sufficiently large we can reparameterize the coherent outer approximations of Proposition \ref{p:coherent} over a single plane, which we can without loss of generality take to be the extension of some half-plane in $\Sbf_k$. After subtracting the  linear functions whose graphs are the half-planes in $\Sbf_k$, we can then perform our blow-up procedure (this procedure is analogous to the original idea of Wickramasekera \cite{W14_annals}, just in arbitrary codimension).

On the other hand, in Proposition \ref{p:noncollapsed}, we fix $\eps_c^\star>0$ while we take a sequence $\eps_{nc}^k$ tending to zero. Thus, even though the ratio
\[
    \frac{\Ebb(T_k,\Sbf_k,\Bbf_1)}{\boldsymbol\sigma(\Sbf_k)^2}
\]
converges to zero, the open books $\Sbf_k$ will stay bonded away from any single plane. In this case we can therefore perform a blow-up procedure by directly using the coherent outer approximations, rescaled suitably.

\subsection{Transversal coherent approximations}
We are now in a position to proceed as described above in the setting of Proposition \ref{p:collapsed}, and reparameterize the coherent outer approximation from Proposition \ref{p:coherent} to a single plane. To keep our notation simple, as a base plane for all graphical parametrization we will use the extension $\pi_1$ of the first page $\Hbf_1$ of the open book $\Sbf_k$ (we can of course without loss of generality rotate coordinates to assume that $\Hbf_1^k\equiv \Hbf_1$ is independent of $k$). 
In particular $\pi_1$ is the union of $\Hbf_1$ and its reflection $-\Hbf_1$ across the spine $V = V (\Sbf_k)$ (which also we can assume without loss of generality is independent of $k$ by rotation, and thus $\pi_1$ is also independent of $k$). We will use a similar shorthand notation for any set which is the reflection along $V$ of some other set $\Omega\subset \Hbf_1$. More precisely, such a reflected set will be denoted by $-\Omega$, while often the starting set will be denoted by $+\Omega$.

Following the notation in \cite{DMS}, given a multi-valued function $g=\sum_i \llbracket g_i \rrbracket$ and a single-valued function $f$ defined on the same domain, we let
\[
    g\ominus f\coloneqq \sum_i \llbracket g_i - f \rrbracket, \qquad g\oplus f\coloneqq \sum_i \llbracket g_i + f \rrbracket.
\]
Moreover, given $L\in \Gcal$, we recall the notation $\Nscr(L)$ for neighbors of $L$ and the notation $\bar\Ebf(L)$, both from Definition \ref{d:coherent}, and further define
\[
    \tilde\Ebf(L)\coloneqq\max\{\bar\Ebf(L'):L'\in \Nscr(L)\} = \max\{\bar\Ebf(L''):L''\in \Nscr(L'), \ L'\in \Nscr(L)\}.
\]
Let us now state the transversal coherent approximation result.

\begin{proposition}[Transversal coherent approximation]\label{p:transversal-coherent}
    Let $m\in\mathbb{N}_{\geq 2}$, $n,\bar{n}\in \mathbb{N}_{\geq 1}$, $q\in \N_{\geq 2}$ 
    There exists $c_0=c_0(m,n,\bar n)>0$, $C=C(q,m,n,\bar n)>0$ such that the following holds. Let $T$, $\Sigma$, $\Abf$, $\Sbf= \Hbf_1\cup\cdots\cup\Hbf_N$ be as in Assumption \ref{a:refined}, let $\pi_i$ be the planes that are the extensions of $\Hbf_i$, and let $\gamma$ be as in Proposition \ref{p:crude-approx}. Let $\bar\ell= \max\{k\in \N:\Gcal_{k+1}\subset\Gcal^o\}$ and for $i=1,\dots, N$ let $\tilde R^o_i = \Hbf_i \cap \bigcup_{L \in \Gcal_{\leq \bar\ell}} R(L) \subset R_i^o$. Let $u_i : R_i^o \to \Acal_{Q_i}(\Hbf_i^\perp)$ be the Lipschitz maps as in Proposition \ref{p:coherent} and let $\boldsymbol\zeta(\Sbf) \leq c_0$. Then 
    \begin{itemize}
        \item[(a)] Each half-plane $\Hbf_i$ is a graph of the restriction to $\pm\Hbf_1$ of a linear map $A_i:\pi_1 \to \pi_1^\perp$ that satisfies $|A_i|\leq C \boldsymbol\zeta(\Sbf)$ and $\ker(A_i)=V(\Sbf)$;
        \item[(b)] For each $i$ there is a choice of sign $\pm$ and a Lipschitz map $v_i:\pm\tilde R^o_1 \to \Acal_{Q_i}(\pi_1^\perp)$ with $\Gbf_{v_i} = \Gbf_{u_i} \mres\mathbf{p}_{\pi_1}^{-1}(\pm\tilde R^o_1)$;
        \item[(c)] Letting $v=\sum_{i=1}^N \llbracket v_i \rrbracket$, we have $v:\tilde R^o_1\cup(-\tilde R^o_1) \to \Acal_Q(\pi_1^\perp)$ after extending each $v_i$ by zero, and
        \begin{equation}
            \|v\|_{L^\infty(\tilde R^o_i)} + \|Dv\|_{L^2(\tilde R^o_i)} \leq C\boldsymbol{\zeta}(\Sbf);
        \end{equation}
        \item[(d)] Letting
        \[
            K^\pm \coloneqq \mathbf{p}_{\pm\Hbf_1}\big((\spt(T)\cap\Bbf_{1/2}\cap \mathbf{p}_{\pi_1}^{-1}(\pm\tilde R^o_1))\setminus\gr(v)\big)
        \]
        denote the region of non-graphicality in $\pm\tilde R^o_1$, recalling that $L_1=R(L)\cap\Hbf_1$, for each $L\in \Gcal_{\leq \bar\ell}$ we have
        \begin{equation}\label{e:transv-coherent-non-graph-set}
            |\pm L_1\setminus K^\pm| + \|T\|(\mathbf{p}_{\pi_1}^{-1}(\pm L_1\setminus K^\pm)) \leq C 2^{-m\ell(L)}(\tilde\Ebf(L)+2^{-2\ell(L)}\Abf^2)^{1+\gamma};
        \end{equation}
        \item[(e)] Letting $w_i\coloneqq v_i \ominus A_i : \tilde R^o_i \cup (-\tilde R^o_i) \to \Acal_{Q_i}(\pi_1^\perp)$, we have
        \begin{align}
            2^{2\ell(L)}\|w_i\|_{L^\infty(L_1)}^2 + 2^{m\ell(L)}\|D w_i\|_{L^2(L_1)}^2 &\leq C(\tilde\Ebf(L) + 2^{-2\ell(L)}\Abf^2) \\
            \Lip(w_i) &\leq C(\tilde\Ebf(L) + 2^{-2\ell(L)}\Abf^2)^\gamma
        \end{align}
    \end{itemize}
\end{proposition}
The proof of Proposition \ref{p:transversal-coherent} follows by the same reasoning to that for the proof of \cite{DMS}*{Proposition 13.4} and thus we do not include the details here. More precisely, we are simply reparameterizing the Lipschitz approximations $u_i$ of Proposition \ref{p:coherent} over the region $\pm \tilde R^o_1$ in the half-plane $\pm\Hbf_1$ and using the estimates in \cite{DLS_multiple_valued}*{Section 5}. Note that \cite{DLS_multiple_valued}*{Proposition 5.2} applies here since we are reparameterizing from an open region in $\Hbf_i \subset \pi_i$ to a (smaller) open region in $\pm\Hbf_1\subset \pi_1$.

\subsection{Non-concentration estimates}
We now proceed to improve the estimate \eqref{e:first-blowup}  to one that does not blow up as we shrink $\sigma \downarrow 0$, using the estimates in Section \ref{s:spine-est}.

\begin{proposition}\label{p:final-blowup-estimates}
    There exists $C=C(q,m,n,\bar n)>0$ such that for each $\varrho>0$ and $\eta>0$, there exists $\eps=\eps(q,m,n,\bar n,\varrho, 
    \eta)>0$ smaller than that in Theorem \ref{t:no-gaps} with $\varrho/2$ in place of $\varrho$, such that the following holds. Suppose that $T$, $\Sigma$, $\Abf$ are as in Assumption \ref{a:main} and suppose that $\Sbf\in \Bscr^q(0)\setminus \Pscr(0)$ is such that
    \begin{equation}
        \Ebb(T,\Sbf,\Bbf_1) + \Abf^2 \leq \eps^2 \boldsymbol\sigma(\Sbf)^2.
    \end{equation}
    For each $L\in \Gcal$, let $\beta(L)\in \{p:\Theta(T,p) \geq Q\}\cap \Bbf_{\varrho/2}(V)$ be a point with least distance to $y_L$ and let $r_*$ be as in Proposition \ref{p:Simon-shift} and let $r=\frac{r_*}{4}$. Then we have
    \begin{equation}\label{e:transversal-coherent-1}
        \int_{\Bbf_r} \frac{\dist^2(p,\Sbf)}{\max\{\varrho,\dist(p,V)\}^{1/2}} d\|T\|(p) \leq C(\hat\Ebf(T,\Sbf,\Bbf_1) + \Abf^2).
    \end{equation}
    In addition, $(\Bbf_r\cap\Hbf_j)\setminus \Bbf_\varrho(V) \subset R_i^o$ and the maps $u_j$ of Proposition \ref{p:coherent} satisfy the estimates
    \begin{align}
        \int_{(\Bbf_r\cap\Hbf_j)\setminus \Bbf_\varrho(V)} \frac{|Du_j(z)|^2}{\dist(z,V)^{1/2}} dz &\leq C(\hat\Ebf(T,\Sbf,\Bbf_1) + \Abf^2); \label{e:transversal-coherent-2} \\
        \sum_{i: 2^{-(i+1)}\geq \varrho}\sum_{L\in \Gcal_i}\int_{L_j} \frac{|u_j(z)\ominus(\mathbf{p}_{\pi_j}^\perp(\beta(L)))|^2}{\dist(z,V)^{5/2}} dz &\leq C(\hat\Ebf(T,\Sbf,\Bbf_1) + \Abf^2)\, . \label{e:transversal-coherent-3}
    \end{align}
    Moreover, if in addition $\boldsymbol\zeta(\Sbf)\leq c_0$ for $c_0$ as in Proposition \ref{p:transversal-coherent}, $(\Bbf_r\cap\pi_1) \setminus \Bbf_\varrho(V) \subset \tilde R_1^o\cup (-\tilde R_1^o)$ and the maps $w_j$ and $A_j$ therein satisfy
    \begin{align}
        \int_{(\Bbf_r\cap\pi_1)\setminus \Bbf_\varrho(V)} \frac{|Dw_j(z)|^2}{\dist(z,V)^{1/2}} dz &\leq C(\hat\Ebf(T,\Sbf,\Bbf_1) + \Abf^2); \label{e:transversal-coherent-4}\\
        \sum_{i: 2^{-(i+1)}\geq \varrho}\sum_{L\in \Gcal_i}\int_{\pm L_1} \frac{|w_j(z)\ominus (\mathbf{p}_{\pi_1}^\perp(\beta(L)) - A_j(\mathbf{p}_{\pi_1}(\beta(L))))|^2}{\dist(z,V)^{5/2}} dz &\leq C(\hat\Ebf(T,\Sbf,\Bbf_1) + \Abf^2) \label{e:transversal-coherent-5}\, . 
    \end{align}
    Finally, denote by $K_i$ the ``non-graphicality regions'', namely the union of the sets $\bar{K}_i (L)$ for $L\in \mathcal{G}^o$. Then for each $j$,
    \begin{equation}
          \int_{(\Bbf_r\cap\Hbf_j)\setminus (\Bbf_\varrho(V)\cup K_j)} |z|^{-(m-2)}\left|\partial_r\left(\frac{u_j(z)}{|z|}\right)\right|^2 dz \leq C(\hat\Ebf(T,\Sbf,\Bbf_1) + \Abf^2). \label{e:nonconc-Hardt-Simon}
    \end{equation}
    Likewise, if $K^\pm$ are the sets of Proposition \ref{p:transversal-coherent}, then
    \begin{equation}
            \int_{(\Bbf_r\cap\pi_1)\setminus (\Bbf_\varrho(V)\cup K^+\cup K^-)} |z|^{-(m-2)}\left|\partial_r\left(\frac{w_j(z)}{|z|}\right)\right|^2 dz \leq C(\hat\Ebf(T,\Sbf,\Bbf_1) + \Abf^2). \label{e:nonconc-Hardt-Simon-collapsed}
    \end{equation}
\end{proposition}

\begin{proof}
    First of all, the estimate \eqref{e:nonconc-Hardt-Simon} follows directly from the estimate \eqref{e:Simons-gradient-1}: when the current coincides with a classical graph this is a computation of Hardt and Simon in \cite{HS}.
    In a setting similar to ours the reader could consult \cite{DLHMSS}*{Proposition 8.3}, but we report the necessary modification for convenience.
    First of all, recall that $u_j$ is a multi-valued function: using Almgren's convention, see e.g. \cite{DLS_MAMS}, we write it as $u_j (z) = \sum_i \llbracket u_{j,i} (z)\rrbracket$. Even though there is not in general a regular selection of the sheets $u_{j,i}$, by \cite{DLS_multiple_valued}*{Lemma 1.1} we can decompose the domain of $u_j$ as the countable union of  measurable subsets $\{M_k\}$ with the property that over each fixed $M_k$ we can order the values of $u_j$ so that the corresponding sheets $u_{j,i}$
    are classical Lipschitz functions {\em over} $M_k$. Fix now a point $z\in M_k$ of differentiability of the function $u_{j,i}$ such that $p = z+u_{j,i} (z)$ belongs to $\spt (T)$ and there is an approximate tangent plane $P$ to $T$ at $p$. The computations in \cite{DLHMSS}*{Proposition 8.3} then show the pointwise bound
    \[
    |z|^{-(m-2)}\left|\partial_r\left(\frac{u_{i,j} (z)}{|z|}\right)\right|^2 \leq 2 \frac{|p^\perp|^2}{|p|^{m+2}}\, ,
    \]
    where we recall that $p^\perp$ is the projection of $p$ onto the orthogonal complement of $P$. We are can 
    thus conclude summing over $i$, integrating in $z$ on the domain in the left hand side of \eqref{e:nonconc-Hardt-Simon} and use \eqref{e:Simons-gradient-1}.
    The proof of \eqref{e:nonconc-Hardt-Simon-collapsed} is entirely analogous. First of all we observe that the exact same argument above proves in fact 
    \begin{equation}
            \int_{(\Bbf_r\cap\pi_1)\setminus (\Bbf_\varrho(V)\cup K^+\cup K^-)} |z|^{-(m-2)}\left|\partial_r\left(\frac{v_j(z)}{|z|}\right)\right|^2 dz \leq C(\hat\Ebf(T,\Sbf,\Bbf_1) + \Abf^2) \label{e:nonconc-Hardt-Simon-collapsed-2}\, .
    \end{equation}
    We then just need to notice that $w_j = v_j\ominus A_j$ and $A_j$ is a linear function with $A_j(0) = 0$; in particular $\partial_r \frac{A_j(z)}{|z|} \equiv 0$. This yields the identity
    \[
    \left|\partial_r\left(\frac{v_j(z)}{|z|}\right)\right|^2 = \left|\partial_r\left(\frac{w_j(z)}{|z|}\right)\right|^2\, 
    \]
    and thus \eqref{e:nonconc-Hardt-Simon-collapsed} follows from \eqref{e:nonconc-Hardt-Simon-collapsed}. 
    
    Next, given the estimates in Section \ref{s:spine-est}, the proofs of the estimates \eqref{e:transversal-coherent-1}--\eqref{e:transversal-coherent-3} and \eqref{e:transversal-coherent-4}, \eqref{e:transversal-coherent-5} are analogous to those in \cite{DMS}*{Proposition 13.7} (see also \cite{DLHMSS-excess-decay}*{Theorem 10.2, Corollary 11.2}). Nevertheless, we repeat them here for clarity.

    Fix such $\varrho$ and $\eta$. First we demonstrate \eqref{e:transversal-coherent-1}. We may without loss of generality assume that $\varrho\leq \frac{r_*}{2}$ since otherwise the estimate is trivial. We subdivide the domain of integration on the left-hand side into a disjoint union $(\Bbf_r\cap \Bbf_\varrho(V))\cup (\Bbf_r\setminus \Bbf_\varrho(V)))$. First we consider $\Bbf_r\cap \Bbf_\varrho(V)$. Covering $V\cap\Bbf_r$ with $C\varrho^{-(m-1)}$ balls $\Bbf_{\varrho}(z_i)$ (each of which can be ensured to intersect $V$), which in turn means that $\{\Bbf_{2\varrho}(z_i)\}_i$ covers $\Bbf_\varrho(V)$, we may apply Theorem \ref{t:no-gaps} to find points $p_i\in \Bbf_{2\varrho}(z_i)\cap \{\Theta(T,\cdot) \geq Q\}$. We may now apply Proposition \ref{p:Simon-shift} to $T_{0,1/4}$ centered at $p_i$ with $\kappa =\frac{1}{4}$, yielding
    \begin{align*}
        \int_{\Bbf_r\cap\Bbf_\varrho(V)} &\frac{\dist^2(p,\Sbf)}{\max\{\varrho,\dist(p,V)\}^{1/2}}d\|T\|(p\\
        &\leq \varrho^{-1/2}\sum_i \int_{\Bbf_r\cap\Bbf_{2\varrho}(z_i)} \dist^2(p,\Sbf)d\|T\|(p) \\
        &\leq C\varrho^{-1/2} \sum_i \int_{\Bbf_r\cap\Bbf_{2\varrho}(z_i)} \dist^2(p, p_i + \Sbf) d\|T\|(p) \\
        &\hspace{8em} + C\varrho^{m-1/2} \sum_i \left(|\mathbf{p}_{\pi_1}^\perp(p_i)|^2 + \boldsymbol\zeta(\Sbf)^2|\mathbf{p}_{V^\perp\cap\pi_1}(p_i)|^2\right) \\
        &\leq C\varrho^{m+7/4-1/2}\sum_i \int_{\Bbf_r\cap\Bbf_{4\varrho}(p_i)} \frac{\dist^2(p, p_i + \Sbf)}{|p-p_i|^{m+7/4}} d\|T\|(p) \\
        &\hspace{8em}+ C\varrho^{-(m-1)+m-1/2}(\hat\Ebf(T,\Sbf,\Bbf_1) + \Abf^2) \\
        &\leq C\varrho^{-(m-1)+m+5/4}(\hat\Ebf(T,\Sbf,\Bbf_1) + \Abf^2) + C\varrho^{1/2}(\hat\Ebf(T,\Sbf,\Bbf_1) + \Abf^2) \\
        &\leq C (\hat\Ebf(T,\Sbf,\Bbf_1) + \Abf^2).
    \end{align*}
    Now we treat the region $\Bbf_r\setminus \bar\Bbf_\varrho(V)$, which we may cover by the regions $R(L)$ for $L\in \Gcal$ with $2^{-(\ell(L)+1)} \geq \varrho$; let $\Fscr$ denote the sub-collection of such cubes $L$ with $R(L)\cap\Bbf_r\neq \emptyset$. First of all, observe that since $|y_L-\beta(L)|\leq \varrho/2$, there exists a constant $C=C(m)>0$ such that for each $p\in R(L)$ we have 
    \[
        C^{-2} 2^{-\ell(L)} \leq C^{-1}|p-\beta(L)|\leq \dist(p,V) \leq C|p-\beta(L)| \leq C^2 2^{-\ell(L)}
    \]
    Once again applying Proposition \ref{p:Simon-shift} to $T_{0,1/4}$ centered at $\beta(L)$ with $\kappa=\frac{1}{4}$ and noting that $\Hcal^m(R(L)) \leq C \ell(L)^m$ and $\# \Gcal_i \leq C 2^{(m-1)i}$, we have
    \begin{align*}
        \int_{\Bbf_r\setminus\Bbf_\varrho(V)} &\frac{\dist^2(p,\Sbf)}{\max\{\varrho,\dist(p,V)\}^{1/2}}d\|T\|(p) \leq C \sum_{i: 2^{-(i+1)}\geq \varrho} \sum_{L\in \Gcal_i \cap \Fscr} 2^{i/2} \int_{R(L)} \dist^2(p,\Sbf) d\|T\|(p) \\
        &\leq C \sum_{i: 2^{-(i+1)}\geq \varrho} \sum_{L\in \Gcal_i\cap\Fscr} 2^{i/2} \int_{R(L)} \dist^2(p,\beta(L)+\Sbf)d\|T\|(p) \\
        &\hspace{6em} + C \sum_{i: 2^{-(i+1)}\geq \varrho} \sum_{L\in \Gcal_i\cap\Fscr} 2^{i/2-mi}\left(|\mathbf{p}^\perp_{\pi_1}(\beta(L))|^2 +\boldsymbol\zeta(\Sbf)^2|\mathbf{p}_{V^\perp\cap\pi_1}(\beta(L))|^2\right) \\
        &\leq C\sum_{i: 2^{-(i+1)}\geq \varrho} \sum_{L\in \Gcal_i\cap\Fscr} 2^{i/2-mi-7i/4} \int_{R(L)} \frac{\dist^2(p,\beta(L)+\Sbf)}{|p-\beta(L)|^{m+7/4}}d\|T\|(p) \\
        &\hspace{6em} + C(\hat\Ebf(T,\Sbf,\Bbf_1) + \Abf^2) \sum_{i: 2^{-(i+1)}\geq \varrho} 2^{i/2-mi}\# \Gcal_i \\
        &\leq C(\hat\Ebf(T,\Sbf,\Bbf_1) + \Abf^2)\left[\sum_{i: 2^{-(i+1)}\geq \varrho} 2^{i/2-mi-7i/4} \# \Gcal_i + \sum_{i: 2^{-(i+1)}\geq \varrho} 2^{i/2-mi}\#\Gcal_i\right]\\
        &\leq C(\hat\Ebf(T,\Sbf,\Bbf_1) + \Abf^2)\sum_i 2^{-i/2} \\
        &\leq C(\hat\Ebf(T,\Sbf,\Bbf_1) + \Abf^2).
    \end{align*}
    In summary, we have established \eqref{e:transversal-coherent-1}.

    For the remaining estimates, we proceed as follows. Firstly, observe that Proposition \ref{p:Simon-shift} with $\kappa=\frac{1}{4}$ tells us that
    \begin{equation}\label{e:shifted-excess}
        \hat\Ebf(T,\beta(L)+\Sbf, \Bbf_{C 2^{-\ell(L)}}(\beta(L))) \leq C 2^{\ell(L)/4}(\hat\Ebf(T,\Sbf,\Bbf_1) + \Abf^2).
    \end{equation}
    Indeed, notice that the estimate \eqref{e:shift-1} applies as long as $C 2^{-\ell(L)}\leq r$, while the case $C 2^{-\ell(L)}\geq r$ follows trivially by the corresponding upper bound on $\ell(L)$.

    The estimate \eqref{e:shifted-excess} in turn tells us that for each cube we may apply Lemma \ref{l:regions}(v) with $\Sbf$ (and its layers) replaced by the shifted open book $\Sbf+\beta(L)$, to yield local $Q_j$-valued Lipschitz approximations $\tilde u_{L,j}$ on $\Omega_j(L) + \mathbf{p}_{V^\perp\cap\Hbf_i}(\beta(L))$ satisfying
    \begin{align*}
        \|D\tilde u_{L,j}\|_{L^2}^2 &\leq C 2^{-m\ell(L)}(\hat \Ebf(T,\Sbf+\beta(L),\Bbf_{C 2^{-\ell(L)}}(\beta(L)))+2^{-2\ell(L)}\Abf^2) \\
        &\leq C 2^{-m\ell(L)+\ell(L)/4} (\hat\Ebf(T,\Sbf,\Bbf_1) + \Abf^2),
    \end{align*}
    and, together with \eqref{e:shift-2}, additionally
    \[
        \|\tilde{u}_{L,j}\ominus (\mathbf{p}_{\pi_j}^\perp(\beta(L)))\|_{L^\infty} \leq C 2^{-7\ell(L)/4}(\hat\Ebf(T,\Sbf,\Bbf_1) + \Abf^2).
    \]
    Now we may use the maps $\tilde u_{L,j}$ in place of $u_{L,j}$ to construct the coherent outer approximation in Proposition \ref{p:coherent}. Note that one may assume these domains contain $L_j$, provided that $\eps$ sufficiently small, in light of \eqref{e:shift-2}. This in turn produces corresponding transversal coherent approximations $\tilde{u}_j$ and $\tilde{w}_j$ as in Proposition \ref{p:transversal-coherent}. Exploiting the above estimates, combined with \eqref{e:shifted-excess}, we have
    \begin{align*}
        \int_{L_j} \frac{|D\tilde{u}_j(z)|^2}{\dist(z,V)^{1/2}}dz &\leq C 2^{-m\ell(L)+3\ell(L)/4} (\hat\Ebf(T,\Sbf,\Bbf_1) + \Abf^2) \\
        \int_{\pm L_1} \frac{|D\tilde{w}_j(z)|^2}{\dist(z,V)^{1/2}}dz &\leq C 2^{-m\ell(L)+3\ell(L)/4} (\hat\Ebf(T,\Sbf,\Bbf_1) + \Abf^2) \\
        \int_{L_j} \frac{|\tilde{u}_j(z)\ominus (\mathbf{p}_{\pi_j}^\perp(\beta(L)))|^2}{\dist(z,V)^{5/2}}dz &\leq C 2^{-m\ell(L)+3\ell(L)/4} (\hat\Ebf(T,\Sbf,\Bbf_1) + \Abf^2) \\
        \int_{\pm L_1} \frac{|\tilde{w}_j(z)\ominus (\mathbf{p}_{\pi_1}^\perp(\beta(L))- A_j(\mathbf{p}_{\pi_1}(\beta(L))))|^2}{\dist(z,V)^{5/2}}dz &\leq C 2^{-m\ell(L)+3\ell(L)/4} (\hat\Ebf(T,\Sbf,\Bbf_1) + \Abf^2).
    \end{align*}
    Summing these estimates over all cubes $L\in \Gcal_i\cap\Fscr$ for $i$ such that $2^{-(i+1)} \geq \varrho$ and once again exploiting the fact that $\# \Gcal_i \leq C 2^{(m-1)i}$, we establish the respective estimates \eqref{e:transversal-coherent-2}-\eqref{e:transversal-coherent-5}.    
\end{proof}

\subsection{Blow-ups and variational identities}
Let us now define our blow-up sequences in the contradiction argument for the proof of Proposition \ref{p:collapsed} and Proposition \ref{p:noncollapsed}. For $m,n,\bar n, q$ fixed and $Q=\frac{q}{2}$, up to extracting a subsequence we have:
\begin{itemize}
    \item a sequence of currents $T_k$, ambient manifolds $\Sigma_k$ with curvatures $\Abf_k$ satisfying Assumption \ref{a:main} and with $\|T_k\|(\Bbf_1) 
    \leq \omega_m(Q+\frac{1}{4})$;
    \item open books $\Sbf_k=\Hbf^k_1\cup\cdots\cup\Hbf^k_N\in \Bscr^q(0)\setminus \Pscr(0)$ for some $N\in \N$, with $\boldsymbol{\zeta}(\Sbf_k) \to 0$ in the case of Proposition \ref{p:collapsed} or $\boldsymbol{\zeta}(\Sbf_k) \geq \eps_c^\star$ for $\eps_c^\star>0$ fixed arbitrarily in the case of Proposition \ref{p:noncollapsed};
    \item we have 
    \begin{equation}\label{e:blowup-exc-to-zero}
        \frac{\Abf_k^2}{\Ebb(T_k,\Sbf_k,\Bbf_1)} + \frac{\Ebb(T_k,\Sbf_k,\Bbf_1)}{\boldsymbol{\sigma}(\Sbf_k)} \to 0,
    \end{equation}
    \item For each $L\in\Gcal$, the $\beta^k(L)$ from Proposition \ref{p:final-blowup-estimates} for $T_k$ satisfy
    \begin{equation}
        \lim_{k\to\infty} |\beta^k(L)-y_L| = 0.
    \end{equation}
\end{itemize}
Up to relabelling the indices, we may additionally assume that the $m$-dimensional planes $\pi_i^k$ which are the extensions of the half-planes $\Hbf_i^k$ satisfy
\[
\dist(\pi_1^k\cap\Bbf_1,\pi_N^k\cap\Bbf_1)=\boldsymbol\zeta(\Sbf_k).
\]
In the collapsed decay regime, up to a rotation of coordinates, $\Hbf_i^k$ converge to $\pm\Hbf_1$. Meanwhile, in the non-collapsed decay regime, up to extracting a further subsequence, each $\Hbf_i^k\to \Hbf_i$ and $\pi_i^k \to \pi_i$ locally in Hausdorff distance and we may assume that $\pi_1\neq \pi_N$ (note that the $\Hbf_i$ need not be distinct). Note that in both cases, we may additionally assume that $V(\Sbf_k) = V$ is fixed along the sequence, again by a rotation of coordinates. We may additionally assume that $T_0\Sigma_k \equiv \tau_0$ is a fixed linear subspace and that $\pi_i^k \subset \tau_0$ for each $k$ and each $i=1,\dots, N$.

Using Proposition \ref{p:transversal-coherent}, in the collapsed decay regime we consider the transversal coherent approximations $w_i^k$, while in the non-collapsed decay regime we instead consider the approximations $u_i^k$. For $r$ as in Proposition \ref{p:transversal-coherent}, by Proposition \ref{p:first-blowup}(i), we may assume that the domains of $w_i^k$ and $u_i^k$ respectively contain
$(\Bbf_r\cap \Hbf_1)\setminus \Bbf_{1/k}(V)$ and $(\Bbf_r\cap \Hbf_i^k)\setminus \Bbf_{1/k}(V)$. Let
\[
    \bar u_i^k \coloneqq \frac{u_i^k}{\Ebb(T_k,\Sbf_k,\Bbf_1)^{1/2}}, \qquad \bar w_i^k \coloneqq \frac{w_i^k}{\Ebb(T_k,\Sbf_k,\Bbf_1)^{1/2}}.
\]
The estimates in Proposition \ref{p:final-blowup-estimates} in particular provide us with $W^{1,2}$ estimates uniformly in $k$ on $\bar u_i^k$ and $\bar{w}_i^k$. Note that for each $k$ the domain of $w_i^k$ is always contained in the same half-plane $\pm \Hbf_1$, while despite the fact that the half-planes on which the $u_i^k$ are defined are varying with $k$, we can use \cite{DMS}*{Lemma 7.4, Remark 8.21} to assume that they are defined on subsets of the fixed limiting half-planes $\Hbf_i$, up to a composition with small rotations (which all converge to the identity map).

Let us begin with the following definition.
\begin{definition}
    A function $W:\Bbf_r\to \R^{m+n}$ is called \emph{cylindrical} if $W(p_1)=W(p_2)$ for any $p_1,p_2$ with $\mathbf{p}_V(p_1)=\mathbf{p}_V(p_2)$ and $\dist(p_1,V)=\dist(p_2,V)$.
\end{definition}

We conclude the following compactness result and variational identities for our blow-ups, which are analogues of similar estimates seen originally in \cite{Simon_cylindrical} and \cite{W14_annals}*{Section 12}.

\begin{proposition}\label{p:final-blowup}
    Let $T_k$, $\Sigma_k$, $\Abf_k$, $\Sbf_k= \Hbf_1^k\cup\cdots\cup \Hbf_N^k$, $\pi_i^k$, and the sequences $\bar u_i^k$ and $\bar w_i^k$ be as described above. Then, up to extracting additional subsequences, the following holds:
    \begin{itemize}
        \item[(a)] $\bar{u}_i^k$ and $\bar{w}_i^k$ converge strongly in $W^{1,2}_{\loc}$ to their respective $Q_i$-valued Dir-minimizing limits $\bar u_i = \sum_{j=1}^{Q_i} \llbracket \bar u_{ij} \rrbracket$ and $\bar w_i = \sum_{j=1}^{Q_i} \llbracket \bar w_{ij} \rrbracket$ on $\Bbf_r\cap \Hbf_i$ and $(\Bbf_r\cap \pi_1)\setminus V$ respectively;
        \item[(b)] The rescaled linear maps $\bar{A}_i^k(te_1,v) \coloneqq \frac{A_i^k(te_1,v)}{\boldsymbol\zeta(\Sbf_k)} = M_i^k t e_1$ converge to linear maps $\bar{A}_i(te_1,v)=M_i t e_1$, where $M_i:\pi_1\cap V^\perp \to \pi_1^\perp$ are linear maps (identified with their matrix representations), $e_1$ is the unit vector in $\Hbf_1\cap V^\perp$ and $M_N \neq 0$;
        \item[(c)] Letting $\beta^k\coloneqq \sum_{L\in \Gcal\cap \Fscr}\beta^k(L) \mathbf{1}_{L_1\cup -L_1}$ on $(\Bbf_r\setminus \Bbf_{1/k}(V))\cap \pi_1$, there exist even, bounded functions $\bar\beta: \Bbf_r\cap\pi_1 \to V^\perp$, $\bar\beta^\perp: \Bbf_r\cap\pi_1 \to \pi_1^\perp$ and $\bar\beta^\parallel: \Bbf_r\cap\pi_1\to V^\perp\cap\pi_1$ such that in the non-collapsed decay case we have
        \begin{equation}
            \frac{\mathbf{p}_{V}^\perp(\beta^k)}{\Ebb(T_k,\Sbf_k,\Bbf_1)} \to \bar\beta \qquad \text{in $L^\infty_\loc$},
        \end{equation}
        whilst in the collapsed decay case we have
        \begin{align}
            \frac{\mathbf{p}_{\pi_1}^\perp(\beta^k)}{\Ebb(T_k,\Sbf_k,\Bbf_1)} &\to \bar\beta^\perp \qquad \text{in $L^\infty_\loc$} \\
            \frac{\boldsymbol{\zeta}(\Sbf_k)\mathbf{p}_{\pi_1}(\beta^k)}{\Ebb(T_k,\Sbf_k,\Bbf_1)} &\to \bar\beta^\parallel \qquad \text{in $L^\infty_\loc$.}
        \end{align}
        \item[(d)] In the collapsed decay case, let $I^+\coloneqq \{i\in \{1,\dots,N\}: w_i^k: \Hbf^k_1 \to \Acal_{Q_i}((\pi^k_1)^\perp)\}$ and $I^-\coloneqq \{i\in \{1,\dots,N\}: w_i^k: -\Hbf^k_1 \to \Acal_{Q_i}((\pi^k_1)^\perp)\}$ and let $\bar{w}^{(1)}:\Bbf_r\cap\Hbf_1\to \Hbf_1^\perp$, $\bar{w}^{(2)}:\Bbf_r\cap\Hbf_1\to V^\perp\cap \Hbf_1$ be given by
        \begin{align}
            \bar{w}^{(1)}(te_1,v)&\coloneqq \sum_{i\in I^+}\sum_{j=1}^{Q_i} \bar{w}_{ij}(te_1,v) + \sum_{i\in I^-} \sum_{j=1}^{Q_i}\bar{w}_{ij}(-te_1,v); \\
            \bar{w}^{(2)}(te_1,v)&\coloneqq \sum_{i\in I^+}M_i^{T}\sum_{j=1}^{Q_i} \bar{w}_{ij}(te_1,v) + \sum_{i\in I^-} M_i^{T} \sum_{j=1}^{Q_i}\bar{w}_{ij}(-te_1,v),
        \end{align}
        where $T$ denotes the transpose. Then, given any vector $v\in V$ and any cylindrical functions $W\in C_c^\infty(\Bbf_r;\R^{m+n})$, $W_o\in C_c^\infty(\Bbf_r;\pi_1^\perp)$, $W_p\in C_c^\infty(\Bbf_r;V^\perp\cap\pi_1)$, we have the following identities
        \begin{align}
            &\sum_i\int_{\Bbf_r\cap\Hbf_i} \sum_{j=1}^{Q_i}\left\langle\nabla\bar u_{ij} , \nabla \frac{\partial W}{\partial v} \right\rangle = 0, \label{e:spine-var-1}
        \end{align}
        in the non-collapsed case, while in the collapsed case we have
        \begin{align}
            &\int_{\Bbf_r\cap\Hbf_1} \left\langle \nabla\bar{w}^{(1)} , \nabla \frac{\partial W_o}{\partial v} \right\rangle = 0, \qquad l=1,2. \label{e:spine-var-2} \\
            &\int_{\Bbf_r\cap\Hbf_1} \left\langle \nabla\bar{w}^{(2)} , \nabla \frac{\partial W_p}{\partial v} \right\rangle = 0, \qquad l=1,2. \label{e:spine-var-3}
        \end{align}
        Here we stress: in the identity \eqref{e:spine-var-1} the operator $\nabla$ denotes differentiation with respect to the variable $z\in \pi_i$, whereas in \eqref{e:spine-var-2} and \eqref{e:spine-var-3} it denotes differentiation with respect to the variable $z\in \pi_1$. 
    \end{itemize}
\end{proposition}

\begin{proof}
The convergence in $W^{1,2}_{\text{loc}}$ from (a) follows from Proposition \ref{p:final-blowup-estimates} together with Proposition \ref{p:first-blowup}(iii) with a sequence of parameters $\varsigma_k \downarrow 0$. This also shows that the maps $\bar{u}_i$, $\bar{w}_i$ are Dir-minimizing on compact subsets in the interior of each domain in question. In turn, this implies that they have smaller energy than any competitor which coincides with itself on a neighbourhood of $V$, i.e. we are in the situation to apply Lemma \ref{l:technical-minimizing} in order to deduce that they are Dir-minimizing with respect to any competitor with the same trace on $V$. 
Property (b) is a trivial consequence of the fact that $V(\Sbf_k)=V$ for each $k$, that $\dist(\pi_1^k\cap\Bbf_1,\pi^k_N\cap\Bbf_1)=\boldsymbol{\zeta}(\Sbf_k)$, and Proposition \ref{p:transversal-coherent}(a). Property (c) follows immediately by Proposition \ref{p:Simon-shift}.

For property (d) we follow the reasoning in \cite{Simon_cylindrical}*{Section 5.1}, \cite{W14_annals}*{Section 12}
(see also \cite{DLHMSS}*{Proposition 10.5(iv)}, \cite{DLHMSS-excess-decay}*{Proposition 11.5}). 
Nevertheless, we repeat the details here as our setting is slightly different to all of these.

Fix $v\in V$ and $W$ with the required assumptions. Let us begin with \eqref{e:spine-var-1}. First, note that since the functions $\bar{u}^k_{i}$ and $w_i^k$ are valued in $V^\perp$ for each $i$, we can assume that $W$ takes values in $V^\perp$ also. Write $(x,y)\in V\times V^\perp$ for coordinates relative to $V$. We next claim that we may without loss of generality assume that $W$ depends only on the variable $x$ in some neighbourhood of $V$. Indeed, fix $\rho\in (0,r)$, and consider a cutoff function $\varphi_\rho\in C^\infty_c([0,\infty);[0,1])$ which equals $0$ on $[0,\rho/2]$, equals $1$ on $[\rho,\infty)$, and satisfies $\|\varphi^\prime_\rho\|_{C^0}\leq C\rho^{-1}$. Consider the vector field
\begin{equation}\label{e:cutoff-approx}
    W_\rho(x,y) \coloneqq W(x,y)\vphi_\rho(|y|) + W(x,0) (1-\vphi_\rho(|y|)).
\end{equation}
As $W-W_\rho = 0$ outside $\Bbf_\rho(V)$, applying Cauchy-Schwartz followed by the estimate \eqref{e:transversal-coherent-2} and then \eqref{e:blowup-exc-to-zero}, we thus have
\begin{align*}
    &\left|\sum_i\int_{\Bbf_r\cap \Hbf^k_i} \sum_{j=1}^{Q_i} \left\langle \nabla \bar u_{ij}^k, \nabla \left(\frac{\partial (W-W_\rho)}{\partial  v}\right)\right\rangle \right| \\
    &\qquad= \left|\sum_i\int_{\Bbf_r\cap\Bbf_\rho(V)\cap \Hbf^k_i} \sum_{j=1}^{Q_i} \left\langle \nabla \bar u_{ij}^k, \nabla \left(\frac{\partial (W-W_\rho)}{\partial  v}\right)\right\rangle \right| \\
    &\qquad\leq \|D^2 W \|_{C^0}\sum_i\int_{\Bbf_r\cap\Bbf_\rho(V)\cap \Hbf^{k}_i} \sum_{j=1}^{Q_i} \left|\nabla \bar u_{ij}^k\right| \\
    &\qquad\leq C \rho^{3/4}\|D^2 W \|_{C^0} \sum_i \sum_{j=1}^{Q_i} \left[\int_{\Bbf_r\cap\Bbf_\rho(V)\cap \Hbf^{k}_i} \frac{|\nabla\bar u_{ij}^k(z)|^{2}}{\dist(z,V)^{1/2}}dz\right]^{1/2} \\
    &\qquad\leq C\rho^{3/4} \|D^2 W \|_{C^0} \to 0,
\end{align*}
as $\rho\downarrow 0$, where $C=C(q,m,n,\bar{n})>0$. Thus, we can without loss of generality assume that $W$ only depends on the variable $x$ in a neighborhood of $V$ by fixing $\rho\in (0,r)$ and working with $W_\rho$, which we henceforth do (and we drop the subscript for notational simplicity). We write $\rho_0<r$ for the radius of the neighborhood of $V$ where $W$ only depends on the variable $x$. 

Now let
$\bar W \coloneqq \frac{\partial W}{\partial v}$. Since $T_k$ satisfy Assumption \ref{a:main} (in particular, $T_k$ is area-minimizing mod$(q)$), we have
\[
    \bar W(p) \cdot \vec H_{T_k}(p) = \mathbf{p}^\perp_{T_p\Sigma_k} \bar W(p) \cdot \vec H_{T_k}(p)
\]
and so
\begin{equation}\label{e:first-var-current}
    |\delta T_k(\bar W)| = \left|\int \mathbf{p}^\perp_{T_p\Sigma_k} \bar W(p) \cdot \vec H_{T_k}(p) d\|T_k\|(p) \right| \leq C\Abf_k^{2} \|\bar W\|_{C^0} \to 0,
\end{equation}
where $\vec H_{T_k}(p) \coloneqq \sum_{i=1}^m A_{\Sigma_k}(\xi_i,\xi_i)$ for an orthonormal basis $\{\xi_i\}_{i=1}^m$ of the approximate tangent plane to $\spt (T_k)$ at $p$.

On the other hand, since $\Sbf_k$ is invariant under the 1-parameter family of translations in the direction $v$ and $\bar{W}$ has compact support, we deduce that $S_k \coloneqq \sum_{i=1}^N Q_i \llbracket \Hbf_i \rrbracket$ satisfies 
\begin{equation}\label{e:first-var-book}
    \delta S_k(\bar W) = 0.
\end{equation}
In particular, we see from \eqref{e:blowup-exc-to-zero} that, if $E_k:= \mathbb{E}(T_k,\Sbf_k,\Bbf_1)$,
$$\lim_{k\to\infty}E_k^{-1/2}(\delta T_k(\bar{W}) - \delta S_k(\bar{W})) = 0.$$
Now fix $\rho\in (0,\rho_0)$, and assume that $k$ is sufficiently large so that $\bar{u}^k_i$ are defined outside $\Bbf_\rho(V)$. We will decompose everything into this neighborhood of $V$ and its complement (where $T_k$ can be approximated by the graph of $\bar{u}^k_i$). Indeed, set $U_\rho:= \Bbf_r\setminus \Bbf_\rho(V)$, and write $T_k = T_k^g + T_k^e$ and $S_k = S_k^g + S_k^e$, where $T_k^g= T_k \mres U_\rho$ and $S_k^g= S_k \mres U_\rho$, while $T_k^e = (T_k - T_k^g)\mres \Bbf_r$ and $S_k^e = (S_k - S_k^g)\mres \Bbf_r$. We claim that
\begin{align}
    \limsup_{k\to\infty} E_k^{-1/2} \left|\int \diverg_{\vec T_k} \bar W\, d\|T_k^e\| -  \int \diverg_{\vec S_k} \bar{W} \,d\|S_k^e\| \right| &\leq C\rho^{1/2} \label{e:first-var-spine} \\
    \lim_{k\to\infty} E_k^{-1/2} \left(\int \diverg_{\vec T_k} \bar W\, d\|T_k^g\| -  \int \diverg_{\vec S_k} \bar W\, d\|S_k^g\| \right) &= \sum_i \int_{(\Bbf_r\cap \Hbf_i)\setminus \Bbf_{\rho}(V)} \sum_{j=1}^{Q_i} \langle \nabla \bar u_{ij}, \nabla \bar W \rangle. \label{e:first-var-graphical}
\end{align}
Once we establish the validity of these claims, \eqref{e:spine-var-1} follows easily. Indeed, bearing in mind the definition of $\delta S_k(\bar W_\rho)$ and $\delta T_k(\bar W_\rho)$, a combination of \eqref{e:first-var-current}, \eqref{e:first-var-book}, \eqref{e:first-var-spine} and \eqref{e:first-var-graphical} yields
\[
    \left|\sum_i \int_{(\Bbf_r\cap \Hbf_i)\setminus \Bbf_{\rho}(V)} \sum_{j=1}^{Q_i} \langle \nabla \bar u_{ij}, \nabla \bar W_\rho \rangle\right| \leq C\rho^{1/2}
\]
for any $\rho\in (0,\rho_0)$, and so taking $\rho\downarrow 0$ gives \eqref{e:spine-var-1}.

It remains to check that \eqref{e:first-var-spine} and \eqref{e:first-var-graphical} hold. First of all, note that since $\bar{W}$ only depends on $x\in V$ in $\Bbf_{\rho_0}(V)$ by construction, we have $\frac{\partial\bar{W}}{\partial y} = 0$ in $\Bbf_{\rho_0}(V)$ for all $y\in V^\perp$. Moreover, as $\bar{W}$ takes values in $V^\perp$, for any subspace $\pi$ containing $V$ we have (for $\{e_1,\dotsc,e_{m-1}\}$ an orthonormal basis of $V$)
\begin{equation}\label{e:div-free}
    \diverg_\pi\bar{W} = \diverg_{V}\bar{W} = \sum_i e_i\cdot D(e_i\cdot \bar W) = 0 \qquad \text{in }\Bbf_{\rho_0}(V).
\end{equation}
In particular, for an arbitrary subspace $\pi$ we have
$$|\diverg_\pi\bar{W}| \leq C\|D\bar{W}\|\cdot|\mathbf{p}_V\circ\mathbf{p}_{\pi^\perp}|$$
(since by adding the components of $V\cap \pi^\perp$ to the sum we get the divergence over a plane containing $V$).

Using this, noting that each plane that is the extension of a half-plane in $\Sbf_k$ contains $V$, together with the fact that $S^e_k$ is supported in $\Bbf_{\rho_0}(V)$, the above discussion immediately gives
\[
    \int \diverg_{\vec S_k} \bar W\, d\| S^e_k \| = 0
\]
(indeed, the integrand vanishes pointwise in this case). Moreover, turning our attention to $T^e_k$, we then have from the above discussion also
\begin{align*}
    \left|\int \diverg_{\vec T_k} \bar W \, d\|T^e_k\|\right| &\leq C \|\bar W\|_{C^1} \left(\|T_k\|(\Bbf_r\cap\Bbf_\rho(V))\right)^{1/2} \left(\int_{\Bbf_r} \left|\mathbf{p}_V \circ \mathbf{p}_{\vec T_k^\perp}\right|^2d\|T_k\|\right)^{1/2} \\
    &\leq C \|\bar W\|_{C^1}\rho^{1/2} E_k^{1/2}
\end{align*}
for $C=C(q,m,n,\bar n)>0$, where we have used Cauchy--Schwarz in the first inequality and then \eqref{e:Simons-gradient-2} and \eqref{e:blowup-exc-to-zero} in the second. Combining the above, this concludes the proof of \eqref{e:first-var-spine}.

To see the validity of \eqref{e:first-var-graphical}, we exploit the graphical approximations $\bar u^k_i$ for $T_k$, which are valid over $U_\rho\cap \Hbf_i$. We also don't use the expressions in \eqref{e:first-var-graphical} directly, but more the first variation from which they come from. For $i=1,\dots, N$, let $A^k_i:\Hbf_i\to \Hbf_i^\perp$ denote the linear maps whose graphs are $\Hbf^k_i$. Observe that in the limit these maps converge to $0$ as $\Hbf^k_i$ converges to $\Hbf_i$. Let
\[
    h_i^k(x) = \sum_{j=1}^{Q_i}\llbracket h_{ij}^k(x) \rrbracket \coloneqq \sum_{j=1}^{Q_i} \llbracket (g^k_{ij}(x),\Psi^k(x,g^k_{ij}(x)))\rrbracket , 
\]
where $\Psi^k\equiv \Psi^k_0$ is as in Assumption \ref{a:main} for $\Sigma_k$ and $g^k_{ij}\coloneqq \mathbf{p}_{T_0\Sigma_k}\circ (u_{ij}^k + A^k_i) \in T_0 \Sigma_k\cap (\pi_i^k)^\perp$ (recall that $\pi_i^k \subset T_0\Sigma_k$). Invoking \eqref{e:transv-coherent-non-graph-set} (summed over $L\in \Gcal_{\leq \bar\ell}\cap U_\rho\cap \Hbf_i$) and \eqref{e:blowup-exc-to-zero}, we thus have (for the difference between graphical and non-graphical regions)
\[
    E_k^{-1/2}\left|\int \diverg_{\vec T_k} \bar W\, d\|T^g_k\| - \sum_{i} \sum_{j=1}^{Q_i}\underbrace{\int_{\mathbf{p}^{-1}_{\Hbf_i}(U_\rho\cap \Hbf_i)} \diverg_{\vec \Gbf_{h^k_{ij}}} \bar W\, d\|\Gbf_{h^k_{ij}}\|}_{=: I_{k,i,j}^{(1)}} \right| \leq C \|\bar W \|_{C^1} E_k^{1/2+\gamma} \to 0,
\]
for $\gamma>0$ as in Proposition \ref{p:crude-approx}. Moreover,
\[
    \int \diverg_{\vec S_k}\bar W\, d\| S_k\| = \sum_{i}\sum_{j=1}^{Q_i}\underbrace{\int_{\mathbf{p}_{\Hbf_i}^{-1}(U_\rho\cap \Hbf_i)} \diverg_{\vec\Gbf_{A^k_i}} \bar W\, d\|\Gbf_{A^k_i}\|}_{=: I_{k,i,j}^{(2)}}.
\]
It therefore remains to verify that for each $i,j$ we have
\begin{equation}\label{e:avg-variation}
    \lim_{k\to\infty} E_k^{-1/2}(I^{(1)}_{k,i,j} - I^{(2)}_{k,i,j}) = \int_{U_\rho\cap\Hbf_i} \langle \nabla \bar u_{ij}, \nabla \bar W \rangle.
\end{equation}
Fix $i\in\{1,\dots,N\}$, $j\in \{1,\dotsc,Q_i\}$, and an orthonormal frame $\{\zeta_l\}_{l=1}^{m+n}$ of $\R^{m+n}$ such that $\{\zeta_1,\dots,\zeta_m\}$ is an orthonormal basis of the plane extending $\Hbf_i$. Then over a point $x\in \Hbf_i$,
\[
    \sigma_l := \zeta_l + (\zeta_l\cdot \nabla)h^k_{ij},\qquad \tau_l:= \zeta_l + (\zeta_l\cdot\nabla) A^k_i \qquad l=1,\dots,m,
\]
are the respective coordinate frames for the tangent planes to $\Gbf_{h^k_{ij}}$ and $\Gbf_{A^k_i}$. For notational simplicity, let us drop the indices $i,j,k$ for the following computations. Form the matrices $A$ and $B$, where the $i^{\text{th}}$ column of $A$ is given by $\sigma_i$ and the $i^{\text{th}}$ column of $B$ is $\tau_i$. The Jacobian determinants for the maps $h$ and $A$ are then given by $\sqrt{\det(A^TA)}$ and $\sqrt{\det(B^TB)}$, respectively. The variation determined by $\bar{W}$ then gives (by differentiating the relevant Jacobian) that the first variation is determined by the matrices
$$M(A) = \sum_{\alpha,\beta}\sqrt{\det(A^TA)}(A^TA)^{-1}_{\alpha\beta}A_\alpha\otimes A_\beta$$
and
$$M(B) = \sum_{\alpha,\beta}\sqrt{\det(B^TB)}(B^TB)^{-1}_{\alpha\beta}B_\alpha\otimes B_\beta.$$
Thus, using the first variation formula to rewrite the expressions for $I^{(1)}$ and $I^{(2)}$, we get
$$I^{(1)} - I^{(2)} = \int_{U_\rho\cap \Hbf}(M(A)-M(B)):D\bar{W}.$$
We can then compute
$$A_\alpha\otimes A_\beta = B_\alpha\otimes B_\beta + \partial_\alpha u\otimes e_\beta + e_\alpha\otimes \partial_\beta u + o(E^{1/2}) + o(1)|\nabla u| + O(|\nabla u|^2)$$
and
$$(A^TA)_{\alpha\beta} = \delta_{\alpha\beta} + \partial_\alpha A\otimes\partial_\beta A + o(E^{1/2}) + o(1)|\nabla u| + O(|\nabla u|^2)$$
$$(B^TB)_{\alpha\beta} = \delta_{\alpha\beta} + \partial_\alpha A\otimes\partial_\beta A + o(E^{1/2}) = \delta_{\alpha\beta} + o(1).$$
Using the fact that $W$ is cylindrical and so $\partial_v W = 0$ on $\Hbf$ whenever $v\in \Hbf^\perp$, and moreover that $\partial_\alpha u\in \Hbf^\perp$ (c.f. \cite{DLHMSS}*{Proposition 10.5(iv)}), we can then compute
$$I^{(1)} - I^{(2)} = E^{1/2}\int_{U_\rho\cap \Hbf}\langle \nabla \bar{u},\nabla\bar{W}\rangle + o(E^{1/2})$$
which completes the proof of \eqref{e:avg-variation}.

To show \eqref{e:spine-var-2}, we follow entirely analogous reasoning to the above proof of \eqref{e:spine-var-1}, replacing $W$ with $W_o$ and observing that $W_o$ identifies with a cylindrical function whose image lies in $\R^{m+n}$ but whose components in the directions of $\pi_1$ are zero (and $\bar w_{ij}$ is also identified with an $\R^{m+n}$-valued function, as in \eqref{e:spine-var-1}). In addition, the application of \eqref{e:transversal-coherent-2} is replaced with \eqref{e:transversal-coherent-4} here. This establishes \eqref{e:spine-var-2}.

Now let us demonstrate the validity of \eqref{e:spine-var-3}. Firstly, by the same reasoning as that for \eqref{e:spine-var-1} (only applying \eqref{e:transversal-coherent-4} in place of \eqref{e:transversal-coherent-2}), we have 
\[
    \left|\int_{\Bbf_r\cap\Hbf_1}\left\langle \nabla \bar w^{(2)}, \nabla \left(\frac{\partial(W_p-(W_p)_\rho)}{\partial v}\right)\right\rangle\right| \leq C\rho^{3/4}\|D^2 W\|_{C^0} \to 0,
\]
as $\rho \downarrow 0$ where $(W_p)_\rho$ is defined as in \eqref{e:cutoff-approx} for $W_p$ in place of $W$. This allows us to once again without loss of generality assume that $W_p$ only depends on the $x$-variable in a neighborhood $\Bbf_{\rho_0}(V)$ of $V$. As before, fix $\rho\in (0,\rho_0)$ and write $U_\rho:= \Bbf_r\setminus\Bbf_\rho(V)$, and for $k$ sufficiently large decompose the currents as $T_k = T^g_k + T^e_k$ and $S_k = S^g_k+S^e_k$. Let $\bar W=\frac{\partial W_p}{\partial v}$, and observe that \eqref{e:first-var-current} and \eqref{e:first-var-book} still hold with this new choice of test vector field. Thus, in light of \eqref{e:blowup-exc-to-zero}, it suffices to demonstrate that
\begin{align}
    \limsup_{k\to\infty} \boldsymbol\zeta(\Sbf_k)^{-1} E_k^{-1/2} \left|\int \diverg_{\vec T_k} \bar W\, d\|T_k^e\| -  \int \diverg_{\vec S_k} \bar W\, d\|S_k^e\| \right| &\leq C\rho^{1/2} \label{e:first-var-spine-collapsed} \\
    \lim_{k\to\infty} \boldsymbol\zeta(\Sbf_k)^{-1}E_k^{-1/2} \left(\int \diverg_{\vec T_k} \bar W\, d\|T_k^g\| -  \int \diverg_{\vec S_k} \bar W\, d\|S_k^g\| \right) &= \int_{(\Bbf_r\cap \Hbf_1)\setminus \Bbf_{\rho}(V)} \langle \nabla \bar w^{(2)}, \nabla \bar W \rangle \label{e:first-var-graphical-collapsed}.
\end{align}
Let us begin with \eqref{e:first-var-spine-collapsed}. We follow similar reasoning to \eqref{e:first-var-spine}. Indeed, because $\bar{W}$ takes values in $V^\perp$ by assumption, and in $\Bbf_{\rho_0}(V)$ we know $\frac{\partial\bar{W}}{\partial y} = 0$ for all $y\in V^\perp$, for any $m$-dimensional plane $\varpi$ containing $V$ we have
$$\diverg_\varpi\bar{W} = 0,$$
by the same computation as in \eqref{e:div-free}. So in particular,
$$\int\diverg_{\vec{S}_k}\bar{W}\, d\|S^e_k\| = 0 \qquad \text{in }\Bbf_{\rho_0}(V).$$
For an arbitrary $m$-dimensional plane $\varpi$, we have by similar reasoning (c.f. the argument leading to \cite{DLHMSS-excess-decay}*{(11.25)}), noting that in fact by assumption we know that $\bar{W}\in \pi_1\cap V^\perp$,
$$|\diverg_\varpi\bar{W}| \leq C|D\bar{W}|_{C^0}\cdot |\mathbf{p}_V\circ\mathbf{p}_\varpi^\perp|\cdot|\mathbf{p}_{\varpi}^\perp\circ\mathbf{p}_{\pi_1}| \qquad \text{in }\Bbf_{\rho_0}(V).$$
Combining this with Allard's tilt excess estimate \cite{Allard_72} (c.f. \cite{DL-All}*{Proposition 4.1}), Lemma \ref{l:planar-excess-zeta}(a), \eqref{e:Simons-gradient-2} and \eqref{e:blowup-exc-to-zero}, we therefore have
\begin{align*}
    \left|\int \diverg_\pi \bar W \, d\|T^e_k\|\right| &\leq C \left(\int |\mathbf{p}^\perp_{\vec T_k}\circ \mathbf{p}_{\pi_1}|^2 \, d\|T^e_k\|\right)^{1/2} \left(\int_{\Bbf_r} \left|\mathbf{p}_V \circ \mathbf{p}_{\vec T_k}^\perp\right|^2\,d\|T_k\|\right)^{1/2} \\
    &\leq C \left(\int |\mathbf{p}_{\vec T_k}-\mathbf{p}_{\pi_1}|^2 \,d\|T^e_k\|\right)^{1/2} \left(\int_{\Bbf_r} \left|\mathbf{p}_V \circ \mathbf{p}_{\vec T_k}^\perp\right|^2\,d\|T_k\|\right)^{1/2} \\
    &\leq C\rho^{1/2} \boldsymbol{\zeta}(\Sbf_k)E_k^{1/2}.
\end{align*}
This completes the proof of \eqref{e:first-var-spine-collapsed}.

Now let us prove \eqref{e:first-var-graphical-collapsed}. For any index $i\in\{1,\dots,N\}$, let $w_{ij}^{k,+}\coloneqq w_{ij}^k$ if $i\in I^+$ and let $w_{ij}^{k,-}(te_1,v) \coloneqq w_{ij}^k(-t e_1,v)$ if $i\in I^-$. Similarly, let $A^{k,\pm}_i(t e_1,v) \coloneqq A^k_i(\pm te_1,v)$ and
\[
    \tilde h_i^{k,\pm}(x) = \sum_{j=1}^{Q_i}\llbracket h_{ij}^{k,\pm}(x) \rrbracket \coloneqq \sum_{j=1}^{Q_i} \llbracket (\tilde g^{k,\pm}_{ij}(x),\Psi^k(x,\tilde g^{k,\pm}_{ij}(x)))\rrbracket , 
\]
where the sign $\pm$ is chosen to be $+$ for $i\in I^+$ and $-$ for $i\in I^-$, $\Psi^k\equiv \Psi^k_0$ is as in Assumption \ref{a:main} for $\Sigma_k$ and $\tilde g^{k,\pm}_{ij}\coloneqq \mathbf{p}_{T_0\Sigma}\circ (w_{ij}^{k,\pm} + A_i^{k,\pm})$. Again observe that $\nabla \tilde h_{ij}^{k,\pm} = \nabla w_{ij}^{k,\pm} + \nabla  A^{k,\pm}_i + O(\|D\Psi^k\|_{C^0})$. Arguing exactly as above, combined with the estimate in Lemma \ref{l:planar-excess-zeta}(a), together with the fact that $\bar{W}_\rho$ is cylindrical, we have
\[
    \boldsymbol{\zeta}(\Sbf)^{-1} E_k^{-1/2}\left|\int \diverg_{\vec T_k} \bar W\, d\|T^g_k\| -\sum_{j=1}^{Q_i}\underbrace{\int_{\mathbf{p}^{-1}_{\pi_1}(U_\rho\cap \Hbf_1)} \diverg_{\vec \Gbf_{\tilde h^{k,\pm}_{ij}}} \bar W\, d\|\Gbf_{\tilde h^{k,\pm}_{ij}}\|}_{=: J_{k,i,j}^{(1)}} \right| \leq C \|\bar W \|_{C^1} E_k^{\gamma} \to 0,
\]
and
\[
    \int \diverg_{\vec S_k}\bar W\, d\| S_k\| = \sum_{i}\sum_{j=1}^{Q_i}\underbrace{\int_{\mathbf{p}_{\pi_1}^{-1}(U_\rho\cap \Hbf_1)} \diverg_{\vec\Gbf_{A^{k,\pm}_i}} \bar W\, d\|\Gbf_{A^{k,\pm}_i}\|}_{=: J_{k,i,j}^{(2)}},
\]
Therefore, observe that it suffices to show that for each $i\in I^\pm$ and each $j=1,\dots,Q_i$ and $k$ sufficiently large, we have
\begin{equation}\label{e:avg-variation-collapsed}
    \lim_{k\to\infty}\boldsymbol{\zeta}(\Sbf_k)^{-1}E_k^{-1/2}(J_{k,i,j}^{(1)}-J_{k,i,j}^{(2)}) =\int_{U_\rho\cap \Hbf_1} \langle\nabla M_i^{T}\bar w_{ij}^{\pm},\nabla \bar W \rangle.
\end{equation}

Observe that since $\bar W$ is directed in $\Hbf_1\cap V^\perp$, for a Lipschitz function $v$ on $\Omega_\rho\coloneqq U_\rho\cap\Hbf_1$, the first variation of $\Gbf_v$ in the direction of the vector field $\bar W$ is the inner variation
\begin{align*}
    \delta \Gbf_v(\bar W) = \frac{d}{dt}\Big|_{t= 0} \int_{\Omega_\rho} \Jbf v_t = \int_{\Omega_\rho} \left\langle D_M \Acal(\nabla v), \nabla v \cdot \nabla \bar W\right\rangle - \Acal(\nabla v) \diverg_{\pi_1} \bar W,
\end{align*}
where $v_t = v\circ \Phi_t$ for $\Phi_t(x)\coloneqq x+ t\bar W(x)$, $\Acal(\nabla v) \coloneqq \Jbf v$ and $D_M$ denotes the Frech\'{e}t derivative of $\Acal$ with respect to the matrix variable. 

Now
\[
    \Acal(\nabla A^{k,\pm}_i) = \sqrt{1+ |\nabla A_i^{k,\pm}|^2}, \qquad D_M \Acal(\nabla A^{k,\pm}_i) = \frac{\nabla A_i^{k,\pm}}{\sqrt{1+ |\nabla A_i^{k,\pm}|^2}}.
\]
Moreover, Lemma \ref{l:planar-excess-zeta}, \eqref{e:blowup-exc-to-zero} and the estimates in Proposition \ref{p:transversal-coherent} imply that
\begin{align*}
    \|D\Psi^k\|_{C^0} &= O(\Abf_k) = o(E_k^{1/2}) =o(\boldsymbol\zeta(\Sbf_k)), \\
    \|\nabla w^{k,\pm}_{ij}\|_{L^2} &= O(E_k^{1/2}) = o(\boldsymbol\zeta(\Sbf_k)) \\
    |\nabla A^{k,\pm}_i| &= O(\boldsymbol\zeta(\Sbf_k)).
\end{align*}
Thus, by a Taylor expansion, we have
\begin{align*}
    \Acal(\nabla \tilde h_{ij}^{k,\pm})&= \sqrt{1+ |\nabla A_i^{k,\pm}|^2}+ \frac{\langle \nabla A^{k,\pm}_i, \nabla w_{ij}^{k,\pm}\rangle}{\sqrt{1+ |\nabla A_i^{k,\pm}|^2}} + o(\boldsymbol\zeta(\Sbf_k) E_k^{1/2}) + O(|\nabla w^{k,\pm}_{ij}|^2), \\
    D_M \Acal(\nabla \tilde h_{ij}^{k,\pm}) &= \frac{\nabla \tilde h_{ij}^{k,\pm}}{\sqrt{1+ |\nabla A_i^{k,\pm}|^2}} + o(\boldsymbol\zeta(\Sbf_k) E_k^{1/2}) + O(|\nabla w^{k,\pm}_{ij}|). 
\end{align*}
Since $J_{k,i,j}^{(1)}-J_{k,i,j}^{(2)} = \delta \Gbf_{\tilde h_{ij}^{k,\pm}}(\bar W) - \delta\Gbf_{A_i^{k,\pm}}(\bar W)$, this combined with the above graphicality estimates yields
\begin{align*}
    & J_{k,i,j}^{(1)}-J_{k,i,j}^{(2)} \\
    = & \int_{\Omega_\rho} \frac{\langle\nabla w_{ij}^{k,\pm}, \nabla w_{ij}^{k,\pm} \cdot \nabla \bar W\rangle}{\sqrt{1+ |\nabla A_i^{k,\pm}|^2}} + \frac{\langle\nabla A_{i}^{k,\pm}, \nabla w_{ij}^{k,\pm} \cdot \nabla \bar W\rangle}{\sqrt{1+ |\nabla A_i^{k,\pm}|^2}} + \frac{\langle\nabla w_{ij}^{k,\pm}, \nabla A_{i}^{k,\pm} \cdot \nabla \bar W\rangle}{\sqrt{1+ |\nabla A_i^{k,\pm}|^2}} \\
    &\qquad - \int_{\Omega_\rho} \frac{\langle \nabla A^{k,\pm}_i, \nabla w_{ij}^{k,\pm}\rangle}{\sqrt{1+ |\nabla A_i^{k,\pm}|^2}}\diverg_{\pi_1} \bar W + o(\boldsymbol\zeta(\Sbf_k)E_k^{1/2}) \\
    &= \int_{\Omega_\rho} \langle\nabla A_{i}^{k,\pm}, \nabla w_{ij}^{k,\pm} \cdot \nabla \bar W\rangle + \langle\nabla w_{ij}^{k,\pm}, \nabla A_{i}^{k,\pm} \cdot \nabla \bar W\rangle - \langle \nabla A^{k,\pm}_i, \nabla w_{ij}^{k,\pm}\rangle\diverg_{\pi_1} \bar W \\
    &\qquad + o(\boldsymbol\zeta(\Sbf_k)E_k^{1/2}).
\end{align*}
Note that above, $\nabla w_{ij}^{k,\pm}\cdot\nabla \bar W$ and $\nabla A_{ij}^{k,\pm}\cdot\nabla \bar W$ denotes tensor multiplication between the gradients of $\R^{m+n}$-valued functions. Now observe that since $\nabla A_i^{k,\pm}= M_i^k e_1$ and $\diverg_{\pi_1}\bar W = e_1\cdot \nabla(e_1\cdot \bar W)$ in light of the fact that $\bar W$ takes values in $V^\perp\cap \pi_1$, we arrive at
\[
    \langle\nabla A_{i}^{k,\pm}, \nabla w_{ij}^k \cdot \nabla \bar W\rangle = \langle \nabla A^{k,\pm}_i, \nabla w_{ij}^k\rangle\diverg_{\pi_1} \bar W.
\]
Transposing $\nabla A_i^{k,\pm}$ in the inner product in the second term on the right-hand side of the above estimate, this concludes the proof of \eqref{e:avg-variation-collapsed}. 
\end{proof}

\subsection{Smoothness of average and decay}
For $t\in \R^+$ and $e_i$ the unit vector in $\Hbf_i\cap V^\perp$, we may parameterize $\Hbf_i$ as
\[
    \Hbf_i= \{(te_i, v): t\in \R^+, \ v\in V\},
\]
and in turn write
\begin{align*}
    \bar{u}_i(te_1,v) &\coloneqq \sum_{j=1}^{Q_i} \llbracket \bar{u}_{ij}(te_i,v) \rrbracket, \\
    \bar{w}_i(te_1,v) &\coloneqq \sum_{j=1}^{Q_i} \llbracket \bar{w}_{ij}(\pm te_1,v) \rrbracket.
\end{align*}
Note that we are abusing notation slightly, since these are not the same as the definitions of $\bar u_i$ and $\bar w_i$ as defined in Proposition \ref{p:final-blowup}; namely we are now reparameterizing to $\Hbf_1$. We additionally write the average $\bar{u}: \Bbf_r\cap\Hbf_1 \to \R^{m+n}$ as
\[
    \bar{u}(te_1,v) \coloneqq \frac{1}{q}\sum_{i=1}^N\sum_{j=1}^{Q_i} \bar{u}_{ij}(te_i,v).
\]
The following proposition characterizes the properties of $\bar{u}$ and $\bar{w}$.

\begin{proposition}\label{p:harmonic}
    The following properties hold:
    \begin{itemize}
        \item[(i)] $\bar u$, $\bar w^{(1)}$ and $\bar w^{(2)}$ are harmonic and extend to harmonic functions on $\Bbf_r\cap\pi_1$ (not relabelled) with $\frac{\partial^2 \bar u}{\partial t \partial v} = \frac{\partial^2 \bar w^{(1)}}{\partial t \partial v} = \frac{\partial^2 \bar w^{(2)}}{\partial t \partial v} = 0$ on $V\cap \Bbf_r$ for any $v\in V$;
        \item[(ii)] $\bar{u}(0)=\bar{w}(0)=0$.
    \end{itemize}
\end{proposition}

\begin{proof}
The proof follows the same line of reasoning as \cite{DLHMSS}*{Lemma 11.2} but we repeat some of the details here for the convenience of the reader.
The harmonicity of $\bar u$, $\bar{w}^{(1)}$ and $\bar w^{(2)}$ on $\Bbf_r\cap\Hbf_1$ follows immediately from Proposition \ref{p:final-blowup}. To see that they extend to harmonic functions, we proceed as follows. Observe that $\bar u, \bar w^{(1)}, \bar w^{(2)} \in W^{1,2}$ on $\Bbf_r\cap\Hbf_1$. Furthermore, the estimates \eqref{e:transversal-coherent-2} and \eqref{e:transversal-coherent-4} imply that
\begin{align*}
    \int_{\Bbf_r\cap\Hbf_1} \frac{|\nabla \bar u|^2}{t^{1/2}} \, dtdv&< +\infty, \\
    \int_{\Bbf_r\cap\Hbf_1} \frac{|\nabla \bar w^{(l)}|^2}{t^{1/2}} \, dtdv &< +\infty, \qquad l=1,2.
\end{align*}
In particular, $\frac{\partial \bar u}{\partial t}$, $\frac{\partial \bar w^{(1)}}{\partial t}$ and $\frac{\partial \bar w^{(2)}}{\partial t}$ are all well-defined on $V\cap\Bbf_r$ (as distributions supported on there). Thus, given any cylindrical vector fields $W\in C_c^\infty(\Bbf_r\cap\pi_1;\R^{m+n})$, $W_o\in C_c^\infty(\Bbf_r\cap\pi_1;\pi_1^\perp)$, $W_p\in C_c^\infty(\Bbf_r\cap\pi_1;V^{\perp}\cap\pi_1)$, Proposition \ref{p:final-blowup}(d) integrated by parts tells us that
\begin{equation}\label{e:spine-derivatives}
    \left\langle\frac{\partial \bar u}{\partial t}\Big|_{V\cap\Bbf_r}, \frac{\partial W}{\partial v} \right\rangle = \left\langle\frac{\partial \bar w^{(1)}}{\partial t}\Big|_{V\cap\Bbf_r}, \frac{\partial W_o}{\partial v} \right\rangle = \left\langle \frac{\partial \bar w^{(2)}}{\partial t}\Big|_{V\cap\Bbf_r}, \frac{\partial W_p}{\partial v}\right\rangle = 0,
\end{equation}
for any $v\in V$, where we use $\langle\cdot,\cdot\rangle$ to denote the pairing of distribution and test function. In other words, the distributions $\frac{\partial \bar u}{\partial t}\Big|_{V\cap\Bbf_r}, \frac{\partial \bar w^{(1)}}{\partial t}\Big|_{V\cap\Bbf_r}, \frac{\partial \bar w^{(2)}}{\partial t}\Big|_{V\cap\Bbf_r}$ are in fact constant vectors, and thus, after translating $\bar u, \bar w^{(1)}, \bar w^{(2)}$ by appropriate respective linear functions dependent only on $t$, we obtain harmonic functions with vanishing normal derivative along $V\cap\Bbf_r$. By Schwartz reflection, we may thus extend the translated functions to harmonic functions on $\Bbf_r\cap\pi_1$, which in turn produces extensions for $\bar u, \bar w^{(1)},\bar w^{(2)}$ respectively. The claimed differential constraint on $V\cap\Bbf_r$ in (i) follows immediately from \eqref{e:spine-derivatives}.

Given the estimates \eqref{e:nonconc-Hardt-Simon} and \eqref{e:nonconc-Hardt-Simon-collapsed}, the validity of the property (ii) follows by the exact same reasoning as that for \cite{DLHMSS}*{Lemma 11.2(ii)}. We omit the details here.
\end{proof}

We are now in a position to state the main decay property for the maps $\bar u_i$ and $\bar w_i$.

\begin{proposition}\label{p:final-decay}
    For every $\eps>0$, there exists:
    \begin{itemize}
        \item $\rho_0=\rho_0(q,m,n,\bar n,\eps) \in (0,r)$ in the collapsed case,
        \item $\rho_0 = \rho_0 (q,m,n,\bar n, \eps, \varepsilon_c^\star)\in (0,r)$ in the non-collapsed case,
    \end{itemize}
    such that for every $\rho \in (0,\rho_0]$, the following holds. Let $\bar u_i$ and $\bar w_i$ be as above. Then there are linear maps $b:V\to V^\perp$, $b^\perp: V \to \pi_1^\perp$, $b^\parallel:V\to V^\perp\cap \pi_1$, and maps $a^c, a^{nc}\in \Lscr$ as in Definition \ref{d:spaces} (for $H=\Hbf_1$) with
    
    \[\sum_i\|\nabla a_i^c \|_{L^\infty} + \sum_i \|\nabla a_i^{nc} \|_{L^\infty} + \|\nabla b \|_{L^\infty}+ \|\nabla b^\perp \|_{L^\infty} + \|\nabla b^\parallel \|_{L^\infty} \leq C;\]
    \begin{align*}
        &\int_{\Hbf_1\cap \Bbf_\rho} \Gcal\left(\sum_{i=1}^N \bar u_i(te_1,v) \ominus b(v), a^{nc}(te_1,v)\right)^2\, dy dt \\
        &\hspace{12em}\leq \eps \left(\frac{\rho}{r}\right)^{m+2} \int_{\Hbf_1\cap \Bbf_r}\left|\sum_{i=1}^N \bar u_i(te_1,v) \ominus b(v)\right|^2 \, dy dt \\
        &\int_{\Hbf_1\cap \Bbf_\rho}\Gcal\left(\sum_{i=1}^N \bar w_i(te_1,v) \ominus (b^\perp(v) + \bar A_i(b^\parallel(v))), a^{c}(te_1,v)\right)^2\, dy dt \\
        &\hspace{12em}\leq \eps \left(\frac{\rho}{r}\right)^{m+2} \int_{\Hbf_1\cap \Bbf_r}\left|\sum_{i=1}^N \bar w_i(te_1,v) \ominus (b^\perp(v) + \bar A_i(b^\parallel(v)))\right|^2 \, dy dt, \\
    \end{align*}
    where $C=C(q,m,n,\bar n)>0$ in the non-collapsed case and $C = C(q,m,n,\bar n, \varepsilon_c^\star)$ in the collapsed case.
\end{proposition}

\begin{proof}
    In light of Lemma \ref{l:bdry-decay}, observe that it suffices to find $b$, $b^\perp$ and $b^\parallel$ as in the statement of the proposition such that
    \[
        \sum_i \bar u_i(te_1,v)\ominus b(v) \in \Hscr, \qquad \sum_i \bar w_i(te_1,v)\ominus (b^\perp(v) + \bar A_i(b^\parallel(v))) \in \Hscr.
    \]
    where $\Hscr$ is as in Definition \ref{d:spaces}, with respect to $H=\Hbf_1$.

    Now, we may first argue as in the proof of \cite{DLHMSS-excess-decay}*{Proposition 12.1} to define the traces of the functions $\bar\beta$, $\bar\beta^\parallel$, and $\bar \beta^\perp$, and use Proposition \ref{p:final-blowup-estimates} to reach the estimates 
    \begin{align}
        \int_{\Hbf_1\cap\Bbf_{r/2}} \frac{\left|\sum_i \bar u_i(te_1,v)\ominus \bar\beta(0,v)\right|^2}{|t|^{9/4}} &\leq C; \\
        \int_{\Hbf_1\cap\Bbf_{r/2}} \frac{\left|\sum_i \bar w_i(te_1,v)\ominus (\bar\beta^\perp(0,v) - \bar A_i(\bar\beta^\parallel(0,v)))\right|^2}{|t|^{9/4}} &\leq C.
    \end{align}
    In particular, we deduce that
    \begin{align}
        \int_{\Hbf_1\cap\Bbf_{r/2}} \frac{\left|\bar u(te_1,v)\ominus \bar\beta(0,v)\right|^2}{|t|^{9/4}} &\leq C; \label{e:trace}\\
        \int_{\Hbf_1\cap\Bbf_{r/2}} \frac{\left|\bar w^{(1)}(te_1,v) - (q\bar\beta^\perp(0,v) - \sum_i \bar A_i(\bar\beta^\parallel(0,v)))\right|^2}{|t|^{9/4}} &\leq C; \\
        \int_{\Hbf_1\cap\Bbf_{r/2}} \frac{\left|\bar w^{(2)}(te_1,v) - (\sum_i M_i^T\bar\beta^\perp(0,v) - \sum_i \|M_i \|^2 \bar\beta^\parallel(0,v))\right|^2}{|t|^{9/4}} &\leq C.
    \end{align}
    Inverting the system corresponding to the latter two estimates (c.f. \cite{DLHMSS-excess-decay}*{Proof of Proposition 12.1} or \cite{W14_annals}*{Proof of Theorem 12.2}), this further yields
    \begin{align}
        \int_{\Hbf_1\cap\Bbf_{r/2}} \frac{\left|\bar\beta^\perp(0,v) - (\alpha_1 \bar w^{(1)}(te_1,v) + B_1 \bar w^{(2)}(te_1,v))\right|^2}{|t|^{9/4}} &\leq C; \label{e:trace-perp}\\
        \int_{\Hbf_1\cap\Bbf_{r/2}} \frac{\left|\bar\beta^\parallel(0,v) - (B^T_2 \bar w^{(1)}(te_1,v) + \alpha_2 \bar w^{(2)}(te_1,v))\right|^2}{|t|^{9/4}} &\leq C;\label{e:trace-parallel}
    \end{align}
    where $\alpha_1,\alpha_2\in \R$ and $B_1,B_2: \pi_1\cap V^\perp \to\pi^\perp$ are linear maps, identified with their matrix representations.

    From \eqref{e:trace}, \eqref{e:trace-perp} and \eqref{e:trace-parallel} and Proposition \ref{p:harmonic}, we therefore deduce that the functions $\bar\beta(0,\cdot)$, $\bar\beta^\perp(0,\cdot)$ and $\bar\beta^\parallel(0,\cdot)$ are traces on $V\cap\Bbf_{r/2}$ of harmonic functions $\bar u$, $\bar h^\perp \coloneqq \alpha_1 \bar w^{(1)} + B_1 \bar w^{(2)}$ and $\bar h^\parallel \coloneqq B^T_2 \bar w^{(1)}(te_1,v) + \alpha_2 \bar w^{(2)}(te_1,v)$ respectively. 

    Since $\bar u(0)$, $\bar h^\perp(0)$, $\bar h^\parallel(0)$ are all zero, we may thus let
    \[
        b(v) \coloneqq \nabla_V \bar u(0) \cdot v, \qquad b^\perp(v) \coloneqq \nabla_V h^\perp(0) \cdot v, \qquad b^\parallel(v) \coloneqq \nabla_V h^\parallel(0)\cdot v,
    \]
    which are the linearizations of these traces. The conclusion follows.
\end{proof}

\subsection{Conclusion} In this section we are going to show Proposition \ref{p:collapsed} and Proposition \ref{p:noncollapsed}, following \cite{DMS}*{Section~13.5}. We fix the decay scale $\varsigma_1$ and we will show that this will be reached at a certain radius, $r_c$ or $r_{nc}$, whether we are in the collapsed or non-collapsed setting, respectively, via a contradiction argument. We start with Proposition \ref{p:collapsed}; we fix a contradiction sequence $T_k$, $\mathbf{S}_k$, and $\Sigma_k$ as in the previous section and use Proposition \ref{p:final-blowup}, Proposition \ref{p:final-blowup-estimates}, and Proposition \ref{p:final-blowup} to extract the blow-up limits $\bar{A}_i$, $\bar{w}_i$, and find the functions $b^\perp$, $b^\parallel$ and $a_i^c$. As before, without loss of generality we have applied a rotation so that $V (\Sbf_k)$ coincide all with a fixed $V$ and the page $\Hbf_1\subset \pi_1$ is common to all the books $\Sbf_k$.

We build a skew-symmetric map of $\pi_1$ onto itself by mapping
\[
V\oplus (V^\perp\cap\pi_1) \ni y+x \mapsto b^\parallel (y) - (b^\parallel)^{T} (x)\, .
\]
This skew-symmetric map generates a one-parameter family $R [t]$ of rotations of $\pi_1$, which we may extend to all of $\R^{m+\bar{n}}$ by setting it to be the identity on $\pi_1^\perp$ and extended linearly. We next define the rotations
\[
R_k := R \left[\frac{\mathbb{E} (T_k, \mathbf{S}_k, \Bbf_1)^{1/2}}{\boldsymbol{\zeta} (\mathbf{S}_k)}\right]
\]
and observe that these rotations map $\pi_1$ and $\pi_1^\perp$ onto themselves. 

The rotated cones $\mathbf{S}'_k := R_k (\mathbf{S}_k)$ are thus a first step towards the cones which will have the desired decay at the radius $r_c$. We next take the linear functions
$A^k_i$ whose graphs over $\pi_1$ give the planes $\pi^k_i$, hence the linear functions $\xi_i$, and construct the maps 
\[
L^k_i := (A^k_i + \mathbb{E} (T_k, \mathbf{S}_k, \Bbf_1)^{1/2} (a^c_i + b^\perp))\circ R_k^{-1}\, .
\]
We now split the indices $i\in \{1, \ldots, N\}$ in $I^+\cup I^-$ according to Proposition \ref{p:final-blowup}(d) and end up defining two multi-valued linear functions:
\[
L^{k,+} := \sum_{i\in I^+} \llbracket L^k_i \rrbracket \qquad \text{and} \qquad
L^{k,-} := \sum_{i\in I^-} \llbracket L^k_i\rrbracket\, .
\]
The union of the graph of $L^{k,+}$ over $\Hbf_1$ and of the graph of $L^{k,-}$ over $-\Hbf_1$ give our final cone $\Sbf_k''$. Using the estimate in Proposition \ref{p:final-blowup-estimates} and Proposition \ref{p:final-decay} we get 
\[
\lim_{k\to \infty} \frac{\mathbb{E} (T_k, \mathbf{S}_k'', \Bbf_{r_c})}{\mathbb{E} (T_k, \mathbf{S}_k, \Bbf_1)} 
\leq C \varepsilon
\]
for any fixed radius $r_c$ smaller than $\rho_0$, where $C$ is a geometric constant, $\varepsilon$ is the fixed threshold with which we apply Proposition \ref{p:final-decay} and $\rho_0$ is the radius given by Proposition \ref{p:final-decay}. It is then obvious that it suffices to choose $C\varepsilon < \varsigma_1$ to reach the desired contradiction.

We now pass to the non-collapsed case. Since at this point we know that Proposition \ref{p:collapsed} has been established, for our purposes we can consider the parameter $\varepsilon_c^\star$ fixed and hence treat all the constants which depend on it as geometric. Again the argument is by contradiction. We again assume that $\varsigma_1$ and $\varepsilon_{c}^\star$ are given, that $r_{nc}$ is fixed, and that there is absence of decay by $\varsigma_1$ for sequences $T_k$, $\Sigma_k$, and $\mathbf{S}_k$. We then apply again Proposition \ref{p:final-blowup}, Proposition \ref{p:final-blowup-estimates}, and Proposition \ref{p:final-decay} to get, this time, the maps $\bar u_i$, $a^{nc}_i$, and $b$. 

First of all we consider $b$ as a map from $V$ to $V^\perp$, we let $b^T: V^\perp \to V$ be its transpose, we again build the skew-symmetric map
\[
V \oplus V^\perp \ni y+x \mapsto b (y) - b^T (x)
\]
and hence we let $t\mapsto R[t]$ be the one parameter family of rotations generated by it. As in the collapsed case we set
\[
R_k := R [ \mathbb{E} (T_k, \mathbf{S}_k, \Bbf_1)^{1/2}]
\]
and we consider the first adjustment to the cones as
\[
\mathbf{S}'_k := R_k (\mathbf{S}_k)\, .
\]
Next recall that along the sequence in our previous discussions we were considering $\Hbf^k_1$ to be always the same half-plane $\Hbf_1$ (by applying a suitable rotation), while we were assuming that $\Hbf^k_i$ converges to $\Hbf_i$. We then fix a rotation $O_{k,i}$ which maps $\Hbf^k_i$ to $\Hbf_i$ (this can be done ``canonically'' using for instance the argument of \cite{DMS}*{Lemma 3.7}). We then define the following multi-valued maps over $\Hbf_{k,i}' := R_k (\Hbf^k_i)$:
\[
L^k_i := R_k \circ O_{k,i} \circ (\mathbb{E} (T_k, \Sbf_k, \Bbf_1)^{1/2} a^{nc}_i) \circ O_{k,i}^{-1} \circ R_k^{-1}\, .
\]
These maps take values on the orthogonal complements of $R_k (\pi^k_i)$, where $\pi^k_i$ is the plane of dimension $m$ containing $\Hbf^k_i$ (i.e. obtained by completing the latter with its reflection along $V= V (\Sbf_k)$. The union of the graphs of these maps gives the new open book $\Sbf_k''$. Using the decay in Proposition \ref{p:final-decay}, the estimates in Proposition \ref{p:final-blowup-estimates} and the convergence in Proposition \ref{p:final-blowup} we then conclude as before that
\[
\lim_{k\to \infty} \frac{\mathbb{E} (T_k, \mathbf{S}_k'', \Bbf_{r_{nc}})}{\mathbb{E} (T_k, \mathbf{S}_k, \Bbf_1)} 
\leq C \varepsilon
\]
for any fixed radius $r_{nc}\leq \rho_0$, where $\varepsilon$ is the fixed threshold chosen for the application of Proposition \ref{p:final-decay} and $\rho_0$ the corresponding radius given by the proposition in the non-collapsed case. This time the constant $C$ can also depend on $\varepsilon^\star_c$, which however has been previously fixed. In particular, choosing $\varepsilon$ so that $C\varepsilon < \varsigma_1$ and $r_{nc}= \rho_0$, we reach the desired contradiction. This therefore completes the proofs of Proposition \ref{p:collapsed} and Proposition \ref{p:noncollapsed}, and hence completes the proof of Theorem \ref{t:decay}. \qed

\section{Proof of Theorem \ref{t:main} and Corollaries}

Having now completed the proof of the main excess decay theorem (Theorem \ref{t:decay}), we now use it to prove Theorem \ref{t:main} and its corollaries.

\begin{proof}[Proof of Theorem \ref{t:main}]
    Our proof closely follows those seen in \cite{W14_annals} and \cite{MW}*{Theorem 3.1}.

    Fix $\zeta = \frac{1}{8}$, and choose $\eps_0 = \eps_0(q,m,n,\bar{n},\zeta)\in (0,1/2]$ as in Theorem \ref{t:decay}. Let $p\in \spt^q(T)\cap U$ be a point at which there is a tangent cone $\Sbf$ with an $(m-1)$-dimensional spine. Then, we know that there is a radius $r>0$ such that $(\iota_{p,r})_\sharp T$ obeys the assumptions of Theorem \ref{t:decay} with the above choice of $\eps_0$; in particular, we may assume without loss of generality that $p = 0$ and $r=1$. Thus, Theorem \ref{t:decay} gives that there is some $\Sbf^\prime\in\mathscr{B}^q(0)\setminus\mathscr{P}(0)$ such that
    \begin{equation*}
        \mathbb{E}(T,\Sbf^\prime,\Bbf_{r_0}) \leq \frac{1}{4}\mathbb{E}(T,\Sbf,\Bbf_1);
    \end{equation*}
    \begin{equation*}
        \frac{\mathbb{E}(T,\Sbf^\prime,\Bbf_{r_0})}{\mathbf{E}^p(T,\Bbf_{r_0})} \leq \frac{1}{4}\frac{\mathbb{E}(T,\Sbf,\Bbf_1)}{\mathbf{E}^p(T,\Bbf_1)};
    \end{equation*}
    \begin{equation*}
    \dist^2(\Sbf^\prime\cap\Bbf_1,\Sbf\cap \Bbf_1) \leq C\mathbb{E}(T,\Sbf,\Bbf_1);
    \end{equation*}
    \begin{equation*}
        \dist^2(V(S)\cap\Bbf_1, V(\Sbf^\prime)\cap\Bbf_1)\leq C\frac{\mathbb{E}(T,\Sbf,\Bbf_1)}{\mathbf{E}^p(T,\Bbf_1)}.
    \end{equation*}
Here, $r_0 = r_0(q,m,n,\bar{n})\in (0,1/2]$ and $C = C(q,m,n,\bar{n})$. We note that since $T$ is assumed to be close to a non-flat cone, we in fact have $\mathbf{E}^p(T,\Bbf_r)\geq \delta>0$ for some fixed $\delta>0$ independent of $r$ (but dependent on the cone $\Sbf$), and so we could remove the denominators involving $\mathbf{E}^p(T,\Bbf_1)$ if we wanted to, allowing the constants to depend on $\Sbf$, however we will keep the more general form of these inequalities for now. In particular, all the assumptions of Theorem \ref{t:decay} still hold with $(\iota_{0,r_0})_\sharp T$, $\Sbf^\prime$ in place of $T,\Sbf$, respectively, and so we may iterate this. The outcome is that for each $k\in \{0,1,2,\dotsc\}$, we find a cone $\Sbf_k\in \mathscr{B}^q(0)\setminus\mathscr{P}(0)$, where $\Sbf_0 = \Sbf$, $\Sbf_1 = \Sbf^\prime$, such that
\begin{equation*}
    \mathbb{E}(T,\Sbf_{k+1},\Bbf_{r_0^{k+1}}) \leq \frac{1}{4}\mathbb{E}(T,\Sbf_k,\Bbf_{r_0^k}) \leq \cdots \leq \frac{1}{4^{k+1}}\mathbb{E}(T,\Sbf,\Bbf_1);
\end{equation*}
\begin{equation*}
    \frac{\mathbb{E}(T,\Sbf_{k+1},\Bbf_{r_0^{k+1}})}{\mathbf{E}^p(T,\Bbf_{r_0^{k+1}})} \leq \frac{1}{4}\cdot\frac{\mathbb{E}(T,\Sbf_k,\Bbf_{r_0^k})}{\mathbf{E}^p(T,\Bbf_{r_0^k})} \leq \cdots \leq \frac{1}{4^{k+1}}\cdot\frac{\mathbb{E}(T,\Sbf,\Bbf_1)}{\mathbf{E}^p(T,\Bbf_1)};
\end{equation*}
\begin{equation}\label{e:decay-3}
    \dist^2(\Sbf_{k+1}\cap\Bbf_1,\Sbf_k\cap \Bbf_1) \leq C\mathbb{E}(T,\Sbf_k,\Bbf_{r_0^k}) \leq \frac{C}{4^k}\mathbb{E}(T,\Sbf,\Bbf_1);
\end{equation}
\begin{equation}\label{e:decay-4}
    \dist^2(V(\Sbf_{k+1})\cap\Bbf_1,V(\Sbf_k)\cap\Bbf_1) \leq C\frac{\mathbb{E}(T,\Sbf_k,\Bbf_{r_0^k})}{\mathbf{E}^p(T,\Bbf_{r_0}^k)} \leq \frac{C}{4^k}\frac{\mathbb{E}(T,\Sbf,\Bbf_1)}{\mathbf{E}^p(T,\Bbf_1)}.
\end{equation}
Then \eqref{e:decay-3} and \eqref{e:decay-4} readily imply that the sequences $(\Sbf_k)_k$ and $(V(\Sbf_k))_k$ are Cauchy sequences in the Hausdorff distance topology, and thus we have $\Sbf_k\to \Sbf_*$ and $V(\Sbf_k)\to V_*$ locally in Hausdorff distance for some $\Sbf_*\in \mathscr{B}^q(0)$ and $(m-1)$-dimensional linear subspace $V_*$, and moreover $V_* = V(\Sbf_*)$. Moreover, the triangle inequality gives for each $k\in \{0,1,2,\dotsc\}$,
\begin{equation*}
    \mathbb{E}(T,\Sbf_*,\Bbf_{r_0^k}) \leq \frac{1}{4^k}\mathbb{E}(T,\Sbf,\Bbf_1);
\end{equation*}
\begin{equation*}
    \frac{\mathbb{E}(T,\Sbf_*,\Bbf_{r_0^k})}{\mathbf{E}^p(T,\Bbf_{r_0^k})} \leq \frac{1}{4^k}\cdot\frac{\mathbb{E}(T,\Sbf,\Bbf_1)}{\mathbf{E}^p(T,\Bbf_1)};
\end{equation*}
\begin{equation*}
    \dist^2(\Sbf_*\cap\Bbf_1,\Sbf_k\cap\Bbf_1)\leq \frac{C}{4^k}\mathbb{E}(T,\Sbf,\Bbf_1);
\end{equation*}
\begin{equation*}
    \dist^2(V(\Sbf_*)\cap\Bbf_1,V(\Sbf_k)\cap\Bbf_1) \leq \frac{C}{4^k}\cdot\frac{\mathbb{E}(T,\Sbf,\Bbf_1)}{\mathbf{E}^p(T,\Bbf_1)},
\end{equation*}
    the second of which evidently implies that $\Sbf_*\in \mathscr{B}^q(0)\setminus\mathscr{P}(0)$.
Moreover, a standard interpolation argument between the geometric sequence of scales $r_0,r_0^2,\dotsc$ gives the following estimate for the \emph{one-sided} excess:
\begin{equation*}
    \hat\Ebf(T,\Sbf_*,\Bbf_r) \leq Cr^{2\beta}\mathbb{E}(T,\Sbf,\Bbf_1) \qquad \text{for all }r\in (0,1),
\end{equation*}
for a suitable $\beta = \beta(q,m,n,\bar{n})\in (0,1)$. This shows that $\Sbf_*$ is the unique tangent cone to $T$ at $0$.

We now wish to repeat this argument about any nearby point $x$ with $\Theta(T,x)\geq\frac{q}{2}$. In fact, since $\mathbf{E}^p(T,\Bbf_1)\geq \delta>0$ with $\delta$ depending on $\Sbf$, and $\mathbb{E}(T,\Sbf,\Bbf_1)\leq \eps_0^2\mathbf{E}^p(T,\Bbf_1) \leq C\eps_0^2$, one can check (c.f. Proposition \ref{p:Simon-shift} and the arguments therein) that $(\iota_{x,1/2})_\sharp T$ obeys the assumptions of Theorem \ref{t:decay} (up to increasing $\eps_0$ by some fixed amount) for any such $x\in \Bbf_{1/4}$. Although we will only need this in the proof of Corollary \ref{c:mult-1}, we remark also that one can check the same claim without assuming largeness of the planar excess $\mathbf{E}^p(T,\Bbf_1)$, by first pruning the cone via Lemma \ref{l:pruning} to get a new cone where the excess is much smaller than the minimal angle in the pruned cone, and then translating (again, c.f. Proposition \ref{p:Simon-shift} and the arguments therein).

Thus, what we currently have is that at any point $x\in \mathcal{S}\cap \Bbf_{1/4}$, i.e. with $\Theta(T,x)\geq\frac{q}{2}$, there exists a cone $\Sbf_x\in \mathscr{B}^q(0)\setminus\mathscr{P}(0)$ with, writing $T_{x,1/2} = (\iota_{x,1/2})_\sharp T$
\begin{equation*}
    \dist^2(\Sbf_x\cap\Bbf_1,\Sbf\cap\Bbf_1) \leq C\mathbb{E}(T,\Sbf,\Bbf_1);
\end{equation*}
\begin{equation*}
    \dist^2(V(\Sbf_x)\cap\Bbf_1,V(\Sbf)\cap\Bbf_1) \leq C\mathbb{E}(T,\Sbf,\Bbf_1);
\end{equation*}
\begin{equation}\label{e:decay-at-x-1}
    \hat\Ebf(T_{x,1/2},\Sbf_x,\Bbf_r) \leq Cr^{2\beta}\mathbb{E}(T,\Sbf,\Bbf_1) \qquad \text{for all }r\in (0,1);
\end{equation}
\begin{equation}\label{e:decay-at-x-2}
    \hat\Ebf(\Sbf_x,T_{x,1/2},\Bbf_{r_0^k}) \leq C(r^k_0)^{2\beta}\mathbb{E}(T,\Sbf,\Bbf_1)
\end{equation}
for all $k\in\{0,1,2,\dotsc\}$. The uniqueness of the tangent cone $\Sbf_x$ follows immediately. Let us now use \eqref{e:decay-at-x-1} and \eqref{e:decay-at-x-2} to verify the remaining claims of the theorem. Indeed, we first claim that for each $y\in V(\Sbf)\cap \Bbf_{1/8}$, we have
\begin{equation}\label{e:decay-slices}
    \mathcal{S}\cap \mathbf{p}_V^{-1}(y) \neq\emptyset
\end{equation}
i.e. every slice orthogonal to $V(\Sbf)$ contains a point of density at least $\frac{q}{2}$. Indeed, if this were false for some $y\in V(\Sbf)\cap \Bbf_{1/8}$, then since $\mathcal{S}\cap \Bbf_{1/8}$ is a relatively closed subset of $\Bbf_{1/8}$ and $0\in \mathcal{S}$, we can find $r\in (0,1/8)$ such that
$$\mathcal{S}\cap \{z: |\mathbf{p}_V(z)-y|<r\} = \emptyset$$
yet
$$\mathcal{S}\cap \{z\in\mathbf{B}_{1/8}:|\mathbf{p}_V(z)-y|=r\} \neq\emptyset.$$
But then if we choose $x\in \mathcal{S}\cap \{z\in\Bbf_{1/8}:|\mathbf{p}_V(z)-y|=r\}$, by \eqref{e:decay-at-x-1} and \eqref{e:decay-at-x-2}, we can recenter about $x$ at a sufficiently small scale and apply Theorem \ref{t:no-gaps} to get a contradiction (as to one side of $x$ there would be a gap by construction).

Moreover, note that for $x\in \mathcal{S}$, for each $\rho\in (0,1/4)$ we see from \eqref{e:decay-at-x-1} and \eqref{e:decay-at-x-2} and \eqref{e:shift-2} that
$$\mathcal{S}\cap \{y\in \Bbf_\rho(x): |\mathbf{p}^\perp_{V(\Sbf_x)}|\geq \rho/8\} = \emptyset.$$

This combined with \eqref{e:decay-slices} implies that for each $y\in V(\Sbf)\cap \Bbf_{1/16}$, we have that $\mathcal{S}\cap \mathbf{p}_V^{-1}(y)$ is a unique point, and thus we see that
$$\mathcal{S}\cap \Bbf_{1/16} = \graph(\phi)$$
is the graph of a function $\phi: V(\Sbf)\cap \Bbf_{1/16}\to V(\Sbf)^\perp$. Our above decay estimates tell us that $\phi$ is Lipschitz and moreover that, when defined, the tangent space to the graph of $\phi$ at $(x,\phi(x))$ is indeed the spine of the corresponding cone $\Sbf_{\phi(x)}$ (see for instance \cite{W14_annals}*{(16.33), (16.39)}).

All that remains is to show that $\phi$ is $C^{1,\beta}$; indeed, once we have this we have already justified the other claims. For this, we can take two points $x_1,x_2\in\mathcal{S}\cap\Bbf_{1/16}$ and recenter $T$ at $x_1$ and rescale by $\frac{1}{2}r_0^k$, where $k$ is such that $r_0^{k+1}<2|x_1-x_2|\leq r_0^k$, i.e. consider $\tilde{T} = (\iota_{x_1,\frac{1}{2}r_0^k})_\sharp T$. The proof is then completed by noting that the point $\tilde{x} := \frac{2(x_2-x_1)}{r_0^k}$ is in the support of $\tilde{T}$ and has density $\frac{q}{2}$, and so we may apply the previous iteration argument based on Theorem \ref{t:decay} to $\tilde{T}$ at $\tilde{x}$. This gives a unique tangent cone to $\tilde{T}$ at $\tilde{x}$, which we can then relate back to the tangent cone to $T$ at $x_2$ by rescaling and translation. The above estimates established from the application of Theorem \ref{t:decay} then give the desired $C^{1,\beta}$ control required to deduce the regularity of $\phi$. We omit the details here and refer the readers to similar arguments in, for example, \cite{W14_annals}*{(16.39) -- (16.51)} or \cite{DLHMSS}*{Proof of Corollary 13.1}. This completes the proof of Theorem \ref{t:main}.
\end{proof}

Finally, Corollary \ref{c:mult-1} follows from the argument given in the proof of Theorem \ref{t:main}, using the additional information that all the multiplicities in the cone are $1$, and thus Allard's regularity theorem applies to give a full sheeting on a neighbourhood of $\mathcal{S}$ (using the notation from Theorem \ref{t:main}) as well as away from $\mathcal{S}$. The argument is essentially identical to that in \cite{W14_annals}*{Theorem 16.1} (see also \cite{MW}*{Theorem 3.1}), and so we omit the details.
\qed

\begin{remark}
To get the full Remark \ref{remark:main} requires more work in the above proof of Theorem \ref{t:main}, as one needs to not allow the constants to depend on the cone, and so cannot use the fact that $\mathbf{E}^p(T,\Bbf_1)\geq\delta>0$, with $\delta$ depending on the cone. In order to get Remark \ref{remark:main}, one needs to instead argue more closely to \cite[Theorem 3.1]{MW}. In fact, one do not need the full quantitative version of this, but just the statement on some neighborhood of the point (in particular, giving the tangent cone behaviour). Again, we leave the details to the reader due to similarity.
\end{remark}

\part{Rectifiability of flat singularities and $\Hcal^{m-2}$-a.e. uniqueness of tangent cones}\label{pt:2}

\section{Setup and preliminaries}
We are now in a position to prove Theorem \ref{c:main}. As previously observed, for most of this part, we only need to treat the case of even $q$. Nevertheless, for the conclusion (e) of Theorem \ref{c:main}, one needs to also treat the case of odd moduli $q$.

We henceforth work with the following underlying assumptions, which we may do without loss of generality after localizing, rescaling and translating (c.f. Assumption \ref{a:main}).
\begin{assumption}\label{a:main-pt2}
    Assume that $m, n, \bar n \geq 2$, $q \geq 3$ are positive integers, and that
\begin{itemize}
\item[(i)] $\Sigma\subset \mathbb R^{m+n}$ is a complete embedded $C^{3,\alpha_0}$ $(m+\bar n)$-dimensional submanifold for some positive $\alpha_0>0$;
\item[(ii)] $T$ is a representative mod$(q)$ of an area-minimizing flat chain mod$(q)$ in $\Sigma \cap \Bbf_{7\sqrt{m}}$ with $\partial T \res \Bbf_{7\sqrt{m}} = 0$  mod$(q)$;
\item[(iii)] $\Sigma \cap \Bbf_{7\sqrt{m}}$ is the graph of a $C^{3,\alpha_0}$ function $\Psi_p : \Trm_p\Sigma \cap \Bbf_{7\sqrt{m}} \to \Trm_p\Sigma^\perp$ for every $p \in \Sigma\cap \Bbf_{7\sqrt{m}}$, with
	\[
	\boldsymbol{c}(\Sigma,\Bbf_{7\sqrt{m}}):=\sup_{p \in \Sigma \cap \Bbf_{7\sqrt{m}}}\|D\Psi_p\|_{C^{2,\alpha_0}} \leq \bar{\eps},
	\]
	where $\bar\eps$ will be determined later.
\end{itemize}
The property (iii) in particular gives us the following uniform control on the second fundamental form $A_\Sigma$ of $\Sigma$:
	\[
	\Abf_\Sigma := \|A_\Sigma\|_{C^0(\Sigma\cap \Bbf_{7\sqrt{m}})} \leq C_0\boldsymbol{c}(\Sigma,\Bbf_{7\sqrt{m}}) \leq C_0 \bar\eps.
	\]
\end{assumption}

The remaining conclusions of Theorem \ref{c:main} that we are yet to verify may thus be reformulated as follows.

\begin{theorem}\label{t:main-pt2-initial}
Suppose that $m$, $n$, $\bar n$, $q$, $T$, $\Sigma$ satisfy Assumption \ref{a:main-pt2}. Then
\begin{itemize}
    \item[(a)] The set of flat singular points $\Ffrak(T)$
is countably $(m-2)$-rectifiable and the unique tangent cone of $T$ at $\Hcal^{m-2}$-almost every $p\in \Ffrak_Q(T)$ is $q^\prime\llbracket \pi \rrbracket$ for some $m$-dimensional oriented plane $\pi$ and positive integer $q^\prime\leq q/2$.
\item[(b)] Suppose for some positive integer $N\leq Q$ and positive integers $Q_i\geq 1$, the tangent cone to $T$ at a given singular point is of the form $\sum^N_{i=1}Q_i\llbracket \pi_i\rrbracket$ for a collection of oriented planes $\pi_1,\dotsc,\pi_N$ with $\sum_i Q_i\leq q/2$ and $\cap^N_{i=1}\pi_i = V$ for some $(m-2)$-dimensional subspace $V$. Then, at $\mathcal{H}^{m-2}$-a.e. such point, this is the unique tangent cone there.
\end{itemize}
\end{theorem}
We note that in fact the only remaining content for Theorem \ref{t:main-pt2-initial} is actually when $q$ is even. Indeed, notice that in both (a) and (b) of Theorem \ref{t:main-pt2-initial}, the tangent cone in question has integer density. If $q$ is odd, then the maximal density is $q/2$ (c.f. \cite[(4.1)]{DLHMS}), which is a half-integer, and thus in both cases the tangent cone in question has density $\leq (q-1)/2$. But then by upper semi-continuity of the density and Proposition \ref{p:integrality}, $T$ must locally about such singular points identify with an area minimizing integral current, and thus both (a) and (b) of Theorem \ref{t:main-pt2-initial} follow from our previous work \cite[Theorem 1.1]{DMS}. Thus, we only need to consider the case where $q=2Q$ is even. Moreover, by the same argument as above, we only need to deal with the case when the tangent cone has density exactly $q/2 = Q$, as otherwise the tangent cone has density at most $Q-1$, and so we can appeal again to Proposition \ref{p:integrality} to reduce to the area minimizing case, which is a consequence of \cite[Theorem 1.1]{DMS}. 

Hence, we are left with showing:

\begin{theorem}\label{t:main-pt2}
Suppose that $m$, $n$, $\bar n$, $q =2Q$, $T$, $\Sigma$ satisfy Assumption \ref{a:main-pt2}. Then
\begin{itemize}
    \item[(a)] The set of flat singular points
\[
    \Ffrak_Q(T) \coloneqq \{p\in \Ffrak(T)\cap\Bbf_{6\sqrt{m}}: \Theta(T,p) {=} Q\}
\]
is countably $(m-2)$-rectifiable and the unique tangent cone of $T$ at $\Hcal^{m-2}$-almost every $p\in \Ffrak_Q(T)$ is $Q\llbracket \pi \rrbracket$ for some $m$-dimensional oriented plane $\pi$;
\item[(b)] {Suppose for some positive integer $N\leq Q$ and positive integers $Q_i\geq 1$, the tangent cone to $T$ at a given singular point is of the form $\sum^N_{i=1}Q_i\llbracket \pi_i\rrbracket$ for a collection of oriented planes $\pi_1,\dotsc,\pi_N$ with $\sum_i Q_i = Q$ and $\cap^N_{i=1}\pi_i = V$ for some $(m-2)$-dimensional subspace $V$. Then, at $\mathcal{H}^{m-2}$-a.e. such point, this is the unique tangent cone there.}
\end{itemize}
\end{theorem}

{Henceforth, we will take $q=2Q$ even.}

\begin{remark}
    It is perhaps informative to note that if $T_k$ is a sequence of area minimizers mod$(2Q)$ which is given by the sum of planes which all meet exactly along a $(m-2)$-dimensional subspace, that $T_k$ cannot converge to a non-trivial open book $T$. Indeed, as we saw in Part \ref{pt:1}, for an area-minimizing mod$(2Q)$ current supported on an open book, the orientations must all be pointing towards the spine or all away from the spine. As each plane in $T_k$ is oriented in a fixed manner, this is incompatible with limiting onto $T$.
\end{remark}

\section{Singularity degree and main reduction}
\subsection{Fine blow-ups and compactness}\label{ss:compactness}
Let us recall the compactness procedure of \cite[Part 4]{DLHMS} relative to center manifolds around a given point $x\in\Ffrak_Q(T)$. We may without loss of generality henceforth work with the following additional assumption.

\begin{assumption}\label{a:main-2}
    Suppose that $m$, $n$, $\bar n$, $q=2Q$, $T$, $\Sigma$ are as in Assumption \ref{a:main}. Suppose that $0\in \Ffrak_Q(T)$.
\end{assumption}

Recall that the non-oriented tilt excess $\Ebf^{no}(T,\Cbf_r(p,\pi),\varpi)$ of $T$ in a given cylinder $\Cbf_r(p,\pi)$ relative to a given $m$-dimensional plane $\varpi$ is defined by
\[
    \Ebf^{no}(T,\Cbf_r(p,\pi),\varpi) := \frac{1}{2\omega_m r^m} \int_{\Cbf_r(p,\pi)} (\min\{|\vec{T}-\vec{\varpi}|, |\vec{T}+\vec{\varpi}|\})^2 \, d\|T\|,
\]
and the non-oriented tilt excess $\Ebf^{no}(T,\Cbf_r(p,\pi))$ is then defined by
\[
    \Ebf^{no}(T,\Cbf_r(p,\pi)):= \min_{\substack{\text{$m$-dim. planes} \\ \varpi\subset \R^{m+n}}} \Ebf^{no}(T,\Cbf_r(p,\pi),\varpi).
\]
The quantities $\Ebf^{no}(T,\Bbf_r(p),\varpi)$ and $\Ebf^{no}(T,\Bbf_r(p))$ are defined analogously. 

As in \cite{DLHMS}, we may subdivide the interval $(0,1]$ of scales around $0$ into a collection of mutually disjoint intervals $(s_j,t_j]$, referred to as \emph{intervals of flattening}, such that for every $r\in(s_j,t_j]$,
\begin{itemize}
    \item $\Ebf^{no}(T,\Bbf_{6\sqrt{m}r}) \leq \eps_3^2$,
    \item $\Ebf^{no}(T,\Bbf_{6\sqrt{m}r}) \leq C (\tfrac{r}{t_j})^{2-2\delta_2} \mbf_{0,j}$,
\end{itemize}
where
\[
    \mbf_{0,j} := \max\{\Ebf^{no}(T,\Bbf_{6\sqrt{m}t_j}), \bar\eps^2 t_j^{2-2\delta_2} \},
\]
and $\eps_3,\delta_2$ are as in \cite{DLHMS}. Note that this definition of $\mbf_{0,j}$ may indeed be taken in place of the one in \cite{DLHMS}, in light of the scaling of $\boldsymbol{c}(\Sigma_{0,t_j},\Bbf_{7\sqrt{m}})$. In each interval $(s_j,t_j]$, we may then construct a center manifold $\Mcal_j$ for $T_{0,t_j}\mres\Bbf_{6\sqrt{m}}$ with associated normal approximation $N_j$ as in \cite{DLHMS}*{Section 17.2}.

In light of the above construction, we additionally impose the following.

\begin{assumption}\label{a:main-3}
    Suppose that $m$, $n$, $\bar n$, $q=2Q$, $T$, $\Sigma$ are as in Assumption \ref{a:main-2}. Suppose in addition that $\bar\eps$ is chosen small enough so that $t_0=1$ and $\mbf_{0,0}\leq \eps_3^2$.
\end{assumption}

We refer to a sequence $r_k\downarrow 0$ as a \emph{blow-up sequence of radii} around $x\in \Ffrak_Q(T)$ if the rescalings $T_{x,r_k}\mres \Bbf_{6\sqrt{m}} \toweakstar Q\llbracket \pi_0 \rrbracket$ for some $m$-dimensional plane $\pi_0$. If $x=0$, we will omit reference to the center point.

Having fixed a blow-up sequence of radii $\{r_k\}$, observe that for each $k$ sufficiently large there exists a unique $j(k)$ such that $r_k \in (s_{j(k)}, t_{j(k)}]$. We use the notation $T_k$ and $\Sigma_k$ respectively for the rescaled currents $T_{0, t_{j(k)}} \res \Bbf_{6\sqrt{m}}$ and ambient manifolds $\Sigma_{0, t_{j(k)}}$. We let $\Mcal_k$ and $N_k: \Bcal_3 \subset \Mcal \to \Ascr_Q(\R^{m+n})$ denote, respectively, the corresponding center manifolds and normal approximations for $T_k$. We refer the reader to e.g. \cite{DLHMS_linear} for the relevant background for the notion of a \emph{special $Q$-valued function} that takes values in the space $\Ascr_Q$. By a suitable small rotation of coordinates, we may further assume that the plane $\pi_k\subset T_0\Sigma_k$ optimizing the excess $\Ebf(T_k, \Bbf_{6\sqrt{m}})$ over which the center manifold $\Mcal_k$ is parameterized, satisfies $\pi_k \equiv \pi_0 \equiv \R^m\times \{0\}^n \subset\R^{m+n}$. We denote by $\boldsymbol\varphi_k: B_3(\pi_0)\to \bar\Mcal_k$ the function parameterizing  $\bar\Mcal_k$, and we let $\mathbf\Phi_k(x):=(x,\boldsymbol\varphi_k(x))$ denote its graph.

\subsection{Fine blow-ups}
Let $T$ and $\Sigma$ satisfy Assumption \ref{a:main-3}. Let $\tfrac{\bar s_k}{t_{j_k}}\in \big(\tfrac{3r_k}{2t_{j(k)}},\tfrac{3r_k}{t_{j(k)}}\big]$ be the scale $\sigma$ at which the reverse Sobolev inequality \cite{DLHMS}*{Corollary 27.3} holds for $r=\tfrac{r_k}{t_{j(k)}}$. Let $\bar r_k:=\tfrac{2\bar s_k}{3t_{j(k)}}\in \big(\tfrac{r_k}{t_{j(k)}}, \tfrac{2r_k}{t_{j(k)}}\big]$, and in turn define the additional rescalings
\[
    \bar T_k := (\iota_{0,\bar r_k})_\sharp T_k = \big((\iota_{0,\bar r_k t_{j(k)}})_\sharp T\big)\mres \Bbf_{6\sqrt{m}/\bar r_k}, \qquad \bar\Sigma_k := \iota_{0,\bar r_k}(\Sigma_k), \qquad \bar \Mcal_k := \iota_{0,\bar r_k}(\Mcal_k).
\]
Moreover, let $\bar N_k:\bar\Mcal_k \to \Ascr_Q(\R^{m+n})$ be given by
\[
    \bar N_k(p) := \frac{1}{\bar r_k} N_k(\bar r_k p).
\]
Now, let $\mathbf{e}_k: T_{p_k}\bar \Mcal_k \simeq \pi_0 \to\bar\Mcal_k$ denote the exponential map for $\bar\Mcal_k$, with $p_k:= \frac{\mathbf\Phi_k(0)}{\bar r_k}$. Consequently, let $u_k: B_3(\pi_0) \to \Ascr_Q(\R^{m+n})$ be given by
\[
    u_k := \frac{\bar N_k\circ \mathbf{e}_k}{\mathbf{h}_k},
\]
where $\mathbf{h}_k := \|\bar N_k\|_{L^2(\Bcal_{3/2})}$.

In light of the reverse Sobolev inequality \cite{DLHMS}*{Corollary 27.3}, the maps $u_k$ are uniformly bounded in $W^{1,2}(B_{3/2}(0,\pi_0);\Ascr_Q(\R^{m+n}))$. Following \cite{DLHMS}, up to extracting a subsequence, there exists a special $Q$-valued Dir-minimizer $u:B_{3/2}(\pi_0)\to\Ascr_Q(\pi_0^\perp)$ satisfying $\boldsymbol\eta\circ u = 0$ and $\|u\|_{L^2(B_{3/2})}=1$, such that
\[
    u_k \longrightarrow u \quad \text{strongly in $W^{1,2}_{\loc}\cap L^2(B_{3/2}(\pi_0))$.}
\]

\begin{definition}
    Any special $Q$-valued Dir-minimizer $u\in W^{1,2}(B_{3/2}(\pi_0);\Ascr_Q(\pi_0^\perp))$ obtained via the above compactness procedure along (a subsequence of) some blow-up sequence of radii $r_k$ is called a \emph{fine blow-up} of $T$ (around $0$).

    Via translation and rotation, we analogously define a fine blow-up around any other point $x\in\Ffrak_Q(T)$.
\end{definition}

\subsection{Frequency function and singularity degree}
Let $u: B_{3/2} \to \Ascr_Q(\R^{m+n})$ be a Dir-minimizer. Let $\phi:[0,\infty)\to[0,1]$ be a monotone non-increasing Lipschitz cutoff that vanishes for all $t\geq 1$ and is identically equal to $1$ for all $t$ sufficiently small. For $x\in B_{3/2}$ and $r\in (0,\dist(x,\partial B_{3/2}))$, we may introduce the quantities
\begin{align*}
    D_u(x,r) &:= \int |Du(y)|^2 \phi\left(\tfrac{|y-x|}{r}\right) \, dy, \\
    H_u(x,r) &:= -\int \frac{|u(y)|^2}{|y-x|} \phi'\left(\tfrac{|y-x|}{r}\right) \, dy, \\
    I_u(x,r) &:= \frac{r D_u(x,r)}{H_u(x,r)}.
\end{align*}
$I_u$ is (a regularization of) Almgren's frequency function for $u$, and the computations within \cite{DLHMS_linear}*{Section 9} may be repeated for this regularization to deduce that $r\mapsto I_u(x,r)$ is monotone non-decreasing in $r$ and is equal to a constant value $\alpha>0$ if any only if $u$ is radially $\alpha$-homogeneous about $x$. In particular, the limit
\[
    I_{u,x}(0) := \lim_{r\downarrow 0} I_u(x,r)
\]
exists and one can readily check that it is independent of the choice of $\phi$. 

We will henceforth fix the following convenient choice of $\phi$:
\[
    \phi(t):=\begin{cases}
        1 & 0\leq t\leq \tfrac{1}{2} \\
        2-2t & \tfrac{1}{2} \leq t \leq 1 \\
        0 & \text{otherwise.}
    \end{cases}
\]
We will often omit dependency of $D_u$, $H_u$ and $I_u$ on $x$ in the case when $x=0$.

We are now in a position to introduce the singularity degree, which is defined analogously to that in \cite{DLSk1}. 

\begin{definition}
Given $T$ as in Assumption \ref{a:main-3}, we define the collection $\Fcal(T,0)$ of \emph{frequency values} of $T$ around 0 by
\[
    \Fcal(T,0) := \{I_u(0) : \text{$u$ is a fine blow-up of $T$ along some $r_k \downarrow 0$}\}.
\]
We define the \emph{singularity degree} of $T$ at $0$ by
\[
    \Irm(T,0) := \inf\{\alpha : \alpha\in \Fcal(T,0)\}.
\]
The singularity degree $\Irm(T,x)$ of $T$ at any other point $x\in\Ffrak_Q(T)$ is analogously defined.
\end{definition}

We record the following key properties relating to the singularity degree of $T$.

\begin{theorem}\label{t:sing-deg}
    Suppose that $T$ satisfies Assumption \ref{a:main-3}. Then
    \begin{itemize}
        \item[(i)] $\Irm(T,0) \geq 1$ and $\Fcal(T,0) = \{\Irm(T,0)\}$;
        \item[(ii)] All fine blow-ups are radially homogeneous with degree $I(T,0)$;
        \item[(iii)] if $s_{j_0}=0$ for some $j_0\in \N$ then $\lim_{r\downarrow 0} \Ibf_{N_{j_0}}(r) = \Irm(T,0)$;
        \item[(iv)] if, conversely, there are infinitely many intervals of flattening $(s_k,t_k]$, the functions $\Ibf_{N_j}$ converge uniformly to the constant function $\Irm(T,0)$ when $\Irm(T,0)>1$, while when $\Irm(T,0)=1$, $\lim_{k \to \infty} \Ibf_{j(k)}(\tfrac{r_k}{t_{j(k)}}) = \Irm(T,0) =1$ for every blow-up sequence of radii $r_k$;
        \item[(v)] if $\Irm(T,0)>1$ then the rescalings $T_{0,r}$ converge polynomially fast to a unique flat tangent cone $Q\llbracket \pi_0 \rrbracket$ as $r\downarrow 0$;
        \item[(vi)] if additionally $\Irm(T,0)> 2-\delta_2$ then $s_{j_0} = 0$ for some $j_0\in\N$;
        \item[(vii)] if $\Irm(T,0) < 2-\delta_2$ then there are infinitely many intervals of flattening and $\inf_j \frac{s_j}{t_j} > 0$.
    \end{itemize}
\end{theorem}
We omit the proof of Theorem \ref{t:sing-deg} here, and simply observe that it follows by the exact same arguments as those in \cite{DLSk1}, with the appropriate preliminary results from \cite{DLS_multiple_valued,DLS16centermfld,DLS16blowup} replaced by their analogues in \cite{DLHMS_linear, DLHMS}. Of course, by translation, we obtain the same consequences of Theorem \ref{t:sing-deg} for $\Irm(T,x)$ for any $x\in \Ffrak_Q(T)$. 

A particularly important byproduct of the proof of Theorem \ref{t:sing-deg} is the following BV-estimate on the frequency function.

\begin{theorem}\label{t:BV}
    There exists $\gamma_4=\gamma_4(m,n,Q)>0$ and $C=C(m,n,Q)>0$ such that the following holds. Let $T$ satisfy Assumption \ref{a:main-3}. Suppose that $\{(s_k,t_k]_{k=j_0}^J\}$ is a sequence of intervals of flattening with coinciding endpoints, i.e. $s_k=t_{k+1}$. For $r\in (s_J, t_{j_0}]$, let
    \[
        \Ibf(r) := \Ibf_{N_k}(\tfrac{r}{t_k}) \mathbf{1}_{(s_k,t_k]}.
    \]
    Then $\log(\Ibf +1)\in \BV((s_J,t_{j_0}])$ with the quantitative estimate
    \begin{equation}
        \left|\left[\frac{d\log(\Ibf+1)}{dr}\right]_-\right|\big((s_J,t_{j_0}]\big) \leq C \sum_{k=j_0}^{J} \mbf_{0,k}^{\gamma_4}.
    \end{equation}
    In addition, if $(a,b]\subset (s_k,t_k]$ for some interval of flattening $(s_k,t_k]$, we have
    \begin{equation}
        \left|\left[\frac{d\log(\Ibf+1)}{dr}\right]_-\right|\big((a,b]\big) \leq C \left(\frac{b}{t_k}\right)^{\gamma_4}\mbf_{0,k}^{\gamma_4}.
    \end{equation}
\end{theorem}

\section{Reduction to degree 1 points}
In light of Theorem \ref{t:sing-deg}, the proof of Theorem \ref{t:main-pt2}{(a)} can therefore be reduced to the following.
\begin{theorem}\label{t:reduction}
    Let $T$ and $\Sigma$ be as in Theorem \ref{t:main-pt2}. Then we may write
    \[
        \Ffrak_Q(T) = \Ffrak_{Q,> 1}(T) \sqcup \Ffrak_{Q,1}(T),
    \]
    as a disjoint union, where
    \[
        \Ffrak_{Q,> 1}(T) \coloneqq \{x\in \Ffrak(T): \Irm(T,x) > 1\}, \qquad \Ffrak_{Q,1}(T) \coloneqq \{x\in \Ffrak(T): \Irm(T,x) = 1\}.
    \]
    Moreover,
    \begin{equation}\label{e:high-degree-dim-est}
       \Ffrak_{Q,>1}(T) \ \text{is countably $(m-2)$-rectifiable,}
    \end{equation}
    and
    \begin{equation}\label{e:degree-1-dim-est}
        \Hcal^{m-2}(\Ffrak_{Q,1}(T)) =0.
    \end{equation}
\end{theorem}

\subsection{Proof of \eqref{e:high-degree-dim-est}}
The conclusion \eqref{e:high-degree-dim-est} follows easily from a combination of \cite{Sk-modp} and \cite{DLSk2}. Indeed, given the BV estimate of Theorem \ref{t:BV} and the associated variational estimates (see \cite{DLSk1}*{Lemma 6.4}), together with the unique continuation and homogeneity results of \cite{Sk-modp}*{Section 5.1}, the arguments of \cite{DLSk2} remain valid in the case when $T$ is area-minimizing mod$(q)$.

Thus, the remainder of this article is dedicated to proving \eqref{e:degree-1-dim-est} and Theorem \ref{t:main-pt2}(b).

\section{Points of singularity degree 1}
In order to obtain the conclusion \eqref{e:degree-1-dim-est} of Theorem \ref{t:reduction}, we will closely follow the arguments of \cite{DMS}. Indeed, we will often demonstrate that we may suitably reduce to the conclusions therein.

Before proceeding, we first recall another key consequence of the conclusions in Theorem \ref{t:sing-deg}. Let us recall the notion of a coarse blow-up, which is defined as follows.

Let $\{r_k\}_k$ be a blow-up sequence of radii. For $r_k:=\frac{r_k}{t_{j(k)}}$, where $(s_{j(k)},t_{j(k)}]$ is the interval of flattening containing $r_k$. Let $\pi_k\subset T_0\Sigma$ be a sequence of $m$-dimensional planes with the property that 
\[
    \mathbf{E} (T_{0, r_k}, \mathbf{B}_{4}) = \mathbf{E} (T_{0, r_k}, \mathbf{B}_{4},\pi_k).
\]
Observe that for $k$ sufficiently large,
\begin{align}\label{eq:E_k}
\Ebf(T_{0, r_k}, \Cbf_{2}, \pi_k) \leq \mathbf{E} (T_{0, r_k}, \mathbf{B}_{4}) =: E_k \to 0\, , 
\end{align}
where the latter limit is due to the definition of $r_k$. Define $\Abf_k \coloneqq \Abf_{\Sigma_{0, r_k}}$. Clearly we must have $\pi_k\to \pi_0$, so by applying a suitable small rotation, we can assume that $\pi_k=\pi_0$, while at the risk of abusing notation, we relabel $T_{0,r_k}$ and $\Sigma_{0,r_k}$ as the rotated objects.

For $k_0 \in \N$ large enough, we can ensure that 
\begin{equation}\label{eq:smallexcess}
    E_k + \Abf_k^2 < \min\Big\{\eps, \frac{1}{2}\Big\} \qquad \text{for every $k \geq k_0$},
\end{equation}
where $\eps$ is the threshold in~\cite{DLHMS}*{Theorem~15.1}. We can therefore let $f_k : B_1(0,\pi_0) \to \Ascr_Q(\pi_0^\perp)$ be the strong Lipschitz approximation of~\cite{DLHMS}*{Theorem~15.1} for $T_{0, r_k}$ and define the rescaled maps
\begin{equation}\label{eq:HSfreq}
        \bar{f}_k \coloneqq \frac{f_k}{E_k^{1/2}}\, .
    \end{equation}
Assume in addition that 
\begin{align}
&\Abf_k^2 \leq C r_k^2 = o (E_k)\, .\label{e:A-E-infinitesimal}
\end{align}
It then follows from \cite{DLHMS} that, up to extracting a subsequence, 
\begin{itemize}
    \item $\bar f_k$ converges strongly in $L^2\cap W^{1,2}_\loc (B_1 (0, \pi_0))$ to a Dir-minimizing map $\bar{f}: B_1(0,\pi_0) \to \Ascr_Q(\pi_0^\perp)$,
    \item $\bar f$ takes values in  $\pi_0^\perp\cap T_0 \Sigma$, 
    \item $\bar f (0) = Q \llbracket 0\rrbracket$.
 \end{itemize}

\begin{definition}\label{d:coarse}
A Dir-minimizing map $\bar f$ as above will be called a {\em coarse} blow-up (at $0$). Its average free part is given by the map
\begin{equation}\label{e:average-free}
v (x) := \sum_i \llbracket \bar f_i (x) - \boldsymbol{\eta} \circ \bar f (x)\rrbracket\, .
\end{equation}
We say that $\bar{f}$ is nontrivial if it does not vanish identically.
\end{definition}

We have the following comparablity of coarse and fine blow-ups for subquadratic singularity degrees (see \cite{DLSk1}*{Corollary 4.3}).

\begin{corollary}\label{c:coarse=fine}
Let $T$ be as in Assumption \ref{a:main-3} and assume the singularity degree $\Irm (T,0)$ is strictly smaller than $2-\delta_2$. Then for any blow-up sequence of radii $r_k$, \eqref{e:A-E-infinitesimal} holds for $k$ suficiently large and the coarse blow-up $\bar{f}$ along (a subsequence of) $r_k$ is nontrivial, $\Irm (T, 0)$-homogeneous, and has average $0$. In fact, there exists $\lambda > 0$ such that $\bar f= \lambda u$ for a fine blow-up $u$ at $0$.
\end{corollary}

We have the following additional classification of 1-homogeneous coarse blow-ups that are translation-invariant in at most $m-2$ independent directions.

\begin{proposition}\label{p:1-homog-Dir-min}
    Suppose that $u: B_1\subset \R^m \to \Ascr_Q(\R^n)$ is a non-zero $1$-homogeneous Dir-minimizer with $\boldsymbol\eta\circ u = 0$. Suppose that its spine $\mathcal{S}_u = \{x\in B_1 : I_{u,x}(0) \geq 1 \}$ has dimension at most $m-2$. Then $u$ is a classical ($\Acal_Q$-valued) Dir-minimizer and either
    \begin{itemize}
        \item[(a)] {$\dim(\mathcal{S}_u)=m-2$, and }$u = \sum_{i=1}^Q \llbracket L_i \rrbracket$, where $L_i:B_1\to \R^n$ are linear functions which simultaneously vanish exactly along $\mathcal{S}_u$;
        \item[(b)] ${\dim(\mathcal{S}_u)} \leq m-3$.
    \end{itemize}
    In particular, if case (a) holds, the Morgan angles (see Definition \ref{d:frank}) of any pair of the linear functions $L_i,L_j$, $i\neq j$, are comparable by an absolute constant.
\end{proposition}

\begin{proof}
    {To show that $u$ is a classical Dir-minimizer, it suffices to show that $\Omega_0:=\{x\in B_1:u(x) = (Q\llbracket 0\rrbracket,1) \equiv (Q\llbracket 0\rrbracket,-1)\}$ has Hausdorff dimension at most $m-2$ (in fact, Hausdorff dimension strictly smaller than $m-1$ suffices). In turn to see this, it suffices to show that the Hausdorff dimension of $\Omega_0\setminus \mathcal{S}_u$ is at most $m-2$. So suppose $x\in \Omega_0\setminus \mathcal{S}_u$, and let $\tilde{u}$ be a tangent map to $u$ at $x$ (see e.g. \cite{DLS_MAMS}). Suppose first that 
    $\dim(\mathcal{S}_{\tilde{u}})=m-1$. Then $\tilde{u}$ has a graph which is an open book, and thus its frequency at origin is $1$. But this would imply that the frequency of $u$ at $x$ is $1$, and so (as $u$ is homogeneous of degree $1$, thus by properties of the frequency function $\mathcal{S}_u$ is exactly the set of points of frequency $1$ in $u$) we would need $x\in \mathcal{S}_u$, a contradiction to our assumption. 
    
    In particular, the only singular points of $u$ away from $\mathcal{S}_u$ are points in the $(m-2)$-stratum of $u$. From this, it follows that $\dim_\Hcal(\Omega_0)\leq m-2$ (in fact, $\dim_\Hcal(\sing(u))\leq m-2$, where $\sing(u)$ is the full singular set of $u$). In particular, this implies that we must have (without loss of generality) $u = (u^+,1)$ everywhere on $B_1$ (as $B_1\setminus\Omega_0$ is connected), for some $1$-homogeneous $\Acal_Q$-valued map $u^+$. In particular, this says that $u$ identifies with a classical $\Acal_Q$-valued Dir-minimizer. In particular, the remaining conclusions then follow immediately from properties of Dir-minimizers (see, for instance, \cite[Proposition 7.6]{DMS}, for the last claim).}
\end{proof} 

\begin{remark}
    In fact, one can instead assume in the above that $\dim_\Hcal(\Omega_0)<m-1$ in place of $\dim_\Hcal(\Scal_u)\leq m-2$, since $\Scal_u\subset \Omega_0$, and thus one can apply the proposition in this case also.
\end{remark}

\section{Decay towards $(m-2)$-invariant cones}

In order to prove \eqref{e:degree-1-dim-est} {and Theorem \ref{t:main-pt2}(b)}, we need the following notion of planar cones with exactly $(m-2)$-directions of invariance, which we recall from \cite{DMS}.

\begin{definition}\label{def:cones}
For every fixed integer $Q\geq 2$ we denote by $\mathscr{C} (Q)$ those subsets of $\mathbb R^{m+n}$ which are unions of $1 \leq N\leq Q$ $m$-dimensional planes $\alpha_1, \ldots, \alpha_N$ satisfying the following properties:
\begin{itemize}
    \item[(i)] $\alpha_i \cap \alpha_j$ is the same $(m-2)$-dimensional plane $V$ for every pair $(i,j)$ with $i<j$;
    \item[(ii)] Each plane $\alpha_i$ is contained in the same $(m+\bar n)$-dimensional plane $\varpi$.
\end{itemize}
If $p\in \Sigma$, then $\mathscr{C} (Q, p)$ will denote the subset of $\mathscr{C} (Q)$ for which $\varpi = T_p \Sigma$.

$\mathscr{P}$ and $\mathscr{P} (p)$ will denote the subset of those elements of $\mathscr{C} (Q)$ and $\mathscr{C} (Q,p)$ respectively which consist of a single plane; namely, with $N=1$. For $\mathbf{S}\in \mathscr{C} (Q)\setminus \mathscr{P}$, the $(m-2)$-dimensional plane $V$ described in (i) above is referred to as the {\em spine} of $\mathbf{S}$ and will often be denoted by $V (\mathbf{S})$.
\end{definition}

We use the same notion as in Part \ref{pt:1} for the \emph{conical $L^2$ height excess} between $T$ and elements in $\Cscr(Q)\cup \Bscr^q$. Namely, we use the following definitions:

\begin{definition}
	Given a ball $\Bbf_r(q) \subset \R^{m+n}$ and a cone $\mathbf{S}\in \mathscr{C} (Q)$, the \emph{one-sided conical $L^2$ height excess of $T$ relative to $\Sbf$ in $\Bbf_r(q)$}, denoted $\hat{\Ebf}(T, \mathbf{S}, \Bbf_r(q))$, is given by
	\[
		\hat{\Ebf}(T, \mathbf{S}, \Bbf_r(q)) \coloneqq \frac{1}{r^{m+2}} \int_{\Bbf_r (q)} \dist^2 (p, \mathbf{S})\, d\|T\|(p).
	\]
We further define the corresponding \emph{reverse one-sided excess} as
 \[
\hat{\Ebf} (\mathbf{S}, T, \Bbf_r (q)) \coloneqq \frac{1}{r^{m+2}}\int_{\Bbf_r (q)\cap \mathbf{S}\setminus \Bbf_{ar} (V (\mathbf{S}))}
\dist^2 (x, {\rm spt}\, (T))\, d\mathcal{H}^m (x)\, ,
\]
where $a=a(m)$ is a dimensional constant, to be determined later. We subsequently define the \emph{two-sided conical $L^2$ height excess} as 
\[
    \mathbb{E} (T, \mathbf{S}, \Bbf_r (q)) :=
\hat{\Ebf} (T, \mathbf{S}, \Bbf_r (q)) + \hat{\Ebf} (\mathbf{S}, T, \Bbf_r (q))\, .
\]

\end{definition}

Let us now state the main excess decay lemma, which is very much analogous to Theorem \ref{t:decay}, only near $(m-2)$-invariant cones.

\begin{lemma}[Excess Decay towards $(m-2)$-invariant cones]\label{l:m-2-decay}
Suppose that $m,n,\bar n, q=2Q$, $T$, $\Sigma$ are as in Assumption \ref{a:main-pt2}. For every $\varsigma>0$, there are positive constants $\varepsilon_0 = \varepsilon_0(q,m,n,\bar n, \varsigma) \leq \frac{1}{2}$, $r_0 = r_0(q,m,n,\bar n, \varsigma) \leq \frac{1}{2}$ and $C = C(q,m,n,\bar n)$ with the following property. Assume that 
\begin{itemize}
\item[(i)] $\|T\| (\Bbf_1) \leq (Q+\frac{1}{4}) \omega_m$;
\item[(ii)] There is $\mathbf{S}\in \mathscr{C} (Q, 0)\setminus\Pscr(0)$ such that 
\begin{equation}\label{e:smallness-2}
\mathbb{E} (T, \mathbf{S}, \Bbf_1) \leq \varepsilon_0^2 \mathbf{E}^p (T, \Bbf_1)\, 
\end{equation}
and 
\begin{equation}\label{e:no-gaps-2}
\Bbf_{ \varepsilon_0} (\xi) \cap \{p: \Theta (T,p)\geq Q\}\neq \emptyset \qquad \forall \xi \in V (\mathbf{S})\cap \Bbf_{1/2}\, ;
\end{equation}
\item[(iii)] $\mathbf{A}^2 \leq \varepsilon_0^2 \mathbb{E} (T, \mathbf{S}', \Bbf_1)$ for {any} $\mathbf{S}'\in \mathscr{C} (Q, 0)$;
\item [(iv)] {$0\in\Sing(T)$} has $\Theta(T,0)\geq Q$ and has at least one tangent cone which belongs to $\mathscr{C}(Q,0)$.
\end{itemize}
Then there is a $\mathbf{S}'\in \mathscr{C} (Q,0) \setminus \mathscr{P} (0)$ such that 
\begin{enumerate}
    \item [\textnormal{(a)}] $\mathbb{E} (T, \mathbf{S}', \Bbf_{r_0}) \leq \varsigma \mathbb{E} (T, \mathbf{S}, \Bbf_1)\,$ \label{e:decay} \\
    \item [\textnormal{(b)}] $\dfrac{\mathbb{E} (T, \mathbf{S}', \Bbf_{r_0})}{\mathbf{E}^p (T, \Bbf_{r_0})} 
\leq 2 \varsigma \dfrac{\mathbb{E} (T, \mathbf{S}, \Bbf_1)}{\mathbf{E}^p (T, \Bbf_1)}$ \\
    \item [\textnormal{(c)}] $\dist^2 (\Sbf^\prime \cap \Bbf_1,\Sbf\cap \Bbf_1) \leq C \mathbb{E} (T, \mathbf{S}, \Bbf_1)$\label{e:cone-change}
    \item[\textnormal{(d)}] $\dist^2 (V (\mathbf{S}) \cap \Bbf_1, V (\mathbf{S}')\cap \Bbf_1) \leq C \dfrac{\mathbb{E}(T,\mathbf{S},\Bbf_1)}{\Ebf^p(T,\Bbf_1)}$\, .\label{e:spine-change}
    \end{enumerate}
\end{lemma}

\begin{remark}\label{remark:planar-open-book-comparison}
Assumption (iv) is an additional assumption made here in comparison to Theorem \ref{t:decay}. Notice that for the purposes of Theorem \ref{t:main-pt2}, this is an assumption we can make (with the possibility that $N=1$ in the case where $0 \in \Ffrak_Q(T)$), and it remains true under iteration of Lemma \ref{l:m-2-decay}. The main reason for making such an assumption is that without it, we would need make the stronger assumption that $\mathbb{E}(T,\Sbf,\Bbf_1)$ is not just small in comparison to $\mathbf{E}^p(T,\Bbf_1)$, but is small in comparison to $\inf_{\Sbf^\prime}\hat{\Ebf}(T,\Sbf^\prime,\Bbf_1)$, where the infimum here is taken over all open books $\Sbf^\prime$; proving Lemma \ref{l:m-2-decay} under this latter assumption would be much more technically involved. However, by making assumption (iv), we automatically get that 
    \begin{equation}\label{E:m-2-plane-comparison-open-book}
        \Ebf^p(T,\Bbf_1)\leq \eps_1^2\mathbf{E}(T,\Sbf^\prime,\Bbf_1)
    \end{equation}
for any open book $\Sbf^\prime$, for suitable $\eps_1 = \eps_1(q,n,m,\bar{n})$, in light of the arguments in Part \ref{pt:1}. Thus, this means that \eqref{e:smallness-2} and (iv) implies the stronger condition that $\mathbb{E}(T,\Sbf,\Bbf_1)$ is small in comparison to the excess relative to any open book. Indeed, if \eqref{E:m-2-plane-comparison-open-book} were to fail, for suitably chosen $\eps_1$ Remark \ref{remark:main} would apply to give that all points of density $\geq Q$ have unique tangent cones which are open books, contradicting (iv).
\end{remark}

We will in fact reduce the majority of the arguments towards the proof of Lemma \ref{l:m-2-decay} to those in \cite{DMS}. The key will be Proposition \ref{p:integrality}, which will allow us to identify $T$ with an area-minimizing integral current outside of an arbitrarily small neighborhood of the spine of a nearby cone $\Sbf \in \Cscr(Q)\setminus \Pscr$.

First, let us introduce some notation (much analogous to that in Part \ref{pt:1}) that we will use frequently. Given a cone $\Sbf= \alpha_1\cup\cdots\cup\alpha_N\in \Cscr(Q,0)\setminus\Pscr(0)$, let
\[
    \boldsymbol\mu(\Sbf) := \max_{i,j} \dist(\alpha_i\cap\Bbf_1,\alpha_j\cap\Bbf_1),
\]
and let
\[
    \boldsymbol{\sigma} (\mathbf{S}) = \min_{i<j}
\dist (\alpha_i \cap \Bbf_1, \alpha_j \cap \Bbf_1)\, .
\]

Let us begin with the following proposition, which provides an analogue of property \eqref{e:density-drop} of Theorem \ref{t:no-gaps} in this setting.

\begin{proposition}\label{p:density-drop-m-2}
    For every $\eta,\rho >0$, there exists $\eps_0=\eps_0(q,m,n,\bar n, \eta, \rho )>0$ such that the following holds. Suppose that $T$, $\Sigma$ satisfy Assumption \ref{a:main-pt2} with $\|T\|(\Bbf_1)\leq (Q+\tfrac{1}{4})\omega_m$ and suppose that $\Sbf\in \Cscr(Q,0)\setminus \Pscr(0)$ satisfies
\begin{equation*}
    \Ebb(T,\Sbf,\Bbf_1) + \Abf^2 \leq \eps^2\boldsymbol{\mu}(\Sbf)^2\, .
\end{equation*}
Furthermore, assume that $0\in \Sing(T)$ has $\Theta(T,0)\geq Q$ and has at least one tangent cone which belongs to $\mathscr{C}(Q,0)$. Then
\begin{equation}\label{e:density-drop-m-2}
	\Theta(T,x)<\frac{q}{2} \qquad \text{for every $x\in \Bbf_{1-\eta/8} \setminus \Bbf_{\rho/4}(V)$.}
\end{equation}
\end{proposition}



\begin{proof}
    We argue by contradiction. Suppose that there exists $\eta,\rho >0$ for which there are sequences $T_k, \Sigma_k$ and $\Sbf_k$ with 
    \begin{equation}\label{e:density-drop-contradiction-1}
        \Ebb(T_k,\Sbf_k,\Bbf_1) + \Abf_k^2 \leq \eps_k^2 \boldsymbol\mu(\Sbf_k)^2,
    \end{equation}
    for $\eps_k \downarrow 0$ and $\Abf_k = \Abf(\Sigma_k)$, but
    \begin{equation}
        \Theta(T_k,x_k) \geq \frac{q}{2} \qquad \text{for some $x_k\in \Bbf_{1-\eta/8}\setminus \Bbf_{\rho/4}(V_k)$,}
    \end{equation}
    where $V_k = V(\Sbf_k)${, and furthermore $0\in \Sing(T_k)$ has $\Theta(T_k,0)\geq Q$ and at least one tangent cone to $T_k$ at $0$ belongs to $\mathscr{C}(Q,0)$}. Then, up to extracting a subsequence,
    \begin{itemize}
        \item $\Sigma_k$ converges in $C^{3,\alpha_0}$ to an $(m+\bar n)$-dimensional subspace of $\R^{m+n}$;
        \item $\Sbf_k$ converges in Hausdorff distance to a cone $\Sbf=\pi_1\cup\cdots\cup\pi_N$ in $\bar\Bbf_1$ with $\dim V(\Sbf) \geq m-2${, and moreover $V_k$ converges in Hausdorff distance to an $(m-2)$-dimensional subspace $V$ with $V\subset V(\Sbf)$};
        \item $T_k$ converges (in the mod$(q)$ flat topology) to an area-minimizing current mod$(q)$ in $\Bbf_1$, denoted by $T$, with $\spt(T) \cap \Bbf_1 = \Sbf\cap\Bbf_1$ and in addition $\|T_k \mres \Bbf_1\| \to \|T\mres\Bbf_1\|$;
        \item $x_k \to x \in \bar\Bbf_{1-\eta/8} \setminus \Bbf_{\rho/4}(V)$ with $\Theta(T,x) \geq \frac{q}{2}$.
    \end{itemize}

    {Notice first that the fourth bullet point above guarantees that $\dim V(\Sbf)\geq m-1$, as $x\in V(\Sbf)\setminus V$. However, we cannot have $\dim V(\Sbf)=m-1$, as then this would imply that $\Sbf$ is an open book, meaning that for all $k$ sufficiently large, we could apply Theorem \ref{t:main} to $T_k$, implying that $0\in\Sing(T_k)$ has a unique tangent cone which is an open book, contradictory to our assumption. Thus, we must have $\dim V(\Sbf) = m$, i.e. that $\Sbf$ is a single plane with multiplicity $Q$. But this case also leads to a contradiction, as one may then argue analogously as in the corresponding proof of Theorem \ref{t:no-gaps}, blowing-up relative to the plane (normalizing by $\boldsymbol\mu(\Sbf_k)$) to get a blow-up whose graph is in $\mathscr{C}(Q)\setminus\mathscr{P}$, yet has a density $Q$ point away from its axis, providing the desired contradiction. Indeed, to see this contradiction, notice that this would give that the graph of the blow-up is either a single plane with multiplicity $Q$, or an open book. But it cannot be a single plane by \eqref{e:density-drop-contradiction-1}, and moreover cannot be an open book by Remark \ref{remark:planar-open-book-comparison}. This gives the desired contradiction.}
\end{proof}

\section{Proof of Lemma \ref{l:m-2-decay}}
In light of Proposition \ref{p:density-drop-m-2} and Proposition \ref{p:integrality}, under the assumptions of Lemma \ref{l:m-2-decay}, $T$ identifies with an area-minimizing integral current outside of a small neighborhood of $V(\Sbf)$. Thus, all of the arguments in \cite{DMS}*{Sections 7 \& 8} remain unchanged, as do those in \cite{DMS}*{Section 10}, assuming the validity of the results in \cite{DMS}*{Section 9}. Moreover, the estimates at the spine in \cite{DMS}*{Section 11} are exactly the same (cf. Section \ref{s:spine-est}). Furthermore, as the Simon and Wickramasekera blow-ups from \cite{DMS}*{Section 13} in the present setting will also be usual Dir-minimizers (and not special Dir-minimizers) the corresponding results used from \cite{DMS}*{Sections 12 \& 13} are also the same. Consequently, in order to conclude the proof of Lemma \ref{l:m-2-decay} (and thus of Theorem \ref{t:main-pt2} and hence Theorem \ref{c:main}), it suffices to verify the results of \cite{DMS}*{Section 9}. In fact, the only argument that one needs to adapt to our current setting is that of \cite{DMS}*{Proposition 9.3}.

\subsection{Morgan angles and cone balancing}
Let us recall the notions of \emph{Morgan angles} and \emph{balanced cones} $\Sbf\in \Cscr(Q)$, first introduced in \cite{DMS}.

\begin{definition}\label{d:frank}
Given two $m$-dimensional linear subspaces $\alpha, \beta$ of $\mathbb R^N$ whose intersection has dimension $m-k$, we order the $k$ positive eigenvalues $\lambda_1 \leq \lambda_2 \leq \cdots \leq \lambda_k$ of the quadratic form $\Qcal_1:\alpha\to\R$ given by $\Qcal_1(v):=\dist^2(v,\beta)$, with the convention that the number of occurrences of the same real $\lambda$ in the list equals its multiplicity as eigenvalue of $\Qcal_1$. The \textit{Morgan angles} of the pair $\alpha$ and $\beta$ are the numbers $\theta_i (\alpha, \beta) := \arcsin \sqrt{\lambda_i}$ for $i=1,\dots,k$.
\end{definition}

The Morgan angles between two intersecting $m$-dimensional affine planes are in turn defined by simply translating an intersection point to the origin. 


\begin{definition}\label{d:balanced}
Let $\mathbf{S}\in \mathscr{C} (Q)$, $M\geq 1$, and let $\alpha_1, \ldots , \alpha_N$ be the $N$ distinct $m$-dimensional planes forming $\mathbf{S}$. We say that $\mathbf{S}$ is \emph{$M$-balanced} if for every $i\neq j$ the inequality
\begin{equation}\label{e:balanced}
\theta_2 (\alpha_i, \alpha_j) \leq M \theta_1 (\alpha_i, \alpha_j)
\end{equation}
holds for the two Morgan angles of the pair of planes $\alpha_i, \alpha_j$.
\end{definition}

We now observe that the argument in \cite{Morgan}*{Theorem 2} showing that a collection of planes that is area-minimizing mod$(q)$ with the property that pairwise the planes intersect along the same $(m-2)$-dimensional subspace have equal Morgan angles.

\begin{lemma}\label{l:Morgan-mod-q}
    Let $\Sbf\subset\R^{m+n}$ be the union of $N$ distinct $m$-dimensional planes $\alpha_1,\dotsc,\alpha_N$ with the property that, for every $i<j$, $\alpha_i\cap \alpha_j$ is the same $(m-2)$-dimensional plane $V$. If $T$ is an $m$-dimensional area minimizing current mod$(q)$ such that $\spt(T) = \Sbf$, then the (two) Morgan angles of any pair $\alpha_i,\alpha_j$, $i\neq j$, coincide.
\end{lemma}

Indeed, the proof of this fact in \cite{Morgan}*{Theorem 2} follows by constructing a competitor, which remains valid in the mod$(q)$ minimizing setting as well (as indeed one allows a larger class of competitors with the area-minimizing mod$(q)$ hypothesis when compared with just area-minimizing).

We are now in a position to justify the corresponding statement to \cite{DMS}*{Proposition 9.3} for the present setting; we will follow the notation used therein. For case (a), the argument is now identical replacing \cite{DMS}*{Lemma 7.5} with Lemma \ref{l:Morgan-mod-q}. For case (b), we instead perform a blow-up using special ($\Ascr_Q$-valued) Lipschitz approximations, resulting in a blow-up which is a special Dir-minimizer. Similarly to Proposition \ref{p:density-drop-m-2}, we know from Remark \ref{remark:planar-open-book-comparison} and the hypotheses that the blow-up is homogeneous of degree $1$ with a spine of dimension $m-2$. Hence, we may now apply Proposition \ref{p:1-homog-Dir-min} to see that in fact the blow-up is a classical $\Acal_Q$-valued Dir-minimizer induced by a superposition of planes with all Morgan angles between pairwise disjoint planes comparable, which therefore leads to the same contradiction as in \cite{DMS}*{Proposition 9.3}. After this, the proof is identical.
\qed

\bibliographystyle{alpha}
\bibliography{references}

\end{document}